\documentclass[12pt,british]{article}
\usepackage[latin9]{inputenc}
\usepackage{prettyref}
\usepackage{amsthm}
\usepackage{amsmath}
\usepackage{amssymb}
\usepackage{esint}

\makeatletter
\numberwithin{equation}{section}
\numberwithin{figure}{section}
\theoremstyle{plain}
\newtheorem{thm}{Theorem}[section]
  \theoremstyle{plain}
  \newtheorem{lem}[thm]{Lemma}
  \theoremstyle{plain}
  \newtheorem{cor}[thm]{Corollary}
  \theoremstyle{plain}
  \newtheorem{prop}[thm]{Proposition}
  \theoremstyle{definition}
  \newtheorem{rem}[thm]{Remark}

\usepackage{amsfonts}
\usepackage[numbers,sort&compress,square]{natbib}
\usepackage{url}
\numberwithin{equation}{section} 
\newrefformat{pro}{Proposition \ref{#1}}
\newrefformat{cor}{Corollary \ref{#1}}
\newrefformat{sub}{Section \ref{#1}}
\newrefformat{rem}{Remark \ref{#1}}

\makeatother

\usepackage{babel}

\begin{document}

\title{Cover Times in the Discrete Cylinder%
\thanks{This research has been supported by the grant ERC-2009-AdG 245728-RWPERCRI.%
}}

\author{David Belius}
\date{}
\maketitle
\begin{abstract}
This article proves that, in terms of local times, the properly rescaled
and recentered cover times of finite subsets of the discrete cylinder
by simple random walk converge in law to the Gumbel distribution,
as the cardinality of the set goes to infinity. As applications we
obtain several other results related to covering in the discrete cylinder.
Our method is new and involves random interlacements, which were introduced
in \cite{Sznitman2007}. To enable the proof we develop a new stronger
coupling of simple random walk in the cylinder and random interlacements,
which is also of independent interest. 
\end{abstract}

\section*{0 Introduction}

In this article we prove precise results about the asymptotic distribution
of cover times of certain finite subsets of the discrete cylinder,
with base a $d-$dimensional torus for $d\ge2$, using the theory
of \emph{random interlacements}. For families of \emph{finite} graphs
the cover time $C_{V}$ of the whole vertex set $V$ has been extensively
studied (see for instance \cite{DemboPeresEtAl-CoverTimesforBMandRWin2D,aldous-fill:book,Aldous-OnTimeTaken...,Brummelhuis1991,Aldous-ThresholdLimitsforCT,MatthewsCoveringProbsForMCs,DevroyeSbihiRWonHighlySymmGraphs}).
For many families one can show that $\mathbb{E}C_{V}$ is of order
$c|V|\log|V|$, and also that $C_{V}/(c|V|\log|V|)\rightarrow1$ in
probability as $|V|\rightarrow\infty$ (see Chapter 6 of \cite{aldous-fill:book}).
For a quite restricted class of families of {}``especially nice graphs'',
one can also prove the finer result that $C_{V}/(c|V|)-\log|V|$ tends
in law to the standard Gumbel distribution (see \cite{MatthewsCoveringProbsForMCs,DevroyeSbihiRWonHighlySymmGraphs}).
In this article we are able to prove the corresponding statement for
the cover times of subsets $F$ of the discrete cylinder (seen as
an infinite graph): we show that $L_{C_{F}}/(cN^{d})-\log|F|$ tends
in law to the Gumbel distribution as $|F|\rightarrow\infty$, provided
the sets $F$ are close to the zero level, where $L_{C_{F}}$ is the
\emph{local time} at the zero level of the cylinder when $F$ is covered.
As applications we obtain several other results related to covering.
To prove the Gumbel distributional limit result we develop an improved
coupling of simple random walk in the cylinder and random interlacements,
which is also of independent interest.

We now introduce the objects of study and our results more precisely.
We denote by $\mathbb{T}_{N}=(\mathbb{Z}/N\mathbb{Z})^{d}$ the discrete
torus and by $E_{N}=\mathbb{T}_{N}\times\mathbb{Z}$ the discrete
cylinder for $d\ge2$. Let $P$ be the canonical law of simple random
walk in $E_{N}$ starting uniformly on the zero level $\mathbb{T}_{N}\times\{0\}$,
and let $X_{n}$ denote the canonical discrete time process. For any
finite set $F\subset E_{N}$ the cover time $C_{F}$ of $F$ is the
first time $X_{n}$ has visited every vertex of $F$: \[
C_{F}=\inf\{n\ge0:F\subset X(0,n)\},\]
where $X(0,n)$ denotes the set of vertices visited up to time $n$.

We start by stating the applications of our main result. In \prettyref{cor:RealCovTimeConv}
we show that if $F_{N}\subset\mathbb{T}_{N}\times[-\frac{N}{2},\frac{N}{2}]$
is a sequence of sets such that $|F_{N}|\rightarrow\infty$, then
under $P$\begin{equation}
\frac{C_{F_{N}}}{(N^{d}\log|F_{N}|)^{2}}\overset{\mbox{law}}{\rightarrow}\zeta(\frac{g(0)}{\sqrt{d+1}}),\mbox{ as }N\rightarrow\infty,\label{eq:RealCovTimeConvInformal}\end{equation}
where $\zeta(\tau)$ denotes the first time the local time at zero
of a Brownian motion reaches $\tau$ and $g(\cdot)$ is the $\mathbb{Z}^{d+1}$
Green function (see \prettyref{eq:DefOfZdGreensFunc}). In \cite{Sznitman-HowUniveral...,Sznitman2006}
the cover time $C_{\mathbb{T}_{N}\times\{0\}}$ was studied and found
to be of order $N^{2d+o(1)}$. The result \prettyref{eq:RealCovTimeConvInformal}
sharpens this and provides the correct form of the log correction
term.

To state our second application we introduce $L_{n}$, the local time
at zero of the $\mathbb{Z}-$component of $X_{n}$ (which we often
refer to as {}``the local time of the zero level''). For any $z\in\mathbb{R}$
let $\mathcal{N}_{N}^{z}$ be the point process on $(\mathbb{R}/\mathbb{Z})^{d}\times\mathbb{R}$
defined by: \begin{equation}
\mathcal{N}_{N}^{z}=\sum_{x\in\mathbb{T}_{N}\times\{0\}}\delta_{x/N}1_{\{L_{H_{x}}>N^{d}u(z)\}},\label{eq:PointProcOfPointsCovereLastDef}\end{equation}
where $u(z)=g(0)\{\log|\mathbb{T}_{N}\times\{0\}|+z\}$. In other
words $\mathcal{N}_{N}^{z}$ counts the vertices of $\mathbb{T}_{N}\times\{0\}$
that are hit after the local time of the zero level reaches $N^{d}u(z)$
(it will later become clear that $N^{d}u(0)$ is the {}``typical''
local time at which covering of the zero level of the cylinder is
completed). We call $\mathcal{N}_{N}^{z}$ the {}``point process
of vertices covered last''. Let $\lambda$ be Lebesgue measure on
$(\mathbb{R}/\mathbb{Z})^{d}\times\{0\}$. We show in \prettyref{cor:PointProcessConv}
that \begin{equation}
\begin{array}{c}
\mathcal{N}_{N}^{z}\mbox{ converges weakly to a Poisson point process}\\
\mbox{ on }(\mathbb{R}/\mathbb{Z})^{d}\times\mathbb{R}\mbox{ of intensity }\exp(-z)\lambda.\end{array}\label{eq:PointProcConvInformal}\end{equation}
As a consequence we obtain in \prettyref{cor:LastTwoIndians} that
\begin{equation}
\begin{array}{c}
\mbox{the last two vertices of }\mathbb{T}_{N}\times\{0\}\mbox{ to be visted by }X_{n}\\
\mbox{are "far apart" at typical distance of order }N.\end{array}\label{eq:LastTwoIndiansInformal}\end{equation}
The proofs of \prettyref{eq:PointProcConvInformal} and \prettyref{eq:LastTwoIndiansInformal}
also provide similar results with other subsets of $\mathbb{T}_{N}\times[-\frac{N}{2},\frac{N}{2}]$
in place of $\mathbb{T}_{N}\times\{0\}$ (for example for $\mathbb{T}_{N}\times[-\frac{N}{2},\frac{N}{2}]$
itself). The three applications \prettyref{eq:RealCovTimeConvInformal},
\prettyref{eq:PointProcConvInformal} and \prettyref{eq:LastTwoIndiansInformal}
are consequences of the following main theorem and the coupling (see
\prettyref{eq:IntroOneBoxCoupling} below):
\begin{thm}[Convergence to Gumbel]
\textup{\label{thm:GumbelForLocTime}}\textup{\emph{Let $F_{N}\subset\mathbb{T}_{N}\times[-\frac{N}{2},\frac{N}{2}],N\ge1$,
be a sequence of sets such that $|F_{N}|\rightarrow\infty$, as $N\rightarrow\infty$.
Then under $P$\begin{equation}
\frac{L_{C_{F_{N}}}}{g(0)N^{d}}-\log|F_{N}|\overset{\mbox{law}}{\rightarrow}G,\mbox{ as }N\rightarrow\infty,\label{eq:GumbelForLocTime}\end{equation}
where $G$ denotes the standard Gumbel distribution (see \eqref{eq:GumbelCDF}).}}
\end{thm}
As mentioned in the first paragraph the class of finite graphs for
which one can obtain a Gumbel distributional limit for the cover time
is quite restricted (it includes the complete graph, the star graph
(see \cite{aldous-fill:book}) and graphs that are {}``highly symmetric''
in the sense of \cite{DevroyeSbihiRWonHighlySymmGraphs}, but for
example not the graph $\mathbb{T}_{N},d\ge3$). An additional interest
of \prettyref{thm:GumbelForLocTime} stems from the method we employ
in its proof, which relies on random interlacements. It is open whether
the method could be used to prove Gumbel distributional limits for
the cover times of other graphs; for more on this see \prettyref{rem:EndRemark}
(1).

Before describing the method in more detail let us briefly discuss
the random interlacement model. The model was introduced in \cite{Sznitman2007}
and helps to understand the {}``local picture'' left by a simple
random walk in e.g. the discrete torus $\mathbb{T}_{N},d\ge3$, (see
\cite{Windisch2008}) or the discrete cylinder $E_{N},d\ge2$, (see
\cite{Sznitman2009-RWonDTandRI}) when the walk is run up to times
of a suitable scale. The random interlacements consist of a Poisson
cloud of doubly infinite trajectories module time-shift in $\mathbb{Z}^{d},d\ge3$,
where $u$ multiplies the intensity. The trace of the trajectories
in the cloud up to a level $u$ is denoted by $\mathcal{I}^{u}\subset\mathbb{Z}^{d}$,
so that $(\mathcal{I}^{u})_{u\ge0}$ is an increasing family of random
sets. Intuitively speaking, for a value $u$ related to the time up
to which the random walk is run, the trace of the random walk in a
{}``local box'' in the torus or cylinder in some sense {}``looks
like'' $\mathcal{I}^{u}$. The previous sentence has further been
made precise in the case of the cylinder by means of a coupling in
\cite{Sznitman2009-OnDOMofRWonDCbyRI,Sznitman2009-UBonDTofDCandRI}.
The first main ingredient in the proof of \prettyref{thm:GumbelForLocTime}
is a strengthened version of this coupling. To state it we fix an
$\varepsilon\in(0,1)$ and let $A$ be a box of side length $N^{1-\varepsilon}$
with centre at $x$ for some $x\in\mathbb{T}_{N}\times[-\frac{N}{2},\frac{N}{2}]$.
We further let $R_{k}$ denote the successive returns to $\mathbb{T}_{N}\times[-N,N]$
and $D_{k}$ the successive departures from $\mathbb{T}_{N}\times(-h_{N},h_{N})$,
where $h_{N}$ has order $N(\log N)^{2}$, (see \prettyref{eq:rnbnbbtilde}).
Then the coupling result, see \prettyref{thm:CouplingManyBoxes},
implies that: \begin{eqnarray}
 & \mbox{For }N\mbox{ large enough, and for any }u\ge\frac{1}{\sqrt{N}},\delta\ge\frac{c_{2}}{(\log N)^{2}}\mbox{, we can construct a coupling}\nonumber \\
 & Q_{1}\mbox{ of }X_{\cdot}\mbox{ under }P\mbox{, and of joint random interlacements }\mathcal{I}^{u(1-\delta)},\mathcal{I}^{u(1+\delta)}\mbox{ for which}\nonumber \\
 & Q_{1}(\mathcal{I}^{u(1-\delta)}\cap A\subset X(0,D_{[uK_{N}]})\cap A\subset\mathcal{I}^{u(1+\delta)}\cap A)\ge1-cuN^{-3d-1},\label{eq:IntroOneBoxCoupling}\end{eqnarray}
where $K_{N}$ essentially equals $\frac{N^{d}}{(d+1)h_{N}}$ (see
\prettyref{eq:DefOfKn}). In fact (and importantly for our proof of
\prettyref{thm:GumbelForLocTime}), \prettyref{thm:CouplingManyBoxes}
is stronger that what is stated in \eqref{eq:IntroOneBoxCoupling}
because it couples the trace of $X_{\cdot}$ in several disjoint regions
of the cylinder with independent random interlacements, as long as
these regions are {}``far apart''.

An interest of \prettyref{eq:IntroOneBoxCoupling} is that it couples
$X(0,D_{[uK_{N}]})$ with \emph{joint} random interlacements $\mathcal{I}^{u(1-\delta)}$
and $\mathcal{I}^{u(1+\delta)}$ (combining the one-sided couplings
of \cite{Sznitman2009-UBonDTofDCandRI,Sznitman2009-OnDOMofRWonDCbyRI}
to get a two-sided coupling does not guarantee the correct joint law
of $\mathcal{I}^{u(1-\delta)}$ and $\mathcal{I}^{u(1+\delta)}$).
This makes it more useful as a {}``transfer mechanism'' from random
interlacements to random walk; see \prettyref{rem:EndRemark} (2)
for more on this topic.

Thanks to the Poissonian structure of random interlacements one has
a number of algebraic properties that only hold approximately for
the trace of random walk (cf. \prettyref{eq:lawofIuCAPK}, \prettyref{eq:markovpropforI},
\prettyref{eq:RIOneandTwoPoint}). In \cite{Belius2010} we could
take advantage of this feature and give a precise result for the asymptotic
distributions of so called \emph{cover levels }in random interlacements.
The cover level a of a finite set $F\subset\mathbb{Z}^{d+1}$ is:
\begin{equation}
\tilde{C}_{F}=\inf\{u\ge0:F\subset\mathcal{I}^{u}\}.\label{eq:coverleveldef}\end{equation}
Theorem 0.1 of \cite{Belius2010} implies that, in the notation of
\eqref{eq:GumbelForLocTime}: \begin{equation}
\frac{\tilde{C}_{F}}{g(0)}-\log|F|\overset{\mbox{law}}{\rightarrow}G,\mbox{ as }|F|\rightarrow\infty.\label{eq:coverlevelgumbnonquant}\end{equation}

The method used to prove \prettyref{thm:GumbelForLocTime} is essentially
speaking to combine \prettyref{eq:coverlevelgumbnonquant} with the
coupling \prettyref{eq:IntroOneBoxCoupling}. It will turn out that
when the local time $L_{n}$ of the zero level is $uN^{d}$ then,
roughly speaking, there have been about $[uK_{N}]$ excursions (see
\prettyref{eq:ExcTimeToLocTime}). On the other hand \prettyref{eq:IntroOneBoxCoupling}
intuitively says that after $[uK_{N}]$ excursions the picture left
in a local box (meaning a box of side length $N^{1-\varepsilon},\varepsilon>0$)
looks like random interlacements at level $u$. Thus {}``when the
local time at the zero level is $uN^{d}$ the picture in a local box
looks like $\mathcal{I}^{u}$'' (and this also holds simultaneously
for the picture left in several {}``distant'' regions contained
in local boxes). Now \eqref{eq:coverlevelgumbnonquant} essentially
speaking says that $\tilde{C}_{F}$ is close in distribution to $g(0)\{\log|F|+G\}$,
and thus we roughly find that if $F$ is contained in one or several
{}``distant'' local boxes then $L_{C_{F}}$, the local time at the
zero level when $F$ is covered, is close in distribution to $N^{d}g(0)\{\log|F|+G\}$.
But this is the intuitive meaning of \eqref{eq:GumbelForLocTime}.
When $F$ is contained in one or several {}``distant'' local boxes
this intuitive explanation can be turned into a rigorous proof.

However sets like $F_{N}=\mathbb{T}_{N}\times\{0\}$ can not be split
into pieces that are contained in distant local boxes. To deal with
this problem we consider two cases. The first, considered in \prettyref{pro:PropSmallSets},
is when the $F_{N}$ are small in the sense that $|F_{N}|\le N^{1/8}$.
It turns out that we can split such small sets into pieces $S_{1},S_{2},...,S_{k}$
such that the pieces are contained in {}``distant'' local boxes,
so that we are in the situation discussed in the previous paragraph
and can prove that the limit distribution is the Gumbel distribution.

The second case, considered in \prettyref{pro:PropLargeSets}, is
when the sets are {}``large'' in the sense that $|F_{N}|>N^{1/8}$.
It turns out that such a set is typically covered completely when
the \emph{local }time (at the zero level) reaches roughly $N^{d}g(0)\log|F_{N}|$.
We consider the set $F_{N}^{\rho}$ of vertices not covered when the
local time at the zero level reaches a fraction $(1-\rho)$ of the
typical local time $N^{d}g(0)\log|F_{N}|$ (in fact $F_{N}^{\rho}$
will be defined in terms of excursions). By tiling the cylinder with
local boxes, using the coupling \prettyref{eq:IntroOneBoxCoupling}
once for each box, and using a calculation inside the random interlacements
model we are able to show (for appropriate values of $\rho$) that
$F_{N}^{\rho}$ is with high probability {}``small'' in the sense
that $|F_{N}^{\rho}|\le N^{1/8}$ and that $|F_{N}^{\rho}|$ concentrates
around its typical value, which turns out to be $|F_{N}|^{\rho}$.
By excluding a short segment of the random walk (when it is far away
from $F_{N}$ and thus does not affect $F_{N}^{\rho}$) during which
it {}``forgets'' the shape of $F_{N}^{\rho}$ we will show that
the way in which $X_{\cdot}$ covers $F_{N}^{\rho}$ is essentially
the same as the way an independent random walk would cover $F_{N}^{\rho}$.
Thus $L_{C_{F_{N}}}$ should be close in distribution to $(1-\rho)N^{d}g(0)\log|F_{N}|+L_{C_{F'}}$,
where $F'$ is independent from $X_{\cdot}$ and distributed as $F_{N}^{\rho}$.
Since $F_{N}^{\rho}$ is {}``small'' with high probability we can
apply the previous case for typical realisations of $F'$ to get that
$L_{C_{F'}}$ is close in distribution to $N^{d}g(0)\{\log|F^{'}|+G\}$.
Since $|F_{N}^{\rho}|$ concentrates around $|F_{N}|^{\rho}$ we find
that $\log|F'|\approx\rho\log|F_{N}|$ so adding $L_{C_{F'}}$ to
the deterministic part $(1-\rho)N^{d}g(0)\log|F_{N}|$ we get that
$L_{C_{F_{N}}}$ has law close to $N^{d}g(0)\{\log|F_{N}|+G\}$ (which
is the intuitive interpretation of \prettyref{eq:GumbelForLocTime}).

We now describe how this article is organized. In \prettyref{sec:Notation}
we fix notation, recall some standard results on random walks and
random interlacements and prove some preliminary lemmas. In \prettyref{sec:Applications}
we use our main result \prettyref{thm:GumbelForLocTime} to prove
\prettyref{eq:RealCovTimeConvInformal}, \prettyref{eq:PointProcConvInformal}
and \prettyref{eq:LastTwoIndiansInformal}. In \prettyref{sec:ConvToGumbel}
we then prove \prettyref{thm:GumbelForLocTime}, using the full version
of the coupling \prettyref{eq:IntroOneBoxCoupling} (i.e. \prettyref{thm:CouplingManyBoxes}),
and a quantitative version of \prettyref{eq:coverlevelgumbnonquant}
(see \prettyref{eq:QuantCovLevResult-1}). The proof of \prettyref{thm:CouplingManyBoxes}
is contained in sections 4, 5 and 6.

Finally a note on constants. Named constants are denoted by $c_{0},c_{1},..$
and have fixed values. Unnamed constants are denoted by $c$ and may
change from line to line and within formulas. All constants are strictly
positive and unless otherwise indicated they only depend on $d$.
Further dependence on e.g. parameters $\alpha,\beta$ is denoted by
$c(\alpha,\beta)$.

\section{\label{sec:Notation}Notation and and some useful results}

In this section we fix notation and recall some known results about
random walk and random interlacements. We also state and prove \prettyref{lem:GreensFuncDecayInCylinder}
which gives an upper bound on a certain killed Green function in the
cylinder, \prettyref{lem:ExcTimeToLocTime} which relates local time
of the random walk to excursion times and to Brownian local time,
and \prettyref{lem:Calc} which gives a bound on certain sums of the
{}``two point function'' in the random interlacements model.

In this article $\mathbb{N}=\{0,1,2,...\}$. For any real $x\ge0$
we denote the integer part of $x$ by $[x]$. If $U$ is a set $|U|$
denotes the cardinality of $U$.

We denote by $|\cdot|_{\infty}$ and $|\cdot|$ the $l_{\infty}$
and Euclidean norms on $\mathbb{R}^{d+1}$ and by $d_{\infty}(\cdot,\cdot)$
and $d(\cdot,\cdot)$ the corresponding induced distances on $(\mathbb{R}/\mathbb{Z})^{d}\times\mathbb{R}$,
$\mathbb{Z}^{d+1}$, and $E_{N}$. For any two sets $A,B\subset\mathbb{Z}^{d+1}$
or $A,B\subset E_{N}$ we denote their mutual Euclidean distance $\inf_{x\in A,y\in B}d(x,y)$
by $d(A,B)$. The closed $l_{\infty}$-ball centred at $x$ in $\mathbb{Z}^{d+1}$
or $E_{N}$ of radius $R$ is denoted by $B(x,R)$. For any set $U\subset\mathbb{Z}^{d+1}$
or $E_{N}$ we define the inner and outer boundaries by\[
\partial_{i}U=\{x\in U:d(\{x\},U^{c})=1\}\mbox{ and }\partial_{e}U=\{x\in U^{c}:d(\{x\},U)=1\}.\]
A trajectory (or path) is a sequence $w(n),n\in\mathbb{N}$, in $\mathbb{Z}^{d+1}$
or $E_{N}$ such that $d(w(n+1),w(n))\le1$ for all $n\ge0$. We define
the trace of the trajectory as follows:\begin{equation}
w(a,b)=\{x:w(n)=x\mbox{ for some }n\in[a,b]\},a\le b\mbox{ in }\mathbb{N}.\label{eq:TraceDefinition}\end{equation}
We write $\mathcal{T}$ for the space of trajectories in $E_{N}$
and $W$ for the space of trajectories in $\mathbb{Z}^{d+1}$. For
any set $F\subset E_{N}$ or $F\subset\mathbb{Z}^{d+1}$ we write
$\mathcal{T}_{F}$ for the countable subset of $\mathcal{T}$ consisting
of trajectories that are contained in $F\cup\partial F$ and stay
constant after a finite time. The canonical coordinates on $\mathcal{T}$
and $W$ are denoted by $(X_{n})_{n\ge0}$ and the canonical shift
by $(\theta_{n})_{n\ge0}$. For a subset $U$ of $E_{N}$ or $\mathbb{Z}^{d+1}$
we define the entrance time $H_{U}$, the hitting time $\tilde{H}_{U}$,
and the exit time $T_{U}$ by:\begin{eqnarray*}
H_{U} & = & \inf\{n\ge0:X_{n}\in U\},\tilde{H}_{U}=\inf\{n\ge1:X_{n}\in U\},\\
T_{U} & = & \inf\{n\ge0:X_{n}\notin U\}.\end{eqnarray*}
When $U$ is the singleton $\{x\}$ we write $H_{x}$ or $\tilde{H}_{x}$
for simplicity. We define the special levels $r_{N},h_{N}$ and the
special slabs $B,\tilde{B}$ of $E_{N}$ by \begin{equation}
r_{N}=N,h_{N}=[N(2+(\log N)^{2})]\mbox{ and }B=\mathbb{T}_{N}\times[-r_{N},r_{N}],\tilde{B}=\mathbb{T}_{N}\times(-h_{N},h_{N}).\label{eq:rnbnbbtilde}\end{equation}
The successive returns to $B$ and departures from $\tilde{B}$ are
given by \begin{equation}
\begin{array}{cl}
 & R_{1}=H_{B},D_{1}=T_{\tilde{B}}\circ\theta_{R_{1}}+R_{1},\mbox{ and for }k\ge1,\\
 & R_{k+1}=R_{1}\circ\theta_{D_{k}}+D_{k},\mbox{ and }D_{k}=D_{1}\circ\theta_{D_{k}}+D_{k}.\end{array}\label{eq:ExcursionDef}\end{equation}
For $x\in\mathbb{Z}^{d+1}$ we denote by $P_{x}^{\mathbb{Z}^{d+1}}$
the law on $W$ of simple random starting at $x$. For $x\in E_{N}$
we denote by $P_{x}$ the law on $\mathcal{T}$ of simple random starting
at $x$. If $e$ is a measure on $\mathbb{Z}^{d+1}$ or $E_{N}$ we
denote by $P_{e}^{\mathbb{Z}^{d+1}},P_{e}$ the measures $\sum_{a}e(a)P_{e}^{\mathbb{Z}^{d+1}}$
and $\sum_{a}e(a)P_{e}$ respectively. A special role will be played
by the measures

\begin{equation}
q=\frac{1}{2N^{d}}\sum_{x\in\mathbb{T}_{N}\times\{-r_{N},r_{N}\}}\delta_{x}\mbox{ and }q_{z}=\frac{1}{N^{d}}\sum_{x\in\mathbb{T}_{N}\times\{z\}}\delta_{x},z\in\mathbb{Z}.\label{eq:qnotation}\end{equation}
Note that the measure $P$ that appears in the introduction coincides
with $P_{q_{0}}$. For any finite $K\subset\mathbb{Z}^{d+1}$ we define
the escape probability (or equilibrium measure) $e_{K}$ and capacity
$\textnormal{cap}(K)$ by\begin{eqnarray}
 &  & e_{K}(x)=P_{x}^{\mathbb{Z}^{d+1}}(\tilde{H}_{K}=\infty)1_{K}(x)\mbox{ and }\textnormal{cap}(K)=\sum_{x\in K}e_{K}(x).\label{eq:DefOfCapInZd}\end{eqnarray}
If $K\subset U\subset E_{N}$ with $U$ finite, then we define the
escape probability and capacity of $K$ relative to $U$ by\begin{eqnarray}
 &  & e_{K,U}(x)=P_{x}(\tilde{H}_{K}>T_{U})1_{K}(x)\mbox{ and }\textnormal{cap}_{U}(K)=\sum_{x\in K}e_{K,U}(x).\label{eq:DefOfCapInCylinder}\end{eqnarray}
We define the $\mathbb{Z}^{d+1}$ Green function by\begin{equation}
g(x,y)=\sum_{n\ge0}P_{x}^{\mathbb{Z}^{d+1}}(X_{n}=y)\mbox{ and }g(\cdot)=g(\cdot,0)\mbox{ for }x,y\in\mathbb{Z}^{d+1}.\label{eq:DefOfZdGreensFunc}\end{equation}
The Green function killed on exiting $U$ for $U\subset\mathbb{Z}^{d+1}$
is defined by\[
g_{U}(x,y)=\sum_{n\ge0}P_{x}^{\mathbb{Z}^{d+1}}(X_{n}=y,n<T_{U}),\]
and similarly for $U\subset E_{N}$ with $P_{x}$ in place of $P_{x}^{\mathbb{Z}^{d+1}}$.
Classically, if $K\subset U\subset E_{N}$ with $U$ finite, then
for all $x\in U$\begin{equation}
P_{x}(H_{K}<T_{U})=\sum_{y\in K}g_{U}(x,y)e_{K,U}(y).\label{eq:HittingToGreenEquil}\end{equation}
For two disjoint sets $S_{1},S_{2}\subset\tilde{B}$ we define their
{}``mutual energy'' relative to $\tilde{B}$: \begin{equation}
\mathcal{E}(S_{1},S_{2})=\sum_{x\in S_{1},y\in S_{2}}e_{S_{1},\tilde{B}}(x)g_{\tilde{B}}(x,y)e_{S_{2},\tilde{B}}(y).\label{eq:DefOfMutualEnergy}\end{equation}
The following classical bounds on the Green function $g(x)$ follow
Theorem 1.5.4 p. 31 of \cite{LawlersLillaGrona}:
\begin{lem}
($d\ge2$) For all non-zero $x\in\mathbb{Z}^{d+1}$\begin{equation}
c|x|^{1-d}\le g(x)\le c|x|^{1-d}.\label{eq:ZdGreensFuncBound}\end{equation}

\end{lem}

We also have similar bounds on $g_{\tilde{B}}(x,y)$:
\begin{lem}
\label{lem:GreensFuncDecayInCylinder}$(d\ge2,N\ge1)$ For any $x,y\in B$
with $x\ne y$\begin{equation}
c|x-y|^{1-d}\le g_{\tilde{B}}(x,y)\le c|x-y|^{1-d}+c\frac{h_{N}}{N^{d}}.\label{eq:GreensFunctionBoundInCylinder}\end{equation}
\end{lem}
\begin{proof}
Let $e$ denote the vector $(0,...,0,1)\in\mathbb{Z}^{d+1}$. By {}``unwrapping''
the cylinder $E_{N}$ we see that for any $x,y\in\tilde{B}$ \begin{equation}
g_{\tilde{B}}(x,y)=\sum_{n\in\mathbb{Z}^{d+1},n\cdot e=0}g_{U}(x',y'+nN),\label{eq:GreensFunctionUnwrapping}\end{equation}
where $U=\{x\in\mathbb{Z}^{d+1}:|x\cdot e|<h_{N}\}$ and $x',y'\in\mathbb{Z}^{d}\times[-r_{N},r_{N}]$
are representatives in $\mathbb{Z}^{d+1}$ of $x,y$ such that $|x-y|=|x'-y'|$.
Now the lower bound of \prettyref{eq:GreensFunctionBoundInCylinder}
is a consequence of $g_{\tilde{B}}(x,y)\ge g_{U}(x',y')\ge g_{B(y',\frac{1}{3}h_{N})}(x',y')\ge c|x'-y'|^{1-d}$
where the last inequality follows from Proposition 1.5.9 p. 35 of
\cite{LawlersLillaGrona}. Furthermore it follows from (2.13) of \cite{SznitmanNewExofRWRE}
with $L=h_{N}$ that if $n\cdot e=0$\begin{eqnarray*}
 &  & g_{U}(x',y'+nN)\le c|x-y|^{1-d}1_{\{|n|<3\}}+\frac{1}{(|n|N)^{d-1}}\exp\left(-c\frac{N|n|}{h_{N}}\right)1_{\{|n|\ge3\}}.\end{eqnarray*}
But $\sum_{n\in\mathbb{Z}^{d}}\frac{1}{|n|^{d-1}}\exp(-c\frac{N|n|}{h_{N}})\le c\frac{h_{N}}{N}$,
so summing over $n$ in \prettyref{eq:GreensFunctionUnwrapping} one
obtains the upper bound of \prettyref{eq:GreensFunctionBoundInCylinder}.
\end{proof}
Note that thanks to \prettyref{eq:GreensFunctionBoundInCylinder}
we have the following bound on $\mathcal{E}(S_{1},S_{2})$ when $S_{1},S_{2}\subset B$:\begin{eqnarray}
\begin{array}{ccl}
\mathcal{E}(S_{1},S_{2}) & \overset{\prettyref{eq:DefOfCapInCylinder},\prettyref{eq:GreensFunctionBoundInCylinder}}{\le} & c\textnormal{cap}_{\tilde{B}}(S_{1})\textnormal{cap}_{\tilde{B}}(S_{2})\left\{ (d(S_{1},S_{2}))^{1-d}+\frac{h_{N}}{N^{d}}\right\} \\
 & \overset{\prettyref{eq:DefOfCapInCylinder}}{\le} & c|S_{1}||S_{2}|\left\{ (d(S_{1},S_{2}))^{1-d}+\frac{h_{N}}{N^{d}}\right\} .\end{array}\label{eq:BoundOnMutualEnergy}\end{eqnarray}
The equalities contained in the following lemma will be essential:
\begin{lem}
($N\ge3$) For all $K\subset\mathbb{T}_{N}\times(-r_{N},r_{N})$\begin{equation}
P_{q}(H_{K}<T_{\tilde{B}},X_{H_{K}}=x)=\frac{1}{K_{N}}e_{K,\tilde{B}}(x),x\in K\mbox{ and }\label{eq:GeometricLemma}\end{equation}
\begin{equation}
P_{q}(H_{K}<T_{\tilde{B}},(X_{H_{K}+\cdot})\in dw)=\frac{1}{K_{N}}P_{e_{K,\tilde{B}}}(dw),\label{eq:GeometricLemmaStrongMarkov}\end{equation}
where \begin{equation}
K_{N}=\frac{N^{d}}{(d+1)(h_{N}-r_{N})}.\label{eq:DefOfKn}\end{equation}
\end{lem}
\begin{proof}
\eqref{eq:GeometricLemma} follows from Lemma 1.1 of \cite{Sznitman2009-UBonDTofDCandRI}
and \eqref{eq:GeometricLemmaStrongMarkov} follows from \eqref{eq:GeometricLemma}
by an application of the strong Markov property.
\end{proof}
Incidentally \prettyref{eq:GeometricLemma} can be used to see that
$\mbox{cap}_{\tilde{B}}(\{x\})\le\mbox{cap}_{\tilde{B}}(K)$ when
$x\in K\subset\mathbb{T}_{N}\times(-r_{N},r_{N})$ and therefore together
with the bound $\mbox{cap}_{\tilde{B}}(\{x\})\ge P_{x}^{\mathbb{Z}^{d+1}}(T_{B(y,\frac{1}{4}N)}>\tilde{H}_{x})-\sup_{y\in\partial B(y,\frac{1}{4}N)}g_{\tilde{B}}(x,y)\overset{\prettyref{eq:GreensFunctionBoundInCylinder}}{\ge}c$
valid for all $x\in\mathbb{T}_{N}\times[r_{N},r_{N}]$ and $N\ge c$
we see that \begin{equation}
\mathcal{E}(S_{1},S_{2})\overset{\prettyref{eq:DefOfMutualEnergy},\prettyref{eq:GreensFunctionBoundInCylinder}}{\ge}cN^{1-d}\mbox{ for all }N\ge1\mbox{ and non-empty }S_{1},S_{2}\subset\mathbb{T}_{N}\times(-r_{N},r_{N}).\label{eq:MutualEnergyLowerBound}\end{equation}

The local time of $X_{n}$ at the zero level (or equivalently the
local time at $0$ of the $\mathbb{Z}-$component of $X_{n}$) is
denoted by\begin{equation}
L_{n}=|\{i\in[0,n]:X_{i}\in\mathbb{T}_{N}\times\{0\}\}|,n\in\mathbb{N},\label{eq:DefOfLocTime}\end{equation}
and the first time the local time at the zero level is at least $u$
by\begin{equation}
\gamma_{u}=\inf\{n\ge0:L_{n}\ge u\},u\ge0.\label{eq:DefOfGamma}\end{equation}
Similarly we define\begin{equation}
\zeta(u)=\inf\{t\ge0:\hat{L}_{t}\ge u\},\label{eq:DefOfZeta}\end{equation}
where the continuous process $\hat{L}_{t}$ is the local time at zero
of a canonical Brownian motion. Note that by the scaling invariance
of Brownian motion \begin{equation}
\zeta(u)\mbox{ satisfies the scaling relation }\zeta(u)\overset{\textnormal{law}}{=}u^{2}\zeta(1),u>0.\label{eq:ZetaScalingRelation}\end{equation}
The cumulative distribution function of $\zeta(u)$ is known explicitly
(see e.g. Theorem 2.3 p. 240 of \cite{RevuzYor-ContMartAndBM}) it
is the continuous function\begin{equation}
F(z)=1_{\{z>0\}}\sqrt{\frac{2}{\pi}}\int_{-\infty}^{-\frac{u}{\sqrt{z}}}e^{-x^{2}/2}dx.\label{eq:ZetaCDF}\end{equation}
The following lemma relates $\gamma_{u}$ to the excursion times $D_{k}$
and to the law of $\zeta(\cdot)$:
\begin{lem}
\label{lem:ExcTimeToLocTime}For any $N\ge c$, $\frac{1}{2}>\delta\ge c_{0}\frac{r_{N}}{h_{N}}$
and $u$ such that $uK_{N}\ge2$ we have \begin{equation}
\begin{array}{ccc}
P(D_{[(1-\delta)uK_{N}]}<\gamma_{uN^{d}}\le D_{[(1+\delta)uK_{N}]})\ge1-c\exp(-cu\sqrt{N}).\end{array}\label{eq:ExcTimeToLocTime}\end{equation}
Also under $P$, for any fixed $N\ge3$,\begin{equation}
\frac{\gamma_{u}}{u^{2}}\overset{\mbox{law}}{\rightarrow}\zeta(\frac{1}{\sqrt{d+1}}),\mbox{ as }u\rightarrow\infty.\label{eq:RWLocTimeToBMLocTime}\end{equation}
\end{lem}
\begin{proof}
We start with \prettyref{eq:ExcTimeToLocTime}. Note that\begin{equation}
\{\gamma_{a}\le b\}\overset{\prettyref{eq:DefOfLocTime},\prettyref{eq:DefOfGamma}}{=}\{a\le L_{[b]}\}\mbox{ for real }a,b\ge0.\label{eq:GammaLocRelation}\end{equation}
Thus it suffices to show (using also that $\delta^{2}K_{N}\ge c\sqrt{N}$)
\begin{equation}
P(L_{D_{[(1-\delta)uK_{N}]}}<uN^{d}\le L_{D_{[(1+\delta)uK_{N}]}})\ge1-c\exp(-c\delta^{2}uK_{N}).\label{eq:ExcLocSuffToShow}\end{equation}
Define the successive returns $\tilde{R}_{k},k\ge1$, to $\mathbb{T}_{N}\times\{0\}$
and departures $\tilde{D}_{k},k\ge1$, from $\mathbb{T}_{N}\times\{0\}$
analogously to \prettyref{eq:ExcursionDef} with $\mathbb{T}_{N}\times\{0\}$
replacing both $B$ and $\tilde{B}$. Let $V_{n}=|\{k:\tilde{D}_{k}\in[R_{n},D_{n}]\}|$
denote the number of contiguous intervals of time spent in the zero
level during the $n-$th excursion between $B$ and $\partial_{e}\tilde{B}$.
Then \begin{equation}
L_{D_{n}}=\sum_{k=1}^{V_{1}+...+V_{n}}(\tilde{D}_{k}-\tilde{R}_{k})\mbox{ for all }n\ge1.\label{eq:LDnDecomposition}\end{equation}
By the strong Markov property $V_{n},n\ge1$, are independent and
$V_{1}$ is geometric with support $\{1,2,...\}$ and parameter $\frac{1}{h_{N}}$
(the probability $P_{x}(\tilde{R}_{2}>D_{1})$ when $x\in\mathbb{T}_{N}\times\{0\}$)
and $V_{n}\overset{\mbox{law}}{=}UV_{1}$ for $n\ge2$, where $U$
is a Bernoulli random variable, independent from $V_{1}$, with $\mathbb{P}(U=0)=\frac{r_{N}}{h_{N}}$
(the probability that $X_{n}$ leaves $\tilde{B}$ before hitting
$\mathbb{T}_{N}\times\{0\}$, when starting in $\partial_{i}B$).
By a standard large deviation bound using the exponential Chebyshev
inequality and the small exponential moments of $\frac{V_{i}}{h_{N}}$
we see that if $\delta\ge4\frac{r_{N}}{h_{N}}$ then\begin{eqnarray*}
P(V_{1}+...+V_{[(1+\delta)uK_{N}]}\le(1+\frac{\delta}{2})uK_{N}h_{N}) & \le & \exp(-c\delta^{2}uK_{N})\mbox{ and }\\
P(V_{1}+...+V_{[(1-\delta)uK_{N}]}\ge(1-\frac{\delta}{2})uK_{N}h_{N}) & \le & \exp(-c\delta^{2}uK_{N}).\end{eqnarray*}
Combining this with \prettyref{eq:LDnDecomposition} we see that \prettyref{eq:ExcLocSuffToShow}
(and therefore also \prettyref{eq:ExcTimeToLocTime}) follows once
we show that\begin{equation}
P(\sum_{k=1}^{[(1-\frac{\delta}{2})uK_{N}h_{N}]}(\tilde{D}_{k}-\tilde{R}_{k})<uN^{d}\le\sum_{k=1}^{[(1+\frac{\delta}{2})uK_{N}h_{N}]}(\tilde{D}_{k}-\tilde{R}_{k}))\ge1-c\exp(-c\delta^{2}uK_{N}).\label{eq:SuffToShow2}\end{equation}
Now by the strong Markov property $\tilde{D}_{k}-\tilde{R}_{k},k\ge1,$
are independent geometric random variables with support $\{1,2,...\}$
and parameter $\frac{1}{d+1}$ (the probability that $X_{n+1}\notin\mathbb{T}_{N}\times\{0\}$
conditioned on $X_{n}\in\mathbb{T}_{N}\times\{0\}$). Therefore by
a standard large deviation bound we see that\begin{eqnarray*}
P(\sum_{k=1}^{[(1-\frac{\delta}{2})uK_{N}h_{N}]}(\tilde{D}_{k}-\tilde{R}_{k})\ge(1-\frac{\delta}{4})(d+1)uK_{N}h_{N}) & \le & c\exp(-c\delta^{2}uK_{N}h_{N})\mbox{ and }\\
P(\sum_{k=1}^{[(1+\frac{\delta}{2})uK_{N}h_{N}]}(\tilde{D}_{k}-\tilde{R}_{k})\le(d+1)uK_{N}h_{N}) & \le & c\exp(-c\delta^{2}uK_{N}h_{N}).\end{eqnarray*}
From this \prettyref{eq:SuffToShow2} follows by observing that $(1-\frac{\delta}{4})(d+1)uK_{N}h_{N}\overset{\prettyref{eq:DefOfKn}}{=}uN^{d}\frac{1-\delta/4}{1-r_{N}/h_{N}}\le uN^{d}\overset{\prettyref{eq:DefOfKn}}{\le}(d+1)uK_{N}h_{N}$
if $\delta\ge4\frac{r_{N}}{h_{N}}$. This completes the proof of \prettyref{eq:ExcTimeToLocTime}.

We now turn to \prettyref{eq:RWLocTimeToBMLocTime}. For fixed $N\ge3$
let $Z_{n}$ denote the $\mathbb{Z}-$component of $X_{n}$. Then
$L_{n}$ is the local time of $Z_{n}$ at zero. By (1.22) of \cite{StrongInvForLocTime}
we can couple $L_{\cdot}$ with $\hat{L}_{\cdot}$, the local time
at $0$ of a Brownian motion, so that\begin{equation}
\sup_{n\ge1}n^{-3/8}|(d+1)\hat{L}_{\frac{n}{d+1}}-L_{n}|<\infty\mbox{ a.s.}\label{eq:RWLocTimeBMLocTimeCouple}\end{equation}
For any $u\ge0,z\ge0$ we have $P(\gamma_{u}\le zu^{2})=P(u\le L_{[zu]^{2}})$
by $\prettyref{eq:GammaLocRelation}$ and thus it follows from \prettyref{eq:RWLocTimeBMLocTimeCouple}
that for any $\alpha\in(0,1)$\begin{equation}
\begin{array}{cll}
\underline{\lim}_{u\rightarrow\infty}P(u(1+\alpha)\le(d+1)\hat{L}_{\frac{[zu^{2}]}{d+1}}) & \le & \underline{\overline{\lim}}_{u\rightarrow\infty}P(\gamma_{u}\le zu^{2})\\
 & \le & \overline{\lim}_{u\rightarrow\infty}P(u(1-\alpha)\le(d+1)\hat{L}_{\frac{[zu^{2}]}{d+1}}).\end{array}\label{eq:Sandwhich}\end{equation}
But by \prettyref{eq:DefOfZeta} and \prettyref{eq:ZetaScalingRelation}
we have $P(u(1\pm\alpha)\le(d+1)\hat{L}_{\frac{[zu^{2}]}{d+1}})=P(\zeta(\frac{1}{\sqrt{d+1}})\le\frac{[zu^{2}]}{u^{2}(1\pm\alpha)^{2}})$
and therefore from \prettyref{eq:Sandwhich}\[
\mathbb{P}(\zeta(\frac{1}{\sqrt{d+1}})\le\frac{z}{(1+\alpha)^{2}})\overset{\prettyref{eq:ZetaCDF}}{\le}\underline{\overline{\lim}}_{u\rightarrow\infty}P(\gamma_{u}\le zu^{2})\overset{\prettyref{eq:ZetaCDF}}{\le}\mathbb{P}(\zeta(\frac{1}{\sqrt{d+1}})\le\frac{z}{(1-\alpha)^{2}}).\]
Thus taking $\alpha\rightarrow0$ we get \prettyref{eq:RWLocTimeToBMLocTime}.
\end{proof}
In the proof of the coupling result \prettyref{eq:IntroOneBoxCoupling}
(i.e. \prettyref{thm:CouplingManyBoxes}) the first step is to couple
random walk with so called {}``Poisson Processes of Excursions''.
To introduce them we first define for any law $e$ on $E_{N}$ the
probability\begin{equation}
\kappa_{e}(dw)=P_{e}(X_{\cdot\wedge T_{\tilde{B}}}\in dw).\label{eq:KappaDefinition}\end{equation}
A special role will be played by $\kappa_{q}$ where $q$ as in \prettyref{eq:qnotation}.
A {}``Poisson process of excursions'' is a Poisson process on the
space $\mathcal{T}_{\tilde{B}}$ (see below \prettyref{eq:TraceDefinition})
of intensity which is a multiple of\begin{equation}
\nu(dw)=K_{N}\kappa_{q}(dw)=K_{N}P_{q}(X_{\cdot\wedge T_{\tilde{B}}}\in dw).\label{eq:PPoEIntensity}\end{equation}
If $\mu=\sum_{n\ge0}\delta_{w_{n}}$ is a point process on one of
the spaces $\mathcal{T}_{\tilde{B}}$ or $W$ we define the trace
$\mathcal{I}(\mu)$ of $\mu$ by (see \eqref{eq:TraceDefinition}
for notation)\begin{equation}
\mathcal{I}(\mu)=\bigcup_{n\ge0}w_{n}(0,\infty)\subset E_{N}\mbox{ or }\mathbb{Z}^{d+1}.\label{eq:ITraceDef}\end{equation}

We now recall some facts about random interlacements. They are defined
as a Poisson point process on a certain space of trajectories modulo
time-shift, on a probability space we denote by $(\Omega_{0},\mathcal{A}_{0},Q_{0})$.
For a detailed construction we refer to Section 1 of \cite{Sznitman2007}
or Section 1 of \cite{SidoraviciusSznitman2009}. In this article
we will only need the facts that now follow. On $(\Omega_{0},\mathcal{A}_{0},Q_{0})$
there is a family $(\mathcal{I}^{u})_{u\ge0}$ of random subsets of
$\mathbb{Z}^{d+1}$, indexed by a parameter $u$. We call $\mathcal{I}^{u}$,
or any random set with the law of $\mathcal{I}^{u}$, \emph{a random
interlacement at level $u$}. Intuitively speaking $\mathcal{I}^{u}$
is the trace of the Poisson cloud of trajectories mentioned in the
introduction. Two basic properties of random interlacements are that
\begin{equation}
\begin{array}{c}
\mbox{the }\mathcal{I}^{u}\mbox{ are translation invariant and increasing,}\\
\mbox{ in the sense that if }v\le u\mbox{ then }\mathcal{I}^{v}\subset\mathcal{I}^{u}.\end{array}\label{eq:InterlacementIsIncreasing}\end{equation}
We can characterise the law of $\mathcal{I}^{u}\cap K$ for finite
sets $K\subset\mathbb{Z}^{d+1}$ in the following manner (see (1.18),
(1.20), (1.53) of \cite{Sznitman2007}) \begin{eqnarray}
\mathcal{I}^{u}\cap K\overset{\textnormal{law}}{=}\mathcal{I}(\mu_{K,u})\cap K\mbox{ for each }u\ge0\mbox{ where }\mu_{K,u}\mbox{ is a}\label{eq:lawofIuCAPK}\\
\mbox{Poisson point process on }W\mbox{ of intensity }uP_{e_{K}}^{\mathbb{Z}^{d+1}}.\nonumber \end{eqnarray}
An important fact is that if $u\le v$ and $\mathcal{I}_{1}$ and
$\mathcal{I}_{2}$ are independent random interlacements then\begin{equation}
(\mathcal{I}_{1}^{u},\mathcal{I}_{1}^{v})\overset{\mbox{\textnormal{law}}}{=}(\mathcal{I}_{1}^{u},\mathcal{I}_{1}^{u}\cup\mathcal{I}_{2}^{v-u}).\label{eq:markovpropforI}\end{equation}
The law of $(\mathcal{I}^{u})^{c}$ (also called the vacant set) on
$\{0,1\}^{\mathbb{Z}^{d+1}}$ is characterized by (see (2.16) of \cite{Sznitman2007}):\[
Q_{0}(A\subset(\mathcal{I}^{u})^{c})=\exp(-u\cdot\textnormal{cap}(A))\mbox{ for all finite }A\subset\mathbb{Z}^{d+1}.\]
Since $\textnormal{cap}(\{x\})=\frac{1}{g(0)}$ and $\textnormal{cap}(\{x,y\})=\frac{2}{g(x)+g(x-y)}$
(see (1.62) and (1.64) of \cite{Sznitman2007}) we have:\begin{eqnarray}
Q_{0}(x\notin\mathcal{I}^{u})=\exp(-\frac{u}{g(0)})\mbox{ and }Q_{0}(x,y\notin\mathcal{I}^{u})=\exp(-u\frac{2}{g(0)+g(x-y)}).\label{eq:RIOneandTwoPoint}\end{eqnarray}
If $A$ and $B$ are two disjoint finite sets in $\mathbb{Z}^{d+1}$
that are {}``far apart'' we have the following independence result
which is a direct consequence of Lemma 2.1 of \cite{Belius2010} \begin{equation}
\begin{array}{c}
|Q_{0}(A\subset\mathcal{I}^{u},B\subset\mathcal{I}^{u})-Q_{0}(A\subset\mathcal{I}^{u})Q_{0}(B\subset\mathcal{I}^{u})|\le cu\frac{\textnormal{cap}(A)\textnormal{cap}(B)}{(d(A,B))^{d-1}}\\
\mbox{for all disjoint }A,B\subset\mathbb{Z}^{d+1}\mbox{ and }u>0.\end{array}\label{eq:Lemma21Belius2010}\end{equation}
The next lemma gives a bound on certain sums of the {}``two point
probability'' $Q_{0}(x,y\notin\mathcal{I}^{u})$. It is a generalisation
of Lemma 2.5 of \cite{Belius2010}.
\begin{lem}
\label{lem:Calc}($d\ge2$)There is a constant $c_{1}>1$ such that
for any $F\subset\mathbb{Z}^{d+1},u\ge0$ and $a>0$\begin{equation}
\sum_{0<|x|<a,x\in F}Q_{0}(0,x\notin\mathcal{I}^{g(0)u})\le e^{-2u}\{|F|\wedge a^{d+1}+cu|F|^{\frac{2}{d+1}}\}+ce^{-c_{1}u}.\label{eq:LemmaCalcBound}\end{equation}
\end{lem}
\begin{proof}
We assume $|F|\ge1$. The left-hand side of \prettyref{eq:LemmaCalcBound}
equals $I_{1}+I_{2}$ where\begin{eqnarray}
I_{1} & \overset{\eqref{eq:RIOneandTwoPoint}}{=} & \sum_{0<|x|<u\wedge a,x\in F}\exp\left(-2u\left(1+g(x)/g(0)\right)^{-1}\right)\mbox{ and }\label{eq:DefOfI1}\\
I_{2} & \overset{\eqref{eq:RIOneandTwoPoint}}{=} & \sum_{u\wedge a\le|x|<a,x\in F}\exp\left(-2u\left(1+g(x)/g(0)\right)^{-1}\right).\label{eq:DefOfI2}\end{eqnarray}
We first bound $I_{1}$. Note that for $x\ne0$ we have $g(x)<g(e_{1})<g(0)$,
where $e_{1}$ is a unit vector in $\mathbb{Z}^{d}$, so that the
summand in \eqref{eq:DefOfI1} is bounded by $\exp(-c_{1}^{'}u)$
where $c_{1}^{'}=2\left(1+g(1)/g(0)\right)^{-1}>1$. Therefore we
find that\begin{equation}
I_{1}\le cu^{d+1}\exp(-c_{1}^{'}u)\le c\exp(-c_{1}u)\label{eq:BoundOnI1}\end{equation}
for a constant $c_{1}$ such that $c_{1}^{'}>c_{1}>1$. Next to bound
$I_{2}$ we use the elementary inequality $(1-x)^{-1}\ge1+x,x\ge0$
to get \begin{eqnarray*}
I_{2} & \le & \sum_{u\wedge a\le|x|<a,x\in F}\exp\left(-2u\left(1-g(x)/g(0)\right)\right)\\
 & \le & \exp(-2u)\sum_{u\wedge a\le|x|<a,x\in F}\left(1+cug(x)/g(0)\right)\end{eqnarray*}
where in the last inequality we have used that $ug(x)\overset{\prettyref{eq:ZdGreensFuncBound}}{\le}(cu\cdot u^{1-d}\wedge c)\le c$
for $u\wedge a\le|x|<a$. We thus get\[
I_{2}\le\exp(-2u)\left\{ |F|\wedge(ca^{d+1})+cu\sum_{x\in F}g(x)\right\} .\]
Let $F=\{f_{1},...,f_{|F|}\}$ with $|f_{1}|\le|f_{2}|\le...\le|f_{|F|}|$
and let $h_{1},h_{2},..$ be an enumeration of $\mathbb{Z}^{d+1}$
such that $|h_{1}|\le|h_{2}|\le...$. Then since $g(f_{i})\overset{\prettyref{eq:ZdGreensFuncBound}}{\le}c|f_{i}|^{1-d}\le c|h_{i}|^{1-d}$
and $\{h_{1},...,h_{|F|}\}$ is contained in the Euclidean ball $B'$
of radius $c|F|^{\frac{1}{d+1}}$ around the origin we have $\sum_{x\in F}g(x)\le\sum_{i=1}^{|F|}(c|h_{i}|^{1-d}\wedge g(0))\le\sum_{x\in B'}(c|x|^{1-d}\wedge g(0))\le c|F|^{2/(d+1)}$.
Therefore we see that\[
I_{2}\le c\exp(-2u)\left\{ |F|\wedge a^{d+1}+cu|F|^{2/(d+1)}\right\} ,\]
which in combination with \eqref{eq:BoundOnI1} yields \eqref{eq:LemmaCalcBound}.
\end{proof}
In proving \prettyref{thm:GumbelForLocTime} we will often consider
events similar to $\{L_{C_{F_{N}}}\le N^{d}g(0)\{\log|F_{N}|+z\}\}$.
To simply formulas we define \begin{equation}
u_{F}(z)=g(0)\{\log|F|+z\}\mbox{ and }u_{N}=u_{F_{N}},\label{eq:DefOfuF}\end{equation}
so that the previous event coincides with $\{L_{C_{F_{N}}}\le N^{d}u_{N}(z)\}$.
We denote the cumulative distribution function of the standard Gumbel
distribution by\begin{equation}
G(z)=\exp(-e^{-z}).\label{eq:GumbelCDF}\end{equation}

We now state a quantitative version of \prettyref{eq:coverlevelgumbnonquant}
which will be what we actually use in the proof of \prettyref{thm:GumbelForLocTime}.
Recall the definition \prettyref{eq:coverleveldef} of the cover level
$\tilde{C}_{F}$ of a set $F$. We have that (see Theorem 0.1 of \cite{Belius2010})
for all finite non-empty $F\subset\mathbb{Z}^{d+1}$\begin{eqnarray}
 & \sup_{z\ge-\log|F|}|Q_{0}(\tilde{C}_{F}\le u_{F}(z))-G(z)|\le c|F|^{-c}.\label{eq:QuantCovLevResult-1}\end{eqnarray}

\section{\label{sec:Applications}Applications: Convergence of cover times
and point process of uncovered vertices}

In this section we will derive from \prettyref{thm:GumbelForLocTime}
and the coupling \prettyref{eq:IntroOneBoxCoupling} the convergence
in law of the rescaled cover times (i.e. \prettyref{eq:RealCovTimeConvInformal}),
the convergence in law of the point process of vertices covered last
(i.e. \prettyref{eq:PointProcConvInformal}) and the statement that
{}``the last two vertices to be hit are far apart'' (i.e. \prettyref{eq:LastTwoIndiansInformal}),
in the form of \prettyref{cor:RealCovTimeConv}, \prettyref{cor:PointProcessConv}
and \prettyref{cor:LastTwoIndians} respectively. \prettyref{cor:RealCovTimeConv}
will follow easily from \prettyref{thm:GumbelForLocTime} once we
have related the random walk local time $L_{n}$ to the local time
of Brownian motion using \prettyref{eq:RWLocTimeToBMLocTime}. To
prove \prettyref{cor:PointProcessConv} we will use Kallenberg's theorem
which allows us to conclude that $\mathcal{N}_{N}^{z}$ converges
weakly to a homogeneous Poisson point process by checking two conditions
involving convergence of the intensity measure and the probability
that the point measure does not charge a set. Finally \prettyref{cor:LastTwoIndians}
follows from \prettyref{cor:PointProcessConv} by a calculation involving
the Palm measure of the limit of the $\mathcal{N}_{N}^{z}$.
\begin{cor}[Convergence in law for cover times]
\label{cor:RealCovTimeConv}$(d\ge2)$ Let $F_{N}$ be a sequence
of sets as in \prettyref{thm:GumbelForLocTime}. \textup{\emph{Then
under $P$}}\begin{equation}
\frac{C_{N}}{(N^{d}\log|F_{N}|)^{2}}\overset{\mbox{law}}{\rightarrow}\zeta(\frac{g(0)}{\sqrt{d+1}}),\mbox{ as }N\rightarrow\infty,\label{eq:RealCovTimeConv}\end{equation}
where $C_{N}=C_{F_{N}}$.\end{cor}
\begin{proof}
We set $u=N^{d}\log|F_{N}|$. For any $z\ge0$ we have\begin{equation}
\begin{array}{cccc}
\{L_{C_{N}}<L_{[zu^{2}]}\} & \overset{\prettyref{eq:DefOfLocTime}}{\subset} & \{C_{N}\le zu^{2}\} & \overset{\prettyref{eq:DefOfLocTime}}{\subset}\{L_{C_{N}}\le L_{[zu^{2}]}\}.\end{array}\label{eq:CNSandwich}\end{equation}
It follows from \prettyref{eq:GumbelForLocTime} that $L_{C_{N}}/u\rightarrow g(0)$
in probability as $N\rightarrow\infty$, so that for any $\delta\in(0,1)$\begin{equation}
\begin{array}{ccl}
\underline{\lim}_{N\rightarrow\infty}P(g(0)(1+\delta)u\le L_{[zu^{2}]}) & \overset{\prettyref{eq:CNSandwich}}{\le} & \overline{\underline{\lim}}_{N\rightarrow\infty}P(C_{N}\le zu^{2})\\
 & \overset{\prettyref{eq:CNSandwich}}{\le} & \overline{\lim}_{N\rightarrow\infty}P(g(0)(1-\delta)u\le L_{[zu^{2}]}).\end{array}\label{eq:CNSandwhich2}\end{equation}
But by \prettyref{eq:GammaLocRelation} we have $P(g(0)(1\pm\delta)u\le L_{[zu^{2}]})=P(\gamma_{g(0)u(1\pm\delta)}\le zu^{2})$
and thus from \prettyref{eq:CNSandwhich2}, \prettyref{eq:RWLocTimeToBMLocTime}
(using that $u\rightarrow\infty$ as $N\rightarrow\infty$) and the
continuity of the limit law (see \eqref{eq:ZetaCDF}) we get:\begin{eqnarray*}
\mathbb{P}(\zeta(\frac{1}{\sqrt{d+1}})\le\frac{z}{(g(0)(1+\delta))^{2}}) & \le & \overline{\underline{\lim}}_{N\rightarrow\infty}P(C_{N}\le zu^{2})\\
 & \le & \mathbb{P}(\zeta(\frac{1}{\sqrt{d+1}})\le\frac{z}{(g(0)(1-\delta))^{2}}).\end{eqnarray*}
Once again by the continuity of the law of $\zeta(\cdot)$ we can
now let $\delta\downarrow0$ to get $\lim_{N\rightarrow\infty}P(C_{N}\le zu^{2})=\mathbb{P}(\zeta(\frac{1}{\sqrt{d+1}})\le\frac{z}{g(0)^{2}})$
and \prettyref{eq:RealCovTimeConv} then follows by the scaling relation
\prettyref{eq:ZetaScalingRelation}.
\end{proof}
Next we prove the weak convergence of the point process of vertices
covered last (recall the definition of this process from \prettyref{eq:PointProcOfPointsCovereLastDef}).
\begin{cor}[Convergence of point process of vertices covered last]
\label{cor:PointProcessConv}($z\in\mathbb{R}$) In the topology
of point processes\textup{\begin{equation}
\begin{array}{c}
\mathcal{N}_{N}^{z}\mbox{ converges weakly to a Poisson point process}\\
\mbox{ on }(\mathbb{R}/\mathbb{Z})^{d}\times\mathbb{R}\mbox{ of intensity }\exp(-z)\lambda,\end{array}\label{eq:asdsad}\end{equation}
}where $\lambda$ denotes Lebesgue measure on $(\mathbb{R}/\mathbb{Z})^{d}\times\{0\}$.\end{cor}
\begin{proof}
Let $\mathcal{M}_{N}^{z}=\sum_{x\in\mathbb{T}_{N}\times\{0\}}1_{\{L_{H_{x}}\ge N^{d}u(z)\}}$.
We will first show \eqref{eq:asdsad} with $\mathcal{M}_{N}^{z}$
in place of $\mathcal{N}_{N}^{z}$. By Kallenberg's theorem (see Proposition
3.22, p. 156 of \cite{Resnick_ExtremeValRegVarAndPP}) to show \eqref{eq:asdsad}
with $\mathcal{M}_{N}^{z}$ in place of $\mathcal{N}_{N}^{z}$ it
suffices to check that for all $I$ in\[
\mathcal{J}=\{I:I\mbox{ a union of products of open intervals in }(\mathbb{R}/\mathbb{Z})^{d}\times\mathbb{R},\lambda(I)>0\},\]
the following two statements hold:\begin{eqnarray}
\lim_{N\rightarrow\infty}E\mathcal{M}_{N}^{z}(I) & = & \exp(-z)\lambda(I),\label{eq:KallenbergToCheck1-1}\\
\lim_{N\rightarrow\infty}P(\mathcal{M}_{N}^{z}(I)=0) & = & \exp(-e^{-z}\lambda(I)).\label{eq:KallenbergToCheck2-1}\end{eqnarray}
For any $I\in\mathcal{J}$ define $F_{N}^{'}=NI\cap\mathbb{T}_{N}\times\{0\}$.
Note that because of the special form of $I$ (recall $\lambda(I)>0$)
we have \begin{equation}
|F_{N}^{'}|\rightarrow\infty\mbox{ and }\frac{|F_{N}^{'}|}{|\mathbb{T}_{N}\times\{0\}|}\rightarrow\lambda(I)\mbox{ as }N\rightarrow\infty.\label{eq:niceconvergencetolamda}\end{equation}
To show \eqref{eq:KallenbergToCheck2-1} we note that\begin{equation}
P(\mathcal{M}_{N}^{z}(I)=0)=P(L_{C_{N}^{'}}<N^{d}u(z)),\label{eq:EqualityQWERTY}\end{equation}
where $C_{N}^{'}=C_{F_{N}^{'}}$ and $u(z)=u_{\mathbb{T}_{N}\times\{0\}}(z)$
(recall the notation from \prettyref{eq:DefOfuF}). Let $u_{N}^{'}=u_{F_{N}^{'}}$
and note that $u(z)=u_{N}^{'}(z-\log\frac{|F_{N}^{'}|}{|\mathbb{T}_{N}\times\{0\}|})$
so that for all $a>0$ and $N\ge c(I,z,a)$ we have \begin{equation}
u_{N}^{'}(z-a-\log\lambda(I))\overset{\prettyref{eq:niceconvergencetolamda}}{<}u(z)\overset{\prettyref{eq:niceconvergencetolamda}}{<}u_{N}^{'}(z+a-\log\lambda(I)).\label{eq:InequalityQWERTY}\end{equation}
Therefore using \prettyref{eq:GumbelForLocTime} with $F_{N}^{'}$
and $C_{N}^{'}$ in the place of $F_{N}$ and $C_{F_{N}}$ we get
\begin{eqnarray*}
\exp(-e^{-z+a}\lambda(I)) & \overset{\eqref{eq:GumbelForLocTime}}{=} & \underline{\lim}_{N\rightarrow\infty}P(L_{C_{N}^{'}}\le N^{d}u_{N}^{'}(z-a-\log\lambda(I)))\\
 & \overset{\prettyref{eq:EqualityQWERTY},\prettyref{eq:InequalityQWERTY}}{\le} & \underline{\overline{\lim}}_{N\rightarrow\infty}P(\mathcal{M}_{N}^{z}(I)=0)\\
 & \overset{\prettyref{eq:EqualityQWERTY},\prettyref{eq:InequalityQWERTY}}{\le} & \overline{\lim}_{N\rightarrow\infty}P(L_{C_{N}^{'}}\le N^{d}u_{N}^{'}(z+a-\log\lambda(I)))\\
 & \overset{\eqref{eq:GumbelForLocTime}}{=} & \exp(-e^{-z-a}\lambda(I)),\end{eqnarray*}
so that so letting $a\downarrow0$ we find \prettyref{eq:KallenbergToCheck2-1}.

It remains to check \eqref{eq:KallenbergToCheck1-1}. Note that\begin{equation}
E\mathcal{M}_{N}^{z}(I)\overset{\prettyref{eq:GammaLocRelation}}{=}\sum_{x\in F_{N}^{'}}P\left(\gamma_{N^{d}u(z)}\le H_{x}\right).\label{eq:SplitIntoSum-1}\end{equation}
Let us now record (for use now and later) that\begin{equation}
\begin{array}{c}
\mbox{for }F\subset\mathbb{T}_{N}\times[-\frac{N}{2},\frac{N}{2}]\mbox{ with }|F|\ge2,z\in[-\frac{1}{2}\log|F|,-\frac{1}{2}\log|F|],\\
a\in(0,\frac{1}{10}],\delta=(\log N)^{-3/2}\mbox{ and }N\ge c(a)\mbox{ we have }\\
\frac{2}{\sqrt{N}}\le u_{F}(z-a)\le u_{F}(z)(1-\delta)\le u_{F}(z)(1+\delta)\le u_{F}(z+a)\le c\log N.\end{array}\label{eq:uMinusuPluscondition}\end{equation}
Thus for any $a\in(0,\frac{1}{10}]$ and $N\ge c(z,a)$ it follows
from \eqref{eq:ExcTimeToLocTime} (note that $\frac{1}{2}>\delta\ge c_{0}\frac{r_{N}}{h_{N}}$
for $N\ge c$, with $\delta$ as in \prettyref{eq:uMinusuPluscondition})
that $P(D_{[K_{N}u(z+a)]}<\gamma_{N^{d}u(z)}),P(D_{[K_{N}u(z-a)]}\ge\gamma_{N^{d}u(z)})\le ce^{-cN^{c}}$
and thus  \begin{eqnarray}
\begin{array}{ccl}
P\left(D_{[K_{N}u(z+a)]}<H_{x}\right)-ce^{-cN^{c}} & \le & P\left(\gamma_{N^{d}u(z)}\le H_{x}\right)\\
 & \le & P\left(D_{[K_{N}u(z-a)]}<H_{x}\right)+ce^{-cN^{c}}.\end{array}\label{eq:Bound1-1}\end{eqnarray}
Now since $\{D_{k}<H_{x}\}\overset{\eqref{eq:TraceDefinition}}{=}\{x\notin X(0,D_{k})\}$
for any $k\ge1$ we have \begin{equation}
P(D_{[K_{N}u(z\pm a)]}<H_{x})=P(x\notin X(0,D_{[K_{N}u(z\pm a)]})).\label{eq:Equality}\end{equation}
We can now use \prettyref{eq:IntroOneBoxCoupling} twice to get that
if $a\in(0,\frac{1}{10}]$ and $N\ge c(z,a)$ then\begin{eqnarray}
\begin{array}{ccc}
P\left(x\notin X(0,D_{[K_{N}u(z-a)]})\right) & \le & Q_{0}\left(x\notin\mathcal{I}^{u(z-2a)}\right)+cN^{-3d},\\
P\left(x\notin X(0,D_{[K_{N}u(z+a)]})\right) & \ge & Q_{0}\left(x\notin\mathcal{I}^{u(z+2a)}\right)-cN^{-3d},\end{array}\label{eq:DominateSplitGuys}\end{eqnarray}
where we have also used \prettyref{eq:uMinusuPluscondition} with
$u(z\pm a)$ in the place of $z$, that $\delta\ge c_{2}\frac{r_{N}}{h_{N}}$
($\delta$ as in \prettyref{eq:uMinusuPluscondition}) for $N\ge c$,
the fact that $\mathcal{I}^{(1\pm\delta)u_{F}(z\pm a)}$ has the same
law under the coupling $Q_{1}$ as under the canonical probability
$Q_{0}$ and \prettyref{eq:InterlacementIsIncreasing}. But $Q_{0}\left(x\notin\mathcal{I}^{u(z\pm2a)}\right)=\frac{\exp(-z\mp2a)}{|\mathbb{T}_{N}\times\{0\}|}$
by \prettyref{eq:RIOneandTwoPoint}, so combining \prettyref{eq:Bound1-1},
\eqref{eq:Equality} and \prettyref{eq:DominateSplitGuys} we get
that if $N\ge c(z,a)$ then

\texttt{\begin{eqnarray*}
\frac{\exp(-z-2a)}{|\mathbb{T}_{N}\times\{0\}|}-cN^{-3d}\le P\left(\gamma_{N^{d}u(z)}\le H_{x}\right)\le\frac{\exp(-z+2a)}{|\mathbb{T}_{N}\times\{0\}|}+cN^{-3d}.\end{eqnarray*}
}So by \prettyref{eq:niceconvergencetolamda} and \prettyref{eq:SplitIntoSum-1}
and we have that for all $z\in\mathbb{R}$ and $a\in(0,\frac{1}{10}]$

\[
\lambda(I)\exp(-z-2a)\le\overline{\underline{\lim}}_{N\rightarrow\infty}E\mathcal{M}_{N}^{z}(I)\le\lambda(I)\exp(-z+2a).\]
We simply have to take $a\downarrow0$ to get \eqref{eq:KallenbergToCheck1-1}.
This completes the proof of \eqref{eq:asdsad} with $\mathcal{M}_{N}^{z}$
in place of $\mathcal{N}_{N}^{z}$. Thus the limit of the Laplace
functionals of $\mathcal{M}_{N}^{z}$ is the map $f\rightarrow\exp(-e^{-z}\int(1-e^{-f(x)})\lambda(dx))$,
which is continuous in $z$. By using the inequality $\mathcal{M}_{N}^{z+\delta}\overset{\prettyref{eq:PointProcOfPointsCovereLastDef}}{\le}\mathcal{N}_{N}^{z}\overset{\prettyref{eq:PointProcOfPointsCovereLastDef}}{\le}\mathcal{M}_{N}^{z}$,
valid for any $z\in\mathbb{R},\delta>0$ and $N\ge1$, and letting
$\delta\downarrow0$ we see that the limit of the Laplace functionals
of $\mathcal{N}_{N}^{z}$ is $f\rightarrow\exp(-e^{-z}\int(1-e^{-f(x)})\lambda(dx))$
as well, and therefore \eqref{eq:asdsad} holds.
\end{proof}
Finally we derive \prettyref{cor:LastTwoIndians} from \prettyref{cor:PointProcessConv}.
\begin{cor}[Last vertices to be hit are far apart]
\label{cor:LastTwoIndians}Let the random vector $(Y_{1,}Y_{2},...,Y_{|\mathbb{T}_{N}\times\{0\}|})$
be the vertices of $\mathbb{T}_{N}\times\{0\}$ ordered by their entrance
time with $Y_{1}$ being hit last, so that:\[
C_{\mathbb{T}_{N}\times\{0\}}=H_{Y_{1}}>H_{Y_{2}}>..>H_{Y_{|\mathbb{T}_{N}\times\{0\}|}}=H_{\mathbb{T}_{N}\times\{0\}}.\]
Then for all $k\ge2$\begin{equation}
\lim_{\delta\rightarrow0}\overline{\lim}_{N\rightarrow\infty}P(\exists1\le i<j\le k\mbox{ such that }d(Y_{i},Y_{j})\le\delta N)=0,\label{eq:LastTwoIndians}\end{equation}
or in other words for large $N$ the last $k$ vertices of $\mathbb{T}_{N}\times\{0\}$
to be hit are separated, at typical distance of order $N$.\end{cor}
\begin{proof}
For $z\in\mathbb{R}$, \prettyref{cor:PointProcessConv} says that
$\mathcal{N}_{N}^{z}$ converges weakly to $\mathcal{N}^{z}$, a Poisson
point process on $(\mathbb{R}/\mathbb{Z})^{d}\times\mathbb{R}$ of
intensity $\exp(-z)\lambda$ ($\lambda$ as in \prettyref{cor:PointProcessConv}).
Note that for any $z\in\mathbb{R}$ and $\delta>0$ the limsup of
the probability in \prettyref{eq:LastTwoIndians} is bounded above
by \begin{eqnarray}
\underset{N\rightarrow\infty}{\overline{\lim}}P(\exists x\ne y\in\textnormal{Supp}(\mathcal{N}_{N}^{z}),d(x,y)\le\delta)+\mathbb{P}(\mathcal{N}^{z}((\mathbb{R}/\mathbb{Z})^{d}\times\mathbb{R})<k),\label{eq:SplittGuys}\end{eqnarray}
where we have used that $P(\mathcal{N}_{N}^{z}((\mathbb{R}/\mathbb{Z})^{d}\times\mathbb{R})<k)\rightarrow\mathbb{P}(\mathcal{N}^{z}((\mathbb{R}/\mathbb{Z})^{d}\times\mathbb{R})<k)$
as $N\rightarrow\infty$. Let $f:(\mathbb{R}/\mathbb{Z})^{d}\times\mathbb{R}\rightarrow[0,1]$
be a continuous function such that $f(x)=1$ if $d(0,x)\le\delta$
and $f(x)=0$ if $d(0,x)\ge2\delta$. Then\begin{equation}
\begin{array}{rc}
\lim_{\delta\rightarrow0}\overline{\lim}_{N\rightarrow\infty}P(\exists x\ne y\in\mbox{Supp}(\mathcal{N}_{N}^{z})\mbox{ s.t. }d(x,y)\le\delta) & \le\\
\lim_{\delta\rightarrow0}\overline{\lim}_{N\rightarrow\infty}P(\sum_{x,y\in\textnormal{Supp}(\mathcal{N}_{N}^{z}),x\ne y}f(x-y)\ge1) & \le\\
\lim_{\delta\rightarrow0}\overline{\lim}_{N\rightarrow\infty}E(\sum_{x,y\in\textnormal{Supp(}\mathcal{N}_{N}^{z}),x\ne y}f(x-y)),\end{array}\label{eq:MarkovETC}\end{equation}
where we have used Markov's inequality going from the middle to the
last line. Consider the sum $\sum_{x,y\in\textnormal{Supp(}\mathcal{N}_{N}^{z}),x\ne y}f(x-y)=\mathcal{N}_{N}^{z}\otimes\mathcal{N}_{N}^{z}(f(\cdot-\cdot))-f(0)\mathcal{N}_{N}^{z}((\mathbb{R}/\mathbb{Z})^{d}\times\mathbb{R})$.
Note that $\mathcal{N}_{N}^{z}\otimes\mathcal{N}_{N}^{z}$ tends weakly
to $\mathcal{N}^{z}\otimes\mathbf{\mathcal{N}}^{z}$ so that\begin{equation}
\lim_{N\rightarrow\infty}E\left[\sum_{x,y\in\textnormal{Supp(}\mathcal{N}_{N}^{z}),x\ne y}f(x-y)\right]=\mathbb{E}\left[\sum_{x,y\in\textnormal{Supp(}\mathcal{N}^{z}),x\ne y}f(x-y)\right].\label{eq:ApplicationofWeakConvGen}\end{equation}
By Proposition 13.1.VII p. 280 of \cite{DaleyVereJones_IntroToTheTheoryOfPointProcesses}
the local Palm distribution for $\mathcal{N}^{z}$ at $x\in(\mathbb{R}/\mathbb{Z})^{d}\times\mathbb{R}$
is the distribution of $\mathcal{N}^{z}+\delta_{x}$. Therefore (by
e.g. Proposition 13.1.IV p. 273 of \cite{DaleyVereJones_IntroToTheTheoryOfPointProcesses})
the right hand side of \eqref{eq:ApplicationofWeakConvGen} equals\begin{equation}
\begin{array}{ccl}
e^{-z}\intop\mathbb{E}\left[(\mathcal{N}^{z}+\delta_{x})(1_{\{y\ne x\}}f(x-\cdot))\right]\lambda(dx) & = & e^{-2z}\iint f(x-y)\lambda(dy)\lambda(dx)\\
 & \le & c(z)\delta^{d}.\end{array}\label{eq:ApplicationofPalmMeasureGen}\end{equation}
Combining \prettyref{eq:ApplicationofWeakConvGen} and \prettyref{eq:ApplicationofPalmMeasureGen}
we get that the right-hand side of \prettyref{eq:MarkovETC} equals
zero, and therefore from \prettyref{eq:SplittGuys} we see that for
all $z\in\mathbb{R}$\begin{equation}
\lim_{\delta\rightarrow0}\underset{N\rightarrow\infty}{\overline{\lim}}P(\exists1\le i<j\le k\mbox{ such that }d(Y_{i},Y_{j})\le\delta N)\le\mathbb{P}(\mathcal{N}^{z}((\mathbb{R}/\mathbb{Z})^{d}\times\mathbb{R})<k).\label{eq:LastTwoIndiansAlmostDone}\end{equation}
But if we take $z\rightarrow-\infty$ then the right hand side of
\prettyref{eq:LastTwoIndiansAlmostDone} tends to zero, so \prettyref{eq:LastTwoIndians}
follows.
\end{proof}

\section{\label{sec:ConvToGumbel}Convergence to Gumbel}

We now turn to the proof of \prettyref{thm:GumbelForLocTime}. Intuitively
speaking \prettyref{eq:GumbelForLocTime} says that {}``$L_{C_{F_{N}}}$
is approximately distributed as a Gumbel random variable with location
$N^{d}g(0)\log|F_{N}|$ and scale $N^{d}g(0)$'' (recall that a Gumbel
random variable with location $\mu$ and scale $\beta$ has cumulative
distribution function $\exp(-e^{-(x-\mu)/\beta})$ and that the standard
Gumbel distribution has location $0$ and scale $1$). The first step
of the proof is to use \prettyref{eq:ExcTimeToLocTime} to reduce
this to the statement {}``the number of excursions needed to cover
$F_{N}$ is approximately Gumbel distributed with location $K_{N}g(0)\log|F_{N}|$
and scale $K_{N}g(0)$''. To prove this latter statement we want
to use the coupling result \prettyref{thm:CouplingManyBoxes} from
\prettyref{sec:CoupleRWwithRI} that couples the trace of the random
walk $X_{\cdot}$ with random interlacements, and apply \prettyref{eq:QuantCovLevResult-1}
(i.e. Theorem 0.1 of \cite{Belius2010}) which gives the corresponding
distributional limit result in the random interlacements model. It
is however not immediately obvious how this might be done because
\prettyref{thm:CouplingManyBoxes} only couples the trace of $X_{\cdot}$
in several \emph{separated} {}``local boxes'' of side length $N^{1-\varepsilon},0<\varepsilon<1$
centred in $\mathbb{T}_{N}\times[-\frac{N}{2},\frac{N}{2}]$ with
random interlacements, and in general it will not be possible to cover
$F_{N}$ by a collection of such local boxes. We are able to deal
with this problem by splitting the sequence $F_{N}$ into two subsequences,
one with $F_{N}$ that are small in the sense that $|F_{N}|\le N^{1/8}$
and one with $F_{N}$ that are big in the sense that $|F_{N}|>N^{1/8}$.
In the first case (small $F_{N}$) we \emph{are }able to apply \prettyref{thm:CouplingManyBoxes}
and \prettyref{eq:QuantCovLevResult-1} to get that the number of
excursions needed to cover $F_{N}$ is approximately Gumbel distributed
with appropriate parameters. Moreover we are able to reduce the second
case (big $F_{N}$) to the first case.

We now state \prettyref{pro:PropSmallSets} which deals the first
case. Recall \prettyref{eq:qnotation}, \prettyref{eq:DefOfuF} and
\prettyref{eq:GumbelCDF} for notation.
\begin{prop}[Gumbel for {}``small'' $F$]
\label{pro:PropSmallSets}($d\ge2$) For any $\theta\in(0,\frac{1}{10}]$,
$N\ge1$ and $F\subset\mathbb{T}_{N}\times[-\frac{N}{2},\frac{N}{2}]$
such that $N^{1/8}\ge|F|\ge c(\theta)$ we have:\begin{equation}
\sup_{z\in[-\frac{1}{2}\log|F|,\frac{1}{2}\log|F|],l\in[-N,N]}|P_{q_{l}}(F\subset X(0,D_{[K_{N}u_{F}(z)]}))-G(z)|\le\theta.\label{eq:PropSmallSets}\end{equation}

\end{prop}

We postpone the proof of \prettyref{pro:PropSmallSets} until after
the proof of \prettyref{thm:GumbelForLocTime} and instead state \prettyref{pro:PropLargeSets}
which deals with the second case.
\begin{prop}[Gumbel for {}``big'' $F_{N}$]
\label{pro:PropLargeSets}($d\ge2$) Let $F_{N}$ be a sequence of
sets as in \prettyref{thm:GumbelForLocTime} with the additional requirement
that $\left|F_{N}\right|>N^{1/8}$ for all $N\ge1$. Then for all
$z\in\mathbb{R}$ (recall that $u_{N}$ is a shorthand for $u_{F_{N}}$):
\begin{equation}
P(F_{N}\subset X(0,D_{[K_{N}u_{N}(z)]}))\rightarrow G(z)\mbox{ as }N\rightarrow\infty.\label{eq:PropLargeSets}\end{equation}

\end{prop}
The proof of \prettyref{pro:PropLargeSets} is also postponed until
after the proof of \prettyref{thm:GumbelForLocTime}, which we now
start.
\begin{proof}[Proof of \prettyref{thm:GumbelForLocTime}]
We write $C_{N}$ for $C_{F_{N}}$. By \prettyref{eq:GammaLocRelation}
we have $P(L_{C_{N}}<N^{d}u_{N}(z))=P(C_{N}<\gamma_{N^{d}u_{N}(z)})$
for all $z\ge-\frac{1}{2}\log|F_{N}|$. Also note that for all $z\in\mathbb{R}$,
$a\in(0,\frac{1}{10}]$ and $N\ge c(z,a,(F_{N})_{N\ge1})$ it holds
that \begin{eqnarray*}
P(C_{N}<\gamma_{N^{d}u_{N}(z)}) & \ge & P(C_{N}\le D_{\left[u_{N}(z-a)K_{N}\right]})-P(D_{\left[u_{N}(z-a)K_{N}\right]}\ge\gamma_{N^{d}u_{N}(z)}),\\
P(C_{N}<\gamma_{N^{d}u_{N}(z)}) & \le & P(C_{N}\le D_{\left[u_{N}(z+a)K_{N}\right]})+P(D_{\left[u_{N}(z+a)K_{N}\right]}<\gamma_{N^{d}u_{N}(z)}).\end{eqnarray*}
Now using that $P(C_{N}\le D_{\left[u_{N}(z\pm a)K_{N}\right]})\overset{\prettyref{eq:TraceDefinition}}{=}P(F_{N}\subset X(0,D_{\left[u_{N}(z\pm a)K_{N}\right]}))$
and \prettyref{eq:uMinusuPluscondition} with $F=F_{N}$ (similarly
to under \prettyref{eq:uMinusuPluscondition} but with $z\pm a$ in
place of $z$) it follows from the above and two applications of \prettyref{eq:ExcTimeToLocTime}
that for all $z\in\mathbb{R}$ and $a\in(0,\frac{1}{10}]$ \begin{equation}
\begin{array}{ccl}
\underline{\lim}_{N\rightarrow\infty}P(F_{N}\subset X(0,D_{\left[u_{N}(z-a)K_{N}\right]})) & \le & \overline{\underline{\lim}}_{N\rightarrow\infty}P(L_{C_{N}}<N^{d}u_{N}(z))\\
 & \le & \overline{\lim}_{N\rightarrow\infty}P(F_{N}\subset X(0,D_{\left[u_{N}(z+a)K_{N}\right]})).\end{array}\label{eq:BajsMacka}\end{equation}
But by splitting the sequence $F_{N}$ into two subsequences and applying
\prettyref{pro:PropSmallSets} (recall that $P=P_{q_{0}}$) and \prettyref{pro:PropLargeSets}
it follows that $\lim_{N\rightarrow\infty}P(F_{N}\subset X(0,D_{\left[u_{N}(z\pm a)K_{N}\right]}))=G(z\pm a)$.
We can thus replace the right- and left-hand sides of \eqref{eq:BajsMacka}
with $G(z+a)$ and $G(z-a)$ respectively, and then let $a\downarrow0$
and use the continuity of $G$ to conclude that $\lim_{N\rightarrow\infty}P(L_{C_{N}}<N^{d}u_{N}(z))=G(z)$
for all $z\in\mathbb{R}$, and therefore that \eqref{eq:GumbelForLocTime}
holds.
\end{proof}

We now turn to the proof of \prettyref{pro:PropSmallSets} which deals
with {}``small'' sets $F_{N}$. It turns out that such small sets
can be chopped into pieces $S_{1},S_{2},...,S_{k}$ such that each
individual piece is contained in a local box of side length $N^{1/2}$,
and is separated from the other pieces by a distance of at least $|F|^{3}$.
The separation will allow us to apply \prettyref{thm:CouplingManyBoxes}
to conclude that the traces left by $X_{\cdot}$ on each $S_{i}$
are approximated by $k$ independent random interlacements, and thus
that the number of excursions needed to cover $F$ is approximately
$K_{N}$ times $\max_{k}\tilde{C}_{S_{k}}$, the maximum of the cover
levels of the $S_{k}$ by the $k$ random interlacements. Using \prettyref{eq:Lemma21Belius2010}
we will be able to assemble the pieces (placing them suitably far
apart) into a single random interlacement, so that $\max_{k}\tilde{C}_{S_{k}}$
is close in distribution to the cover level $\tilde{C}_{\tilde{F}}$
of a set $\tilde{F}$ which consists of the reassembled pieces $S_{1},S_{2},...,S_{k}$
and thus has the same cardinality as $F$. It will then be straight-forward
to prove that $\tilde{C}_{\tilde{F}}$ (and thus $\max_{k}\tilde{C}_{S_{k}}$)
is approximately distributed as a Gumbel random variable with location
$g(0)\log|F|$ and scale $g(0)$ by applying \prettyref{eq:QuantCovLevResult-1}.
This in turn will imply that the number of excursions needed to cover
$F$ is approximately distributed as a Gumbel random variable with
location $K_{N}g(0)\log|F|$ and scale $K_{N}g(0)$, which is essentially
speaking what \prettyref{pro:PropSmallSets} claims.
\begin{proof}[Proof of \prettyref{pro:PropSmallSets}]
Construct a graph $(F,E_{F})$ with vertices in $F$ and edge set
$E_{F}$ such that $\{a,b\}\in E_{F}$ iff $a,b\in F$ and $d_{\infty}(a,b)\le|F|^{3}$.
Let $S_{1},S_{2},...,S_{k}$ be the connected components of $(F,E_{F})$
and let $x_{1},...,x_{k}\in E_{N}$ be arbitrarily selected $x_{i}\in S_{i},i=1,...,k$.
Then for each $i$ and for all $y\in S_{i}$ we have $d_{\infty}(y,x_{i})\le|F|\times|F|^{3}=|F|^{4}\le N^{1/2}$,
where the last inequality follows by our assumption on the cardinality
of $F$. Therefore $S_{i}\subset B(x_{i},N^{1/2})$ for all $i$,
and if $i\ne j$ we have $d_{\infty}(S_{i},S_{j})>|F|^{3}$.

We now apply \prettyref{thm:CouplingManyBoxes} with $\varepsilon=\frac{1}{2}$,
$u=u_{F}(z)$, $\delta$ as in \prettyref{eq:uMinusuPluscondition}
(using also \prettyref{eq:uMinusuPluscondition} with $a=\theta/4$,
similarly to under \prettyref{eq:DominateSplitGuys}) and $l$ in
the place of $z$ to get that for any $l\in[-N,N]$, $z\in[-\frac{1}{2}\log|F|,\frac{1}{2}\log|F|]$
and $|F|\ge c(\theta)$ 

\begin{equation}
\begin{array}{ccl}
\prod_{i=1}^{k}Q_{0}(S_{i}-x_{i}\subset\mathcal{I}^{u_{F}(z-\frac{\theta}{4})})-\frac{\theta}{4} & \le & P_{q_{l}}(F\subset X(0,D_{[K_{N}u_{F}(z)]}))\\
 & \le & \prod_{i=1}^{k}Q_{0}(S_{i}-x_{i}\subset\mathcal{I}^{u_{F}(z+\frac{\theta}{4})})+\frac{\theta}{4},\end{array}\label{eq:RISandwhich}\end{equation}
where we have also used that $N\ge|F|^{8}$ and (when $k>1$) that
$\sum_{i\ne j}\mathcal{E}(S_{i},S_{j})\le\sum_{i\ne j}c\frac{|S_{i}||S_{j}|}{|F|^{3}}\le\frac{c}{|F|}$
by \prettyref{eq:BoundOnMutualEnergy} and $d(S_{i},S_{j})\ge c|F|^{3}$.
Consider the sets $\tilde{S}_{i}=(S_{i}-x_{i})+i\cdot\exp(N)e_{1}\subset\mathbb{Z}^{d+1}$,
where $S_{i}-x_{i}$ is viewed as a subset of $\mathbb{Z}^{d+1}$
and $e_{1}=(1,0,...,0)\in\mathbb{Z}^{d+1}$. Let $\tilde{F}=\bigcup_{i=1}^{k}\tilde{S}_{i}$.
By \prettyref{eq:Lemma21Belius2010} and \prettyref{eq:DefOfCapInZd}
the following holds for $z\in[-\frac{1}{2}\log|F|,\frac{1}{2}\log|F|]$:\[
|Q_{0}(\tilde{F}\subset\mathcal{I}^{u_{F}(z\pm\frac{\theta}{4})})-Q_{0}(\tilde{S}_{1}\subset\mathcal{I}^{u_{F}(z\pm\frac{\theta}{4})})Q_{0}(\tilde{F}\backslash\tilde{S}_{1}\subset\mathcal{I}^{u_{F}(z\pm\frac{\theta}{4})})|\le c\log|F|\frac{|\tilde{S}_{1}||\tilde{F}|}{\exp(N)}.\]
Applying \prettyref{eq:Lemma21Belius2010} another $k-1$ times and
using the triangle inequality we get that \begin{eqnarray}
\begin{array}{r}
\left|Q_{0}(\tilde{F}\subset\mathcal{I}^{u_{F}(z\pm\frac{\theta}{4})})-\prod_{i=1}^{k}Q_{0}(S_{i}-x_{i}\subset\mathcal{I}^{u_{F}(z\pm\frac{\theta}{4})})\right|\le c\log|F|\frac{|F|^{2}}{\exp(N)}\le\frac{\theta}{4}\end{array}\label{eq:PropA3}\end{eqnarray}
holds for all $|F|\ge c(\theta)$ and $z\in[-\frac{1}{2}\log|F|,\frac{1}{2}\log|F|]$
(using also that $Q_{0}(S_{i}-x_{i}\subset\mathcal{I}^{u_{F}(z\pm\frac{\theta}{4})})\overset{\prettyref{eq:InterlacementIsIncreasing}}{=}Q_{0}(\tilde{S}_{i}\subset\mathcal{I}^{u_{F}(z\pm\frac{\theta}{4})})$).
Now finally we apply \prettyref{eq:QuantCovLevResult-1}, using that
$|F|=|\tilde{F}|$, to get that for all $|F|\ge c(\theta)$ and $z\ge-\frac{1}{2}\log|F|$:\begin{equation}
\left|Q_{0}(\tilde{F}\subset\mathcal{I}^{u_{F}(z\pm\frac{\theta}{4})})-G(z\pm\frac{\theta}{4})\right|\le\frac{\theta}{4}.\label{eq:PropA4}\end{equation}
Combining \prettyref{eq:RISandwhich}, \prettyref{eq:PropA3} and
\prettyref{eq:PropA4} with the fact that $|G(z)-G(z\pm\frac{\theta}{4})|\le\frac{\theta}{4}$
for all $z$, we get \prettyref{eq:PropSmallSets}.
\end{proof}
Next we prove \prettyref{pro:PropLargeSets}, which deals with {}``big''
sets $F_{N}$. In this case we will consider for some $0<\rho<1$
the set $F_{N}^{\rho}=F_{N}\backslash X(0,D_{[(1-\rho)K_{N}u_{N}(0)]})$,
i.e. the subset of $F_{N}$ left uncovered by a fraction $1-\rho$
of the typical number of excursions $K_{N}u_{N}(0)$ needed to cover
$F_{N}$, and show that the cardinality of $F_{N}^{\rho}$ concentrates
around its typical value which turns out to be $|F_{N}|^{\rho}$ (this
is the content of \prettyref{lem:GoodEvent}). To do this we will
once again split $F_{N}$ into pieces $S_{1},...,S_{n}$ that are
contained in local boxes (but this time they will not in general be
far apart). Using the coupling \prettyref{thm:CouplingManyBoxes}
for one $i$ at a time will allow us to use a random interlacements
calculation to prove that $|F_{N}^{\rho}\cap S_{i}|$ concentrates
around $\frac{|S_{i}|}{|F_{N}|}\times|F_{N}|^{\rho}$. A union bound
will then ensure that $|F_{N}^{\rho}\cap S_{i}|$ concentrates around
this value for all $i$ at the same time with high probability, and
thus that $|F_{N}^{\rho}|$ concentrates around $|F_{N}|^{\rho}$
with high probability.

Now for the $\rho$ we pick (cf. \prettyref{eq:DefOfRho}) we will
have $|F_{N}|^{\rho}\le((2N+1)N^{d})^{\rho}\le N^{1/8}$ so that (with
high probability) $F_{N}^{\rho}$ is a {}``small set'' in the sense
of \prettyref{pro:PropSmallSets}. It will turn out that the excursions
after the $[(1-\rho)K_{N}u_{N}(0)]$-th departure are {}``almost''
independent of $F_{N}^{\rho}$ and therefore we will be able to apply
\prettyref{pro:PropSmallSets} to the set $F_{N}^{\rho}$ to prove
that the number of additional excursions needed to cover it is approximately
a Gumbel random variable with location $K_{N}g(0)\log|F_{N}^{\rho}|\approx K_{N}g(0)\log|F_{N}|^{\rho}=K_{N}g(0)\rho\log|F_{N}|$
and scale $K_{N}g(0)$. Adding to this random variable $(1-\rho)K_{N}u_{N}(0)=(1-\rho)K_{N}\log|F_{N}|$,
the approximate number of excursions that reduced $F_{N}$ to $F_{N}^{\rho}$,
we find that the number of excursions needed to cover $F_{N}$ is
approximately distributed as a Gumbel random variable with location
$K_{N}g(0)\log|F_{N}|$ and scale $K_{N}g(0)$, which is essentially
speaking what \prettyref{pro:PropLargeSets} claims.
\begin{proof}[Proof of \prettyref{pro:PropLargeSets}]
Fix a $z\in\mathbb{R}$. Define for any $\rho\in(0,\frac{1}{4}]$
and $N\ge c$\begin{equation}
r(\rho)=r(\rho,N)=[(1-\rho)K_{N}u_{N}(0)]\mbox{ and }F_{N}^{\rho}=F_{N}\backslash X(0,D_{r(\rho)}).\label{eq:DefinitionofR}\end{equation}
Because $\{F_{N}\subset X(0,D_{[K_{N}u_{N}(z)]})\}=\{F_{N}^{\rho}\subset X(R_{r(\rho)+1},D_{[K_{N}u_{N}(z)]})\}$
for $N\ge c$ it suffices to show that for some $\rho\in(0,\frac{1}{4}]$
\begin{equation}
P(F_{N}^{\rho}\subset X(R_{r(\rho)+1},D_{[K_{N}u_{N}(z)]}))\rightarrow G(z)\mbox{ as }N\rightarrow\infty.\label{eq:GumbelForFRhoN}\end{equation}
For each $\lambda\in(0,\frac{1}{100})$ and $N\ge1$ we define the
collection of {}``good sets'' by:\begin{equation}
\mathcal{G}_{N,\lambda}=\{F'\subset F_{N}:(1-\lambda)|F_{N}|^{\rho}\le|F'|\le(1+\lambda)|F_{N}|^{\rho}\}.\label{eq:DefOfGoodSets}\end{equation}
To show \prettyref{eq:GumbelForFRhoN} we will use the following lemma:
\begin{lem}
\label{lem:GoodEvent}For all $\lambda\in(0,\frac{1}{100})$ we have
$\lim_{N\rightarrow\infty}P\left(F_{N}^{\rho}\in G_{N,\lambda}\right)=1.$\end{lem}
\begin{proof}
Let \begin{equation}
\varepsilon=\frac{1}{3(d+1)}\frac{1}{16}\rho.\label{eq:DefofEpsilon}\end{equation}
For each $N$ we can select $x_{1},..,x_{n(N)}\in\mathbb{T}_{N}\times[-\frac{N}{2},\frac{N}{2}],n(N)\le cN^{(d+1)\varepsilon}$
and partition $F_{N}$ into disjoint non-empty sets $S_{N}^{1},...,S_{N}^{n}$
such that $S_{N}^{i}\subset B(x_{i},N^{1-\varepsilon})$ for each
$i$ and $F_{N}=\bigcup_{i=1}^{n}S_{N}^{i}$. Consider for each $i=1,...,n(N)$
the events\begin{equation}
\begin{array}{ccl}
E_{i}^{+} & = & \left\{ \left|S_{N}^{i}\backslash X(0,D_{r(\rho)})\right|\ge(1+\frac{\lambda}{2})a_{i}|F_{N}|^{\rho}+\frac{\lambda}{2t}\sqrt{a_{i}}|F_{N}|^{\rho}\right\} \mbox{ and }\\
E_{i}^{-} & = & \left\{ \left|S_{N}^{i}\backslash X(0,D_{r(\rho)})\right|\le(1-\frac{\lambda}{2})a_{i}|F_{N}|^{\rho}-\frac{\lambda}{2t}\sqrt{a_{i}}|F_{N}|^{\rho}\right\} ,\end{array}\label{eq:DefofEpm}\end{equation}
where $a_{i}=\frac{|S_{N}^{i}|}{|F_{N}|}$ and $t=\sum_{i=1}^{n(N)}\sqrt{a_{i}}$.
Because the right-hand sides of the inequalities in the events in
\eqref{eq:DefofEpm} sum up to $(1\pm\lambda)|F_{N}|^{\rho}$ we have:\begin{equation}
\begin{array}{ccccl}
P\left(\left|F_{N}^{\rho}\right|\ge(1+\lambda)|F_{N}|^{\rho}\right) & \le & \sum P(E_{i}^{+}) & \le & cN^{(d+1)\varepsilon}\sup_{i}P(E_{i}^{+}),\\
P\left(\left|F_{N}^{\rho}\right|\le(1-\lambda)|F_{N}|^{\rho}\right) & \le & \sum P(E_{i}^{-}) & \le & cN^{(d+1)\varepsilon}\sup_{i}P(E_{i}^{-}).\end{array}\label{eq:UnionBound}\end{equation}
We therefore wish to bound $P(E_{i}^{\pm})$. To this end let\begin{eqnarray}
u_{\pm}(\lambda)=g(0)\{(1-\rho)\log|F_{N}|+\log\frac{1}{1\mp\lambda/2}\}\mbox{ and }u=(1-\rho)u_{N}(0),\label{eq:Defofus-1}\end{eqnarray}
and note that we can apply \prettyref{thm:CouplingManyBoxes} with
$k=1$, $\varepsilon$ as in \prettyref{eq:DefofEpsilon}, $u$ as
in \prettyref{eq:Defofus-1}, $\delta$ as in \prettyref{eq:uMinusuPluscondition}
(using that $u_{-}(\lambda)\le(1-\rho)u_{N}(-\frac{\lambda}{4})\le(1-\rho)u_{N}(\frac{\lambda}{4})\le u_{+}(\lambda)$,
$u=(1-\rho)u_{N}(0)$, so that $[uK_{N}]=r(\rho)$, and \prettyref{eq:uMinusuPluscondition}
with $a=\lambda/4$ similarly to under \prettyref{eq:DominateSplitGuys}
and above \prettyref{eq:RISandwhich}) once for each $i$ to show
that if $N\ge c(\lambda)$ then\begin{equation}
\begin{array}{ccc}
P(E_{i}^{+})\le Q_{0}(|S_{N}^{i}\backslash\mathcal{I}^{u_{-}}|\ge(1+\frac{\lambda}{2})a_{i}|F_{N}|^{\rho}+\frac{\lambda}{2t}\sqrt{a_{i}}|F_{N}|^{\rho})+cN^{-3d},\\
P(E_{i}^{-})\le Q_{0}(|S_{N}^{i}\backslash\mathcal{I}^{u_{+}}|\le(1-\frac{\lambda}{2})a_{i}|F_{N}|^{\rho}-\frac{\lambda}{2t}\sqrt{a_{i}}|F_{N}|^{\rho})+cN^{-3d},\end{array}\label{eq:BoundEbyQ}\end{equation}
where we view $S_{N}^{i}$ as a subset of $\mathbb{Z}^{d+1}$ and
have also used \prettyref{eq:InterlacementIsIncreasing}. Note that
\begin{equation}
\mathbb{E}^{Q_{0}}(|S_{N}^{i}\backslash\mathcal{I}^{u_{\mp}}|)=\sum_{x\in S_{N}^{i}}Q_{0}(x\notin\mathcal{I}^{u_{\mp}})\overset{\prettyref{eq:RIOneandTwoPoint},\eqref{eq:Defofus-1}}{=}(1\pm\frac{\lambda}{2})a_{i}|F_{N}|^{\rho}.\label{eq:ExpectationComputation}\end{equation}
Thus the probabilities on the right hand sides of \eqref{eq:BoundEbyQ}
are bounded above by:\begin{equation}
\begin{array}{ccc}
Q_{0}\left(\left||S_{N}^{i}\backslash\mathcal{I}^{u_{\mp}}|-\mathbb{E}^{Q_{0}}(|S_{N}^{i}\backslash\mathcal{I}^{u_{\mp}}|)\right|\ge\frac{\lambda}{2t}\sqrt{a_{i}}|F_{N}|^{\rho}\right). & \mbox{}\end{array}\label{eq:PreChebyshev}\end{equation}
Using the Chebyshev inequality we see that \eqref{eq:PreChebyshev}
is bounded above by:\begin{equation}
\frac{\mathbb{E}^{Q_{0}}(|S_{N}^{i}\backslash\mathcal{I}^{u_{\mp}}|^{2})-\left(\mathbb{E}^{Q_{0}}(|S_{N}^{i}\backslash\mathcal{I}^{u_{\mp}}|)\right)^{2}}{\frac{\lambda^{2}}{4t^{2}}a_{i}|F_{N}|^{2\rho}}\label{eq:PostChebyshev}\end{equation}
We thus wish to bound $\mathbb{E}^{Q_{0}}(|S_{N}^{i}\backslash\mathcal{I}^{u_{\mp}}|^{2})$.
Note that\begin{eqnarray}
\mathbb{E}^{Q_{0}}(|S_{N}^{i}\backslash\mathcal{I}^{u^{\mp}}|^{2}) & = & \sum_{x,y\in S_{N}^{i}}Q_{0}(x,y\notin\mathcal{I}^{u^{\mp}})\overset{\prettyref{eq:ExpectationComputation}}{\le}ca_{i}|F_{N}|^{\rho}\nonumber \\
 &  & +\sum_{x\in S_{N}^{i}}\sum_{y\in S_{N}^{i},y\ne x}Q_{0}(x,y\notin\mathcal{I}^{u_{\mp}}).\label{eq:SecondMomentasSumofTwoPoint}\end{eqnarray}
Now using the translation invariance of $\mathcal{I}^{u^{\mp}}$ (cf.
\prettyref{eq:InterlacementIsIncreasing}) and \prettyref{lem:Calc}
(with $a=10000N$ say) we get that for $N\ge c$: \begin{eqnarray}
\begin{array}{rc}
\sum_{x\in S_{N}^{i},y\ne x}Q_{0}(x,y\notin\mathcal{I}^{u^{\mp}}) & \overset{\eqref{eq:Defofus-1}}{\le}\\
(1\pm\frac{\lambda}{2})^{2}|F_{N}|^{2\rho-2}\{|S_{N}^{i}|^{2}+c\log|F_{N}||S_{N}^{i}|^{\frac{2}{d+1}+1}\}+c|S_{N}^{i}||F_{N}|^{\rho-1} & \overset{\prettyref{eq:ExpectationComputation}}{\le}\\
\left(\mathbb{E}^{Q_{0}}(|S_{N}^{i}\backslash\mathcal{I}^{u^{\mp}}|)\right)^{2}+ca_{i}\log|F_{N}||F_{N}|^{2\rho-\frac{d-1}{d+1}}+ca_{i}|F_{N}|^{\rho} & \overset{\rho\le\frac{1}{4}}{\le}\\
\left(\mathbb{E}^{Q_{0}}(|S_{N}^{i}\backslash\mathcal{I}^{u^{\mp}}|)\right)^{2}+ca_{i}|F_{N}|^{\rho}\end{array}\label{eq:BoundonSumofTwoPoint}\end{eqnarray}
Combining \prettyref{eq:BoundonSumofTwoPoint} with \eqref{eq:PreChebyshev},
\eqref{eq:PostChebyshev} and \eqref{eq:SecondMomentasSumofTwoPoint}
we find from \eqref{eq:BoundEbyQ} that $P(E_{i}^{\pm})\le c(\lambda)t^{2}|F_{N}|^{-\rho}+cN^{-3d}\overset{t\ge1}{\le}c(\lambda)t^{2}N^{-\frac{1}{8}\rho}$
and thus from \prettyref{eq:UnionBound} that\[
P(\left\{ (1-\lambda)|F_{N}|^{\rho}\le\left|F_{N}^{\rho}\right|\le(1+\lambda)|F_{N}|^{\rho}\right\} ^{c})\le c(\lambda)N^{(d+1)\varepsilon}t^{2}N^{-\frac{1}{8}\rho}\le c(\lambda)N^{-\frac{1}{16}\rho},\]
where to get the last inequality we have used that $t=\sum_{i=1}^{n(N)}\sqrt{a_{i}}\le n(N)\le cN^{(d+1)\varepsilon}$
and \prettyref{eq:DefofEpsilon}. Thus we just have to let $N\rightarrow\infty$
and recall the definition \prettyref{eq:DefOfGoodSets} of $\mathcal{G}_{N,\lambda}$
to complete the proof of \prettyref{lem:GoodEvent}.
\end{proof}
We now continue with the proof of \prettyref{eq:GumbelForFRhoN}.
Fix \begin{eqnarray}
\rho & = & \frac{1}{16(d+1)}\in(0,\frac{1}{4}]\label{eq:DefOfRho}\end{eqnarray}
and write $r$ in place of $r(\rho)$. Note that by the strong Markov
property\begin{equation}
\begin{array}{rcc}
P(F_{N}^{\rho}\subset X(R_{r+1},D_{[K_{N}u_{N}(z)]}),F_{N}^{\rho}\in\mathcal{G}_{N,\lambda}) & =\\
\sum P(F_{N}^{\rho}=F',X_{D_{r}}=x)P_{x}(F'\subset X(0,D_{[K_{N}u_{N}(z)]-r})).\end{array}\label{eq:UseOfMarkov}\end{equation}
where the sum is over all $x\in\partial_{e}\tilde{B}$ and $F'\in\mathcal{G}_{N,\lambda}$.
We now need the following lemma:
\begin{lem}
\label{lem:F'CloseToGumbel}For any $\lambda\in(0,\frac{1}{100})$
and $N\ge c(\lambda,z)$ we have that if $F'\in\mathcal{G}_{N,\lambda}$
and $x\in\partial_{e}\tilde{B}$ then\begin{equation}
|P_{x}(F'\subset X(0,D_{[K_{N}u_{N}(z)]-r}))-G(z)|\le c\lambda.\label{eq:F'CloseToGumbel}\end{equation}
\end{lem}
\begin{proof}
By the strong Markov property we have\begin{equation}
P_{x}(F'\subset X(0,D_{[K_{N}u_{N}(z)]-r}))=E_{x}[P_{X_{R_{1}}}(F'\subset X(0,D_{[K_{N}u_{N}(z)]-r}))].\label{eq:SugaBuga}\end{equation}
Let $x=(y,w)$, with $y\in\mathbb{T}_{N}$ and $w=\{-h_{N},h_{N}\}$,
and let $v\in\{-r_{N},r_{N}\}$ with $wv>0$. Then by (2.2) of \cite{Sznitman2009-OnDOMofRWonDCbyRI}
we have\[
\sup_{x'\in\mathbb{T}_{N}\times\{v\}.}|P_{x}(X_{R_{1}}=x')-q_{v}(x')|\le cN^{-5d}.\]
Thus for $N\ge c(\lambda)$ we have \begin{equation}
|E_{x}[P_{X_{R_{1}}}(F'\subset X(0,D_{[K_{N}u_{N}(z)]-r}))]-P_{q_{v}}(F'\subset X(0,D_{[K_{N}u_{N}(z)]-r}))|\le\lambda.\label{eq:TheFirstTheta}\end{equation}
Note that $g(0)\{\rho\log|F_{N}|+z\}=g(0)\{\log|F'|+z+\log\frac{|F_{N}|^{\rho}}{|F'|}\}$
and since $F'\in\mathcal{G}_{N,\lambda}$ we have $-4\lambda\le\log\frac{1}{1+\lambda}\le\log\frac{|F_{N}|^{\rho}}{|F'|}\le\log\frac{1}{1-\lambda}\le4\lambda$.
Thus for all $N\ge c(z,\lambda)$ it holds that \[
K_{N}u_{F'}(z-8\lambda)\overset{\prettyref{eq:DefinitionofR}}{\le}[K_{N}u_{N}(z)]-r\overset{\prettyref{eq:DefinitionofR}}{\le}K_{N}u_{F'}(z+8\lambda),\]
and therefore also that \begin{eqnarray}
P_{q_{v}}(F'\subset X(0,D_{[K_{N}u_{F'}(z-8\lambda)]})) & \le & P_{q_{v}}(F'\subset X(0,D_{[K_{N}u_{N}(z)]-r}))\label{eq:asdasdsda-1}\\
 & \le & P_{q_{v}}(F'\subset X(0,D_{[K_{N}u_{F'}(z+8\lambda)]})).\nonumber \end{eqnarray}
Now if $N\ge c(\lambda,z)$ then $|F'|\overset{\eqref{eq:DefOfGoodSets},\eqref{eq:DefOfRho}}{\le}cN^{\frac{1}{16}}\le N^{\frac{1}{8}}$,
$|F'|\overset{\eqref{eq:DefOfGoodSets}}{\ge}cN^{\frac{1}{8}\rho}\ge c(\lambda)$
and $-\frac{1}{2}\log|F'|\le z-8\lambda\le z+8\lambda\le\frac{1}{2}\log|F^{'}|$.
We can therefore use \prettyref{pro:PropSmallSets} with $\lambda$
in the place of $\theta$ on the right- and left-hand sides of \prettyref{eq:asdasdsda-1}
to get that \begin{equation}
G(z-8\lambda)-\lambda\le P_{q_{v}}(F'\subset X(0,D_{[K_{N}u_{N}(z)]-r})\le G(z+8\lambda)+\lambda.\label{eq:ohlalalala-1}\end{equation}
Now we simply have to combine \prettyref{eq:SugaBuga}, \prettyref{eq:TheFirstTheta}
and \prettyref{eq:ohlalalala-1} with the inequality $|G(z\pm8\lambda)-G(z)|\le c\lambda$
to get \prettyref{eq:F'CloseToGumbel}.
\end{proof}
We are now ready to complete the proof of \prettyref{pro:PropLargeSets}.
From \prettyref{eq:F'CloseToGumbel} and \prettyref{eq:UseOfMarkov}
we see that if $\lambda\in(0,\frac{1}{100})$ and $N\ge c(\lambda,z)$
then \[
|P(F_{N}^{\rho}\subset X(R_{r+1},D_{[K_{N}u_{N}(z)]}),F_{N}^{\rho}\in\mathcal{G}_{N,\lambda})-P(F_{N}^{\rho}\in\mathcal{G}_{N,\lambda})G(z)|\le c\lambda.\]
Letting $N\rightarrow\infty$ and using \prettyref{lem:GoodEvent}
we see that\[
|\lim_{N\rightarrow\infty}P(F_{N}^{\rho}\subset X(R_{r+1},D_{[K_{N}u_{N}(z)]}))-G(z)|\le c\lambda.\]
Now letting $\lambda\downarrow0$ we get \prettyref{eq:GumbelForFRhoN}
and therefore the proof of \prettyref{pro:PropLargeSets} is complete.

\end{proof}
We have now completely reduced the proofs \prettyref{thm:GumbelForLocTime}
and its corollaries to the coupling \prettyref{thm:CouplingManyBoxes}.

\section{\label{sec:CoupleRWwithRI}Coupling random walk with random interlacements}

In this section we state and prove the main coupling theorem, \prettyref{thm:CouplingManyBoxes},
which couples random walk $X_{\cdot}$ with random interlacements.
More precisely, for some $\varepsilon\in(0,1),$ $k\ge1$ and suitably
large $N$ we select $k$ vertices\begin{equation}
\begin{array}{c}
x_{1},...,x_{k}\in\mathbb{T}_{N}\times[-\frac{N}{2},\frac{N}{2}]\mbox{ and }k\mbox{ non-empty sets }\\
S_{1},S_{2},...,S_{k}\subset B(0,N^{1-\varepsilon})\mbox{ such that }S_{i}+x_{i},i=1,...,k\mbox{ are disjoint,}\end{array}\label{eq:CouplingTheBoxes}\end{equation}
and then construct, for appropriate $u$ and $\delta$, $k$ independent
pairs of random sets $\mathcal{I}_{i}^{u(1-\delta)}\cap S_{i},\mathcal{I}_{i}^{u(1+\delta)}\cap S_{i}$,
with the law of random interlacements at level $u(1-\delta)$, respectively
$u(1+\delta)$, intersected with $S_{i}$, such that the following
event holds with high probability (provided the $S_{i}+x_{i}$ have
low mutual energy, see \prettyref{eq:DefOfMutualEnergy}, for example
if they are {}``far apart''):\begin{equation}
\{\mathcal{I}_{i}^{u(1-\delta)}\cap S_{i}\subset(X(0,D_{[uK_{N}]})-x_{i})\cap S_{i}\subset\mathcal{I}_{i}^{u(1+\delta)}\cap S_{i}\mbox{ for all }i\}.\label{eq:CouplingManyBoxesEvent}\end{equation}

A weaker version of the coupling which gave the upper inclusion for
$k=1$, \emph{fixed} $u$ and $\delta$ and large $N$ is contained
in \cite{Sznitman2009-OnDOMofRWonDCbyRI} (a similar lower inclusion
is contained implicitly in \cite{Sznitman2009-UBonDTofDCandRI}).
\prettyref{thm:CouplingManyBoxes} improves on this by allowing $u$
and $\delta$ to vary with $N$, by constructing $\mathcal{I}_{i}^{u(1\pm\delta)}\cap S_{i}$
such that they have the \emph{joint} law of random interlacements
at levels $u(1\pm\delta)$ intersected with $S_{i}$ (the naive way
of combining the explicit coupling from \cite{Sznitman2009-OnDOMofRWonDCbyRI}
with the implicit coupling from \cite{Sznitman2009-UBonDTofDCandRI}
to get a double inclusion, as in \prettyref{eq:CouplingManyBoxesEvent},
does not guarantee the {}``correct'' joint law), and by coupling
the trace in several sets $S_{1}+x_{1},...,S_{k}+x_{k}$. For more
on why constructing $\mathcal{I}_{i}^{u(1\pm\delta)}\cap S_{i}$ such
that they have the correct joint law is useful see \prettyref{rem:EndRemark}
(2).

The proof of \prettyref{thm:CouplingManyBoxes} is divided into three
steps: The first is to construct two independent Poisson processes
of excursions (that is point processes on the space $\mathcal{T}_{\tilde{B}}$
of intensity proportional to $\nu$, see \prettyref{eq:PPoEIntensity})
$\mu_{1}$ and $\mu_{2}$, such that with high probability $(\mathcal{I}(\mu_{1})-x_{i})\cap S_{i}\subset(X(0,D_{[uK_{N}]})-x_{i})\cap S_{i}\subset(\mathcal{I}(\mu_{1})\cup\mathcal{I}(\mu_{2})-x_{i})\cap S_{i}$
for all $i$. This is done in \prettyref{sec:CoupleRWandPPoE}; in
the proof of \prettyref{thm:CouplingManyBoxes} we invoke \prettyref{cor:CoupleExcwithPPP}
for this step. 

The second step is to exploit the fact that {}``the traces of a Poisson
process of excursions on sets of low mutual energy are approximately
independent'' (if $k>1$, otherwise this step is trivial) to construct
$k$ \emph{independent} pairs of Poisson processes of excursions $\mu_{1}^{i},\mu_{2}^{i},i=1,...,k$
such that with high probability $(\mathcal{I}(\mu_{1}^{i})-x_{i})\cap S_{i}=(\mathcal{I}(\mu_{1})-x_{i})\cap S_{i}$
and $(\mathcal{I}(\mu_{1}^{i})\cup\mathcal{I}(\mu_{2}^{i})-x_{i})\cap S_{i}=(\mathcal{I}(\mu_{1})\cup\mathcal{I}(\mu_{2})-x_{i})\cap S_{i}$
for $i=1,...,k$. This is done in \prettyref{lem:IndepLemmaForPPoE}.

The third step is to construct from $\mu_{1}^{i},\mu_{2}^{i}$ the
random sets $\mathcal{I}_{i}^{u(1-\delta)}\cap S_{i},\mathcal{I}_{i}^{u(1+\delta)}\cap S_{i}$
from \prettyref{eq:CouplingManyBoxesEvent} such that with high probability
$\mathcal{I}_{i}^{u(1-\delta)}\cap S_{i}\subset(\mathcal{I}(\mu_{1}^{i})-x_{i})\cap S_{i}$
and $(\mathcal{I}(\mu_{1}^{i})\cup\mathcal{I}(\mu_{2}^{i})-x_{i})\cap S_{i}\subset\mathcal{I}_{i}^{u(1+\delta)}\cap S_{i}$.
This is done mainly in \prettyref{sec:CouplePPoEwithRI}; in the proof
of \prettyref{thm:CouplingManyBoxes} we invoke \prettyref{pro:PPoEandRIcoupling}
for this step.

We thus postpone a large part of the work to Sections 5 and 6, and
here only prove \prettyref{thm:CouplingManyBoxes} conditionally on
the results of these two sections. We now state the theorem.
\begin{thm}
\label{thm:CouplingManyBoxes}($d\ge2$) For any $k\ge1$, $\varepsilon\in(0,1)$,
$N\ge c(\varepsilon)$, $z\in[-N,N]$, $1\ge\delta\ge c_{2}\frac{r_{N}}{h_{N}}$,
\textup{$u$ satisfying $uK_{N}\ge(\log N)^{6}$,} and $x_{1},...,x_{k},S_{1},...,S_{k}$
as in \eqref{eq:CouplingTheBoxes} we can construct on a space $(\Omega_{1},\mathcal{A}_{1},Q_{1})$
an $E_{N}-$valued random walk $X_{\cdot}$ with law $P_{q_{z}}$,
and an independent collection $((\mathcal{I}_{i}^{u(1-\delta)}\cap S_{i},\mathcal{I}_{i}^{u(1+\delta)}\cap S_{i}))_{i=1}^{k}$
such that the $i$-th member of the collection has the (joint) law
of $(\mathcal{I}^{u(1-\delta)}\cap S_{i},\mathcal{I}^{u(1+\delta)}\cap S_{i})$
under $Q_{0}$ and \begin{eqnarray}
\begin{array}{rc}
Q_{1}(F^{c})\le\begin{cases}
cuN^{-3d-1} & \mbox{if }k=1,\\
cu\sum_{i\ne j}\mathcal{E}(S_{i}+x_{i},S_{j}+x_{j}) & \mbox{if }k>1,\end{cases}\end{array}\label{eq:CouplingManyBoxes}\end{eqnarray}
where $F$ is the event from \prettyref{eq:CouplingManyBoxesEvent}.
(Recall the definition of $\mathcal{E}$ from \prettyref{eq:DefOfMutualEnergy}
and the bound on it from \prettyref{eq:BoundOnMutualEnergy}.)\end{thm}
\begin{proof}
Let $c_{2}=2c_{4}$ where $c_{4}$ is the constant from \prettyref{cor:CoupleExcwithPPP}.
We first apply \prettyref{cor:CoupleExcwithPPP} with $\frac{\delta}{2}$
in place of $\delta$ (note that $1>\frac{\delta}{2}\ge c_{4}\frac{r_{N}}{h_{N}}$)
to construct on a space $(\Omega_{3},\mathcal{A}_{3},Q_{3})$ a coupling
of $X_{\cdot}$ with law $P_{q_{z}}$ and two Poisson point processes
$\mu_{1},\mu_{2}$ on $\mathcal{T}_{\tilde{B}}$ with intensities
$u(1-\frac{\delta}{2})\nu,\delta u\nu$ respectively such that\begin{equation}
\begin{array}{rc}
Q_{3}(\forall i,\mathcal{I}(\mu_{1})\cap(S_{i}+x_{i})\subset X(0,D_{[uK_{N}]})\cap(S_{i}+x_{i})\subset(\mathcal{I}(\mu_{2})\cup\mathcal{I}(\mu_{2}))\cap(S_{i}+x_{i})) & \ge\\
1-cuN^{-3d-1}.\end{array}\label{eq:first}\end{equation}
For the case $k>1$ we will need the following lemma:
\begin{lem}
\label{lem:IndepLemmaForPPoE}($N\ge c(\varepsilon)$) Let $\mu$
be a Poisson process on $\mathcal{T}_{\tilde{B}}$ of intensity $s\nu,s\ge0$.
We can then define (by extending the space) an iid collection $\eta_{1},...,\eta_{k}$
of Poisson processes such that $\eta_{i}\overset{\mbox{law}}{=}\mu$
and \begin{equation}
Q_{1}(\exists i,\mathcal{I}(\eta_{i})\cap(S_{i}+x_{i})\ne\mathcal{I}(\mu)\cap(S_{i}+x_{i}))\le2s\sum_{i,j:j\ne i}\mathcal{E}(S_{i}+x_{i},S_{j}+x_{j}).\label{eq:SplitthePPoE}\end{equation}
\end{lem}
\begin{proof}
For all $j,i$ let $\mathcal{A}_{j,i}\subset\mathcal{T}_{\tilde{B}}$
denote the set $\{H_{S_{j}+x_{j}}<H_{S_{i}+x_{i}}<T_{\tilde{B}}\}$,
and for all $i$ let $\mathcal{B}_{i}\subset\mathcal{T}_{\tilde{B}}$
denote $\cup_{j:j\ne i}\mathcal{A}_{j,i}$ and let $\mathcal{C}_{i}\subset\mathcal{T}_{\tilde{B}}$
denote $\{H_{S_{i}+x_{i}}<T_{\tilde{B}}\}$. For each $i$ make the
decomposition $\mu=\phi_{i}+\psi_{i}$ where $\phi_{i}=1_{\mathcal{C}_{i}\cap\mathcal{B}_{i}^{c}}\mu$
and $\psi_{i}=1_{\mathcal{C}_{i}^{c}\cup\mathcal{B}_{i}}\mu$. Because
$\mathcal{C}_{i}\cap\mathcal{B}_{i}^{c}$ ({}``the excursion reaches
$S_{i}+x_{i}$ first'') are disjoint the $\phi_{i},i=1,...,k$ are
mutually independent Poisson point processes. Now extend the space
by adding an independent collection of point processes $\psi_{i}^{'},i=1,...,k$
such that $\psi_{i}^{'}\overset{\mbox{law}}{=}\psi_{i}$ and define
$\eta_{i}=\psi_{i}^{'}+\phi_{i}$. Then the $\eta_{i}$ are independent
and $\eta_{i}\overset{\mbox{law}}{=}\mu$, so to complete the proof
of the lemma it suffices to show \prettyref{eq:SplitthePPoE}. Note
that for each $i$:\begin{equation}
\begin{array}{ccl}
Q_{1}(\mathcal{I}(\eta_{i})\cap(S_{i}+x_{i})\ne\mathcal{I}(\mu)\cap(S_{i}+x_{i})) & \le & Q_{1}(\psi_{i}(\mathcal{C}{}_{i})\ne0\mbox{ or }\psi_{i}^{'}(\mathcal{C}{}_{i})\ne0)\\
 & \overset{\prettyref{eq:PPoEIntensity}}{\le} & 2sK_{N}P_{q}(\mathcal{\mathcal{B}}{}_{i})\\
 & \le & 2s\underset{j:j\ne i}{\sum}K_{N}P_{q}(\mathcal{A}_{j,i}).\end{array}\label{eq:BoundOnBeingDifferent}\end{equation}
If we let $K=(S_{j}+x_{j})\cup(S_{i}+x_{i})$ then (provided $N\ge c(\varepsilon)$
so that $K\subset\mathbb{T}_{N}\times(-r_{N},r_{N})$)\begin{equation}
\begin{array}{ccl}
K_{N}P_{q}(A_{j,i}) & = & \underset{x\in S_{j}+x_{j}}{\sum}K_{N}P_{q}(H_{K}<T_{\tilde{B}},X_{H_{K}}=x)P_{x}(H_{S_{i}+x_{i}}<T_{\tilde{B}})\\
 & \overset{\prettyref{eq:GeometricLemma},\prettyref{eq:HittingToGreenEquil}}{=} & \underset{x\in S_{j}+x_{j},y\in S_{i}+x_{i}}{\sum}e_{K,\tilde{B}}(x)g_{\tilde{B}}(x,y)e_{S_{i}+x_{i},\tilde{B}}(y)\\
 & \overset{\prettyref{eq:DefOfCapInCylinder},\prettyref{eq:DefOfMutualEnergy}}{\le} & \mathcal{E}(S_{j}+x_{j},S_{i}+x_{i}).\end{array}\label{eq:BoundOnHittingTwoSis}\end{equation}
 Now combining \prettyref{eq:BoundOnBeingDifferent} and \prettyref{eq:BoundOnHittingTwoSis}
we get \prettyref{eq:SplitthePPoE}.
\end{proof}
We now continue with the proof of \prettyref{thm:CouplingManyBoxes}.
If $k>1$ we apply \prettyref{lem:IndepLemmaForPPoE} once for $\mu_{1}$
and once for $\mu_{2}$ and extend our space with independent $\mu_{n}^{i},n=1,2,i=1,...,k$,
such that $\mu_{1}^{i}\overset{\mbox{law}}{=}\mu_{1}$ and $\mu_{2}^{i}\overset{\mbox{law}}{=}\mu_{2}$
for all $i$ and \begin{eqnarray}
\begin{array}{c}
Q_{1}(F'^{c})\le cu\sum_{i\ne j}\mathcal{E}(S_{i}+x_{i},S_{j}+x_{j})\mbox{ where }F'\mbox{ is the event}\\
\{(\mathcal{I}(\mu_{n}^{i})-x_{i})\cap S_{i}=(\mathcal{I}(\mu_{n})-x_{i})\cap S_{i}\mbox{ for all }n=1,2,i=1,...,k\},\end{array}\label{eq:second}\end{eqnarray}
If $k=1$ then simply define $\mu_{1}^{1}=\mu_{1}$ and $\mu_{2}^{1}=\mu_{2}$.
When $N\ge c(\varepsilon)$ we now apply \prettyref{pro:PPoEandRIcoupling}
with $x=x_{i}$, $\mu=\mu_{1}^{i}$, $u(1-\delta/2)$ in place of
$u$, $u_{-}=u(1-\delta)$ and $u_{+}=u(1-\frac{1}{4}\delta)$ (note
that $\frac{u(1-\delta/2)}{u_{-}}=1+\frac{\delta/2}{1-\delta}\ge N^{-c_{5}}$
and $\frac{u_{+}}{u(1-\delta/2)}=1+\frac{\delta/4}{1-\delta/2}\ge N^{-c_{5}(\varepsilon)}$
for $N\ge c(\varepsilon$)) once for every $1\le i\le k$, each time
extending our space by adding a pair of independent random sets $\mathcal{I}_{1,i}$,$\mathcal{I}_{2,i}$
depending only on $\mu_{1}^{i}$ with the distribution under $Q_{0}$
of $\mathcal{I}^{u(1-\delta)}\cap B(0,N^{1-\varepsilon})$ and $\mathcal{I}^{u\frac{3}{4}\delta}\cap B(0,N^{1-\varepsilon})$
respectively, such that for each $i$\begin{eqnarray}
\begin{array}{rc}
Q_{1}(\forall i,\mathcal{I}_{1,i}\subset(\mathcal{I}(\mu_{1}^{i})-x_{i})\cap B(0,N^{1-\varepsilon})\subset\mathcal{I}_{1,i}\cup\mathcal{I}_{2,i})\ge1-cukN^{-10(d+1)}.\end{array}\label{eq:third}\end{eqnarray}
We then apply \prettyref{pro:PPoEandRIcoupling} again, this time
with $\mu=\mu_{2}^{i}$, $u\delta$ in place of $u$, $u_{-}=0$ and
$u_{+}=u\frac{5\delta}{4}$ (so that $\frac{u_{+}}{u}=\frac{5}{4}\ge N^{-c_{5}(\varepsilon)}$
for $N\ge c(\varepsilon)$) once for every $1\le i\le k$, each time
extending our space by adding a random set $\mathcal{I}_{3,i}$ depending
only on $\mu_{2}^{i}$ and distributed as $\mathcal{I}^{u\frac{5\delta}{4}}\cap B(0,N^{1-\varepsilon})$
under $Q_{0}$ such that\begin{equation}
Q_{1}(\forall i,(\mathcal{I}(\mu_{2}^{i})-x_{i})\cap B(0,N^{1-\varepsilon})\subset\mathcal{I}_{3,i})\ge1-cukN^{-10(d+1)}.\label{eq:fourth}\end{equation}
We now define $\mathcal{I}_{i}^{u(1-\delta)}\cap S_{i}$ and $\mathcal{I}_{i}^{u(1+\delta)}\cap S_{i}$
by \[
\mathcal{I}_{i}^{u(1-\delta)}\cap S_{i}=\mathcal{I}_{1,i}\cap S_{i}\mbox{ and }\mathcal{I}_{i}^{u(1+\delta)}\cap S_{i}=(\mathcal{I}_{1,i}\cup\mathcal{I}_{2,i}\cup\mathcal{I}_{3,i})\cap S_{i},1\le i\le k.\]
Since $\mathcal{I}_{1,i}$, $\mathcal{I}_{2,i}$ and $\mathcal{I}_{3,i}$
are independent, we get from \prettyref{eq:markovpropforI} that $(\mathcal{I}_{i}^{u(1-\delta)}\cap S_{i},\mathcal{I}_{i}^{u(1+\delta)}\cap S_{i})$
has the law of $(\mathcal{I}^{u(1-\delta)}\cap S_{i},\mathcal{I}^{u(1+\delta)}\cap S_{i})$
under $Q_{0}$. Also the collection $((\mathcal{I}_{i}^{u(1-\delta)}\cap S_{i},\mathcal{I}_{i}^{u(1+\delta)}\cap S_{i}))_{i=1}^{k}$
is independent so it only remains to show \prettyref{eq:CouplingManyBoxes}.
But \prettyref{eq:CouplingManyBoxes} in the case $k=1$ follows directly
from \prettyref{eq:first}, \prettyref{eq:third} and \prettyref{eq:fourth},
and if $k>1$ it follows from \prettyref{eq:first}, \prettyref{eq:second},
\prettyref{eq:third} and \prettyref{eq:fourth} (using the crude
bound $k\le cN^{d+1}$ and \prettyref{eq:MutualEnergyLowerBound}). 
\end{proof}
We have now completed the proofs of all the main results this article
(\prettyref{thm:GumbelForLocTime}, its corollaries and \prettyref{thm:CouplingManyBoxes})
conditionally on the results of Sections 5 and 6.

\section{\label{sec:CoupleRWandPPoE}Coupling random walk with the Poisson
process of excursions}

In this section we state and prove \prettyref{cor:CoupleExcwithPPP}
which couples $E_{N}-$valued random walk $X_{\text{\ensuremath{\cdot}}}$
with two independent Poisson processes of excursions $\mu_{1},\mu_{2}$,
i.e. Poisson processes on $\mathcal{T}_{\tilde{B}}$ with intensity
proportional to $\nu$ (see \prettyref{eq:PPoEIntensity}), such that
with high probability the trace $X(0,D_{[uK_{N}]})$ (for appropriate
$u$) dominates the trace of $\mu_{1}$ (see \prettyref{eq:ITraceDef})
and such that the union of the traces of $\mu_{1}$ and $\mu_{2}$
dominate $X(0,D_{[uK_{N}]})$ (cf. \prettyref{eq:CoupleExcWithPP}).
The majority of the work will be to couple $X_{\cdot}$ with \begin{equation}
\mbox{iid }E_{N}-\mbox{valued processes }\hat{X}^{1},\hat{X}^{2}...,\hat{X}^{'1},\hat{X}^{'2},...,\mbox{ with law }\kappa_{q}\label{eq:TheSequences}\end{equation}
(see \prettyref{eq:KappaDefinition}), such that for suitable $u$
and $\delta$ the following double inclusion event holds with high
probability:\begin{eqnarray}
\begin{array}{clclcc}
 & I=\{\overset{[u(1-\delta)K_{N}]}{\underset{i=1}{\cup}}\hat{X}^{i}(0,D_{1}) & \subset & \overset{[uK_{N}]}{\underset{i=1}{\cup}}X(R_{i},D_{i})\\
 &  & \subset & \overset{[u(1+\delta/2)K_{N}]}{\underset{i=1}{\cup}}\hat{X}^{i}(0,D_{1})\bigcup\overset{[u\delta/2K_{N}]}{\underset{i=1}{\cup}}\hat{X}^{'i}(0,D_{1})\}.\end{array}\label{eq:IIDExcInclusionEvent}\end{eqnarray}
 To get $\mu_{1},\mu_{2}$ one must then carry out {}``poissonization'',
that is one must {}``put a Poisson number of iid the excursions''
into each of $\mu_{1}$ and $\mu_{2}$. This relatively simple step
is carried out in \prettyref{cor:CoupleExcwithPPP}.

The more challenging step of coupling $X_{\cdot}$ with the iid excursions
$\hat{X}^{1},\hat{X}^{2}...$,$\hat{X}^{'1},\hat{X}^{'2},...$, is
carried out in \prettyref{pro:CoupleRWwithIIDExc}. To prove this
proposition we first quote a result from \cite{Sznitman2009-OnDOMofRWonDCbyRI}
that couples $X_{\cdot}$ with {}``conditionally independent'' excursions
$\tilde{X}^{1},\tilde{X}^{2},...$ which are such that conditionally
on $\tilde{X}_{D_{1}}^{i}\in\mathbb{T}_{N}\times\{zh_{N}\}$, where
$z=\pm1$, the next excursion $\tilde{X}^{i+1}$ has law $\kappa_{q_{zr_{N}}}$.
We then couple the conditionally independent excursions $\tilde{X}^{i}$
with the truly independent excursions $\hat{X}^{1},\hat{X}^{2}...,\hat{X}^{'1},\hat{X}^{'2},...$
by using Sanov's theorem for the empirical distribution of successive
pairs of values of the Markov chain $(\frac{1}{h_{N}}\tilde{X}_{D_{1}}^{i})_{i\ge1}$
(with state space $\{-1,1\}$) to show that for any given $z_{1}\in\{-r_{N},r_{N}\}$
and $z_{2}\in\{-h_{N},h_{N}\}$ the number of $\tilde{X}^{i}$ that
start in $\mathbb{T}_{N}\times\{z_{1}\}$ and end in $\mathbb{T}_{N}\times\{z_{2}\}$
is close to what this value would be if the $\tilde{X}^{i}$ were
truly independent.

Weaker forms of the {}``upper inclusions'' in \prettyref{pro:CoupleRWwithIIDExc}
and \prettyref{cor:CoupleExcwithPPP} appeared as Propositions 3.1
and 4.1 in \cite{Sznitman2009-OnDOMofRWonDCbyRI}. However, as opposed
to the results in this paper, the results in \cite{Sznitman2009-OnDOMofRWonDCbyRI}
require that $u$ and $\delta$ are fixed as $N\rightarrow\infty$.
Our proofs follows the proofs in \cite{Sznitman2009-OnDOMofRWonDCbyRI}
with the most important improvement taking the form of the improved
bound \prettyref{eq:RBoundAbove} on the empirical distribution of
successive pairs of the Markov chain $(\frac{1}{h_{N}}\tilde{X}_{D_{1}}^{i})_{i\ge1}$,
which allows for $\delta$ to go to zero as $N\rightarrow\infty$,
as long as it does not do so too quickly. 
\begin{prop}
\label{pro:CoupleRWwithIIDExc}($d\ge2$) For any $N\ge c$ and $z\in[-N,N]$
one can construct on a space $(\Omega_{2},\mathcal{A}_{2},Q_{2})$
a process $X_{\cdot}$ with law $P_{q_{z}}$ and processes $\hat{X}^{1},\hat{X}^{2}...,\hat{X}^{'1},\hat{X}^{'2},...$,
as in \prettyref{eq:TheSequences} such that for any $u$ and $\delta$
satisfying $uK_{N}\ge(\log N)^{6}$ and $\frac{1}{2}\ge\delta\ge c_{3}\frac{r_{N}}{h_{N}}$
one has, for $I$ as in \prettyref{eq:IIDExcInclusionEvent}, \begin{equation}
\begin{array}{c}
Q_{2}(I^{c})\le cuN^{-3d-1}.\end{array}\label{eq:CoupleRWwithIIDExc}\end{equation}
\end{prop}
\begin{proof}
Let $X_{\cdot}=(Y_{\cdot},Z_{\cdot})$ where $Y_{\cdot}$ is $\mathbb{T}_{N}-$valued
and $Z_{\cdot}$ is $\mathbb{Z}-$valued. For $\gamma=(z_{1},z_{2})\in\Gamma\overset{\textnormal{def}}{=}\{-r_{N},r_{N}\}\times\{-h_{N},h_{N}\}$
we let $P_{\gamma}$ denote the law of $X_{\cdot\wedge D_{1}}$ under
$P_{q_{z_{1}}}$ conditioned on $\left\{ Z_{D_{1}}=z_{2}\right\} $.
By Proposition 2.2 of \cite{Sznitman2009-OnDOMofRWonDCbyRI} we can
construct a coupling $(Q^{'},\Omega',\mathcal{A}')$ of the random
walk $X_{\cdot}$ with law $P_{q_{z}}$, a $\mathbb{Z}^{2}-$valued
process $(\tilde{Z}_{R,k},\tilde{Z}_{D,k})_{k\ge1}$ distributed as
$(Z_{R_{k}},Z_{D_{k}})_{k\ge1}$ under $P_{q_{z}}$ and $E_{N}$-valued
processes $(\tilde{X}_{\cdot}^{k})_{k\ge1}$ which conditionally on
$(\tilde{Z}_{R,k},\tilde{Z}_{D,k})_{k\ge1}$ are independent with
the law of $\tilde{X}_{\cdot}^{k}$ given by $P_{\tilde{Z}_{R,k},\tilde{Z}_{D,k}}$,
such that $Q'(X_{(R_{k}+\cdot)\wedge D_{k}}\ne\tilde{X}_{\cdot}^{k})\le cN^{-4d}$
for all $k$. Thus:\begin{equation}
Q'(\exists k\le2uK_{N}\mbox{ such that }X(R_{k},D_{k})\ne\tilde{X}^{k}(R_{1},D_{1}))\overset{\prettyref{eq:DefOfKn}}{\le}cuN^{-3d-1}.\label{eq:CoupleRWWithDepExc}\end{equation}
We will construct on a space $(\Sigma,\mathcal{B},M)$ a coupling
of a sequence of processes $(\bar{X}^{k})_{k\ge1}$ with the law of
$(\tilde{X}{}^{k})_{k\ge1}$ under $Q^{'}$ and $\hat{X}^{1},\hat{X}^{2}...,\hat{X}^{'1},\hat{X}^{'2},...$,
iid with law $\kappa_{q}$, such that:\begin{equation}
\begin{array}{c}
M(F'^{c})\le\exp(-c(\log N)^{2})\overset{u\ge N^{-1000d},N\ge c}{\le}cuN^{-3d-1}.\end{array}\label{eq:MCoupling}\end{equation}
where $F'$ is the event given in \prettyref{eq:IIDExcInclusionEvent}
with $\overset{\left[uK_{N}\right]}{\underset{i=1}{\cup}}X(R_{i},D_{i})$
replaced by $\overset{\left[uK_{N}\right]}{\underset{i=1}{\cup}}\bar{X}^{i}(R_{1},D_{1})$.
Using the argument below (3.22) in \cite{Sznitman2009-OnDOMofRWonDCbyRI},
this, together with \prettyref{eq:CoupleRWWithDepExc}, is enough
to show the existence of the desired coupling of $X_{\cdot}$ and
$\hat{X}^{i},\hat{X}^{'i}$ such that \prettyref{eq:CoupleRWwithIIDExc}
holds (essentially speaking because we can construct $(\Omega',\mathcal{A}',Q')$
such that the regular conditional probability of $X_{\cdot}$ given
$(\tilde{X}_{\cdot}^{i})_{i\ge1}$ exists).

We thus proceed with the construction of $(\Sigma,\mathcal{B},M)$.
We start by defining on $(\Sigma,\mathcal{B},M)$ the following collections
of processes\begin{equation}
\gamma_{k}\in\mathbb{Z}^{2},k\ge1,(\gamma_{k}\in\Gamma\mbox{ if }k\ge2\mbox{) with the law of }(Z_{R_{k}},Z_{D_{k}})_{k\ge1}\mbox{ under }P_{q_{z}},\label{eq:TheCollection1}\end{equation}
\begin{equation}
\gamma_{k}^{'}\in\Gamma,k\ge1,\mbox{ iid, where }\gamma_{k}^{'}\mbox{ has the law of }(Z_{R_{1}},Z_{D_{1}})\mbox{ under }P_{q},\end{equation}
\begin{equation}
\mbox{for all }\gamma\in\Gamma\mbox{ an iid sequence }\left(\zeta_{i}^{\gamma}(\cdot)\right)_{i\ge1}\mbox{of processes with law }P_{\gamma},\end{equation}
\begin{equation}
\mbox{an iid sequence }(\hat{X}_{\cdot}^{'i})_{i\ge1}\mbox{ of processes with law }\kappa_{q},\end{equation}
such that the collections are mutually independent. Also define for
every $\gamma\in\Gamma$: \begin{equation}
N_{k}(\gamma)=|\{j\in[2,k]:\gamma_{j}=\gamma\}|,k\ge2,N_{k}^{'}(\gamma)=|\{j\in[1,k]:\gamma_{j}^{'}=\gamma\}|,k\ge1.\label{eq:DefOfNs}\end{equation}
We further let:\begin{equation}
\begin{array}{cclcc}
\hat{X}_{\cdot}^{k} & = & \zeta_{N_{k}^{'}(\gamma_{k}^{'})}^{\gamma_{k}^{'}}(\cdot)\mbox{ for }k\ge1,\\
\bar{X}_{\cdot}^{k} & = & \zeta_{N_{k}(\gamma_{k})}^{\gamma_{k}}(\cdot)\mbox{ for }k\ge2\mbox{ and }\bar{X}_{\cdot}^{1} & = & \hat{X}_{H_{\mathbb{T}_{N}\times\{z\}}+\cdot}^{'i_{0}},\end{array}\label{eq:XProcDef}\end{equation}
where $i_{0}=\inf\{i\ge1:J_{i}\mbox{ holds}\}$, $J_{i}=\{\hat{X}^{'i}\mbox{ hits }\mathbb{T}_{N}\times\{z\}\mbox{ before leaving }\tilde{B},\hat{X}_{D_{1}}^{'i}\in\mathbb{T}_{N}\times\{z_{2}\}\}$
and $\gamma_{1}=(z,z_{2})$. We then have that:\begin{equation}
(\bar{X}_{\cdot}^{k})_{k\ge1}\mbox{ under }M\mbox{ has the same law as }(\tilde{X}_{\cdot}^{k})_{k\ge1}\mbox{ under }Q^{'}\mbox{, and }\end{equation}
\begin{equation}
\hat{X}^{1},\hat{X}^{2},...,\hat{X}^{'1},\hat{X}^{'2},...,\mbox{ under }M\mbox{ are iid with law }\kappa_{q}.\end{equation}
Thus it only remains to show \eqref{eq:MCoupling}. We introduce
the {}``good event'':\begin{eqnarray*}
\mathcal{G} & = & \{i_{0}\le\left[\frac{u\delta}{2}K_{N}\right],N_{\left[u(1-\delta)K_{N}\right]}^{'}(\gamma)\le N_{\left[uK_{N}\right]}(\gamma)\le N_{\left[u(1+\frac{\delta}{2})K_{N}\right]}^{'}(\gamma)\forall\gamma\in\Gamma\}.\end{eqnarray*}
By \prettyref{eq:DefOfNs} and \prettyref{eq:XProcDef} we have $\mathcal{G}\subset F'$
so to show \eqref{eq:MCoupling} it suffices to show that \begin{equation}
M(\mathcal{G}^{c})\le c\exp(-c(\log N)^{2}).\label{eq:GoodEventSufficesToShow-1}\end{equation}
Note that $M(i_{0}>n)\le M(J_{1}^{c})^{n},n=0,1,...$, and $M(J_{1})\ge\frac{49}{100}$
for $N\ge c$ by a one dimensional random walk calculation (see (3.23)
of \cite{Sznitman2009-OnDOMofRWonDCbyRI}). So if $N\ge c$ and $\delta\ge\frac{r_{N}}{h_{N}}$:\begin{equation}
M(i_{0}>\left[\frac{u\delta}{2}K_{N}\right])\le\left(\frac{51}{100}\right)^{\left[\frac{u\delta}{2}K_{N}\right]}\le c\exp\left(-cuK_{N}\frac{r_{N}}{h_{N}}\right).\label{eq:i0bound}\end{equation}
Recall that the sequence $\gamma_{1}^{'},\gamma_{2}^{'},...,$ is
iid and note that $M(\gamma_{1}^{'}=\gamma)=\frac{1}{2}p_{N}1_{\{z_{1}z_{2}>0\}}+\frac{1}{2}q_{N}1_{\{z_{1}z_{2}<0\}}$,
for $\gamma=(z_{1},z_{2})\in\Gamma$, where $p_{N}=\frac{1}{2}+\frac{1}{2}\frac{r_{N}}{h_{N}}$
and $q_{N}=\frac{1}{2}-\frac{1}{2}\frac{r_{N}}{h_{N}}$. Using the
exponential Chebyshev inequality we get that if $N\ge c$ and $\delta\ge6\frac{r_{N}}{h_{N}}$
(ensuring that $\frac{1}{2}q_{N}\left[u(1+\frac{\delta}{2})K_{N}\right]$,
the {}``typical size'' of $N_{\left[u(1+\frac{\delta}{2})K_{N}\right]}^{'}(z_{1},z_{2})$
when $z_{1}z_{2}<0$, is {}``much larger'' than $\frac{1}{4}u(1+\frac{\delta}{4})K_{N}$):
 \begin{eqnarray}
\begin{array}{ccc}
\sup_{\gamma\in\Gamma}M\left(N_{\left[u(1-\delta)K_{N}\right]}^{'}(\gamma)\ge\frac{1}{4}u(1-\frac{\delta}{4})K_{N}\right) & \le & c\exp\left(-cuK_{N}\left(\frac{r_{N}}{h_{N}}\right)^{2}\right),\\
\sup_{\gamma\in\Gamma}M\left(N_{\left[u(1+\frac{\delta}{2})K_{N}\right]}^{'}(\gamma)\le\frac{1}{4}u(1+\frac{\delta}{4})K_{N}\right) & \le & c\exp\left(-cuK_{N}\left(\frac{r_{N}}{h_{N}}\right)^{2}\right).\end{array}\label{eq:IIDGuysBound}\end{eqnarray}
If we write $\gamma^{i}=(z_{1}^{i},z_{2}^{i})$ then $V_{i}=\frac{z_{2}^{i}}{h_{N}},i\ge1,$
is a Markov chain on $\{-1,1\}$ with transition probabilities $P(V_{i+1}=a|V_{i}=b)=p_{N}1_{\{ab>0\}}+q_{N}1_{\{ab<0\}}$
for $a,b=\pm1$. Also $\gamma^{i}=(V_{i-1}r_{N},V_{i}h_{N})$ almost
surely for all $i\ge2$. The sequence of consecutive pairs $(V_{i-1},V_{i}),i\ge2,$
is a Markov chain on $\{-1,1\}^{2}$. Let $\left(U_{i}\right)_{i\ge0}$
under the probability $\tilde{R}_{\sigma}$ be a Markov chain on $\{-1,1\}^{2}$
with the same transition probabilities as $(V_{i-1},V_{i})$ but with
$U_{0}=\sigma\in\{-1,1\}^{2}$ almost surely. If $a,b=\pm1$ let $I_{1}$
and $I_{2}$ denote the events $\{\overset{\left[uK_{N}\right]-1}{\underset{i=1}{\sum}}1_{\left\{ U_{i}=(a,b)\right\} }\ge\frac{1}{4}u(1+\frac{\delta}{4})K_{N}\}$
and $\{\overset{\left[uK_{N}\right]-1}{\underset{i=1}{\sum}}1_{\left\{ U_{i}=(a,b)\right\} }\le\frac{1}{4}u(1-\frac{\delta}{4})K_{N}\}$
respectively. Then by \prettyref{eq:DefOfNs} we have:\begin{equation}
\begin{array}{ccc}
M\left(N_{\left[uK_{N}\right]}((ar_{N},bh_{N}))\ge\frac{1}{4}u(1+\frac{\delta}{4})K_{N}\right) & \le & \sup_{\sigma}\tilde{R}_{\sigma}(I_{1}),\\
M\left(N_{\left[uK_{N}\right]}((ar_{N},bh_{N}))\le\frac{1}{4}u(1-\frac{\delta}{4})K_{N}\right) & \le & \sup_{\sigma}\tilde{R}_{\sigma}(I_{2}).\end{array}\label{eq:RProbOfPairs}\end{equation}
We have the following lemma:
\begin{lem}
\label{lem:LDPApplication}($N\ge c$) If $\frac{1}{2}\ge\delta\ge32\frac{r_{N}}{h_{N}}$
then for all $a,b=\pm1$ \begin{equation}
\sup_{\sigma}\tilde{R}_{\sigma}(I_{i})\le\exp\left(-cuK_{N}\left(\frac{r_{N}}{h_{N}}\right)^{2}\right)\mbox{ for }i=1,2.\label{eq:RBoundAbove}\end{equation}
\end{lem}
\begin{proof}
By symmetry it suffices to check the cases $a=1,b=1$ and $a=1,b=-1$.
For a probability $\mu(\cdot,\cdot)$ on $\{-1,1\}^{2}$ we write
$\mu_{1},\mu_{2}$ for its marginals and $\mu(j|i)=\frac{\mu(i,j)}{\mu_{1}(i)}$.
By Theorem 3.1.13 p. 79 of \cite{DemboZeitouniLargeDeviations} and
by sub-additivity (cf. Lemma 6.1.11 p. 255 and Lemma 6.3.1 p. 273
of \cite{DemboZeitouniLargeDeviations}), we have that for all $n\ge1$
and $x\in(0,\frac{1}{4}]$ \begin{equation}
\begin{array}{ccl}
\inf_{\sigma}\tilde{R}_{\sigma}(\frac{1}{n}\sum_{i=0}^{n-1}1_{\{U_{i}=(1,b)\}}\ge\frac{1}{4}+x)) & \le & \exp(-n\Psi_{2,N}^{+}(x)),\\
\inf_{\sigma}\tilde{R}_{\sigma}(\frac{1}{n}\sum_{i=1}^{n}1_{\{U_{i}=(1,b)\}}\le\frac{1}{4}-x)) & \le & \exp(-n\Psi_{2,N}^{-}(x))\end{array}\label{eq:LargeDevInfBound}\end{equation}
where\begin{eqnarray*}
\Psi_{2,N}^{+}(x) & = & \inf\{H_{2,N}(\mu):\mu\mbox{ probability on }\{-1,1\}^{2},\mu_{1}=\mu_{2},\mu(1,b)\ge\frac{1}{4}+x\},\\
\Psi_{2,N}^{-}(x) & = & \inf\{H_{2,N}(\mu):\mu\mbox{ probability on }\{-1,1\}^{2},\mu_{1}=\mu_{2},\mu(1,b)\le\frac{1}{4}-x\},\\
H_{2,N}(\mu) & = & \mu_{1}(1)\{\mu(1|1)\log\frac{\mu(1|1)}{p_{N}}+\mu_{1}(-1|1)\log\frac{\mu(-1|1)}{q_{N}}\}+\\
 &  & \mu_{1}(-1)\{\mu(1|-1)\log\frac{\mu(1|-1)}{q_{N}}+\mu_{1}(-1|-1)\log\frac{\mu(-1|-1)}{p_{N}}\}.\end{eqnarray*}
Because $\inf_{\sigma,j}\tilde{R}_{\sigma}(U_{2}=j)\ge c$ we have\begin{eqnarray}
\begin{array}{cl}
 & \sup_{\sigma}\tilde{R}_{\sigma}\left(\sum_{i=1}^{\left[uK_{N}\right]-1}1_{\{U_{i}=(1,b)\}}\ge\frac{1}{4}u(1+\frac{\delta}{4})K_{N}\right)\\
\le & \frac{1}{c}\inf_{\sigma}\tilde{R}_{\sigma}\left(\sum_{i=1}^{\left[uK_{N}\right]+1}1_{\{U_{i}=(1,b)\}}\ge\frac{1}{4}u(1+\frac{\delta}{4})K_{N}\right)\\
\overset{uK_{N}\delta\ge2,\prettyref{eq:LargeDevInfBound}}{\le} & c\exp\left(-\left(\left[uK_{N}\right]+1\right)\Psi_{2,N}^{+}(\frac{\delta}{32})\right)\le c\exp\left(-uK_{N}\Psi_{2,N}^{+}(\frac{\delta}{32})\right),\end{array}\label{eq:inftosup1}\end{eqnarray}
and similarly\begin{equation}
\sup_{\sigma}\tilde{R}_{\sigma}\left(\sum_{i=1}^{\left[uK_{N}\right]-1}1_{\{U_{i}=(1,b)\}}\le\frac{1}{4}u(1-\frac{\delta}{4})K_{N}\right)\le c\exp\left(-uK_{N}\Psi_{2,N}^{-}(\frac{\delta}{32})\right).\label{eq:inftosup2}\end{equation}
To conclude the proof of the lemma it thus suffices to show that for
$b=-1,1$: \begin{equation}
\Psi_{2,N}^{\pm}\left(\frac{r_{N}}{h_{N}}\right)\ge c\left(\frac{r_{N}}{h_{N}}\right)^{2}.\label{eq:SufficesToShowPsi}\end{equation}
Consider the function $f_{p}(x)=x\log\frac{x}{p}+(1-x)\log\frac{1-x}{1-p}$
for $p\in(0,1)$ and $x\in[0,1]$. Since $f_{p}^{''}(x)=\frac{1}{x}+\frac{1}{1-x}\ge4$
and $f_{p}(p)=f_{p}^{'}(p)=0$ integrating twice gives that $f_{p}(x)\ge2(x-p)^{2}$
for all $p\in(0,1)$ and $x\in[0,1]$. Using this and $\mu_{1}(1),\mu_{1}(-1)\ge1$
we get\begin{eqnarray*}
H_{2,N}(\mu) & \ge & 2\left\{ \mu_{1}(1)^{2}(\mu(1|1)-p_{N})^{2}+\mu_{1}(-1)^{2}(\mu(-1|-1)-p_{N})^{2}\right\} \\
 & = & 2\left\{ (\mu(1,1)(1-p_{N})-p_{N}\mu(1,-1))^{2}+(\mu(-1,-1)(1-p_{N})-p_{N}\mu(-1,1))^{2}\right\} \end{eqnarray*}
Write $\theta=\frac{1}{2}p_{N}$, $\mu(1,1)=\theta+y$ and $\mu(-1,-1)=\theta+z$.
If $\mu_{1}=\mu_{2}$ then $\mu(1,-1)=\mu(-1,1)=\frac{1-y-z}{2}-\theta$
and plugging this into the above formula we get \begin{eqnarray*}
H_{2,N}(\mu) & \ge & 2\left\{ \left(y(1-\theta)+z\theta\right)^{2}+\left(z(1-\theta)+y\theta\right)^{2}\right\} \\
 & \overset{yz\ge-\frac{y^{2}+z^{2}}{2}}{\ge} & 2(1-2\theta)^{2}\{y^{2}+z^{2}\}\\
 & \ge & c\left\{ (\mu(1,1)-\frac{1}{2}p_{N})^{2}+(\mu(-1,-1)-\frac{1}{2}p_{N})^{2}\right\} .\end{eqnarray*}
Now if $|\mu(1,1)-\frac{1}{4}|\ge\frac{r_{N}}{h_{N}}$ or $|\mu(-1,-1)-\frac{1}{4}|\ge\frac{r_{N}}{h_{N}}$
and $\mu_{1}=\mu_{2}$ then $|\mu(1,1)-\frac{1}{2}p_{N}|\ge\frac{3}{4}\frac{r_{N}}{h_{N}}$
or $|\mu(-1,-1)-\frac{1}{2}p_{N}|\ge\frac{3}{4}\frac{r_{N}}{h_{N}}$,
so by the above inequality $H_{2,N}(\mu)\ge c\left(\frac{r_{N}}{h_{N}}\right)^{2}$
and thus \prettyref{eq:SufficesToShowPsi} holds when $b=1$. Furthermore
if $|\mu(1,-1)-\frac{1}{4}|\ge\frac{r_{N}}{h_{N}}$ and $\mu_{1}=\mu_{2}$
then $|\mu(1,1)-\frac{1}{4}|\ge\frac{r_{N}}{h_{N}}$ or $|\mu(-1,-1)-\frac{1}{4}|\ge\frac{r_{N}}{h_{N}}$
so \prettyref{eq:SufficesToShowPsi} holds also when $b=-1$. Thus
the proof of the lemma is complete.
\end{proof}
We can now finish the proof of \prettyref{pro:CoupleRWwithIIDExc}.
Combining \prettyref{eq:i0bound}, \emph{\prettyref{eq:IIDGuysBound},
\prettyref{eq:RProbOfPairs} }and\emph{ \prettyref{eq:RBoundAbove}},
and using the bounds $c\exp(-cuK_{N}\frac{r_{N}}{h_{N}})\le c\exp(-cuK_{N}(\frac{r_{N}}{h_{N}})^{2})\le c\exp(-c(\log N)^{2})$
(recall $uK_{N}\ge(\log N)^{6}$), we deduce that \prettyref{eq:GoodEventSufficesToShow-1}
holds. Thus the proof of \prettyref{pro:CoupleRWwithIIDExc} is complete.

\end{proof}

We are now ready to carry out the process of {}``poissonization''
to construct Poisson processes of excursions (i.e. Poisson processes
on $\mathcal{T}_{\tilde{B}}$ of intensity a multiple of $\nu$, cf.
\prettyref{eq:PPoEIntensity}) from the iid excursions of the previous
proposition.
\begin{cor}
\label{cor:CoupleExcwithPPP}For all $N\ge c$, $z\in[-N,N]$, $1>\delta\ge c_{4}\frac{r_{N}}{h_{N}}$
and $u$ satisfying $uK_{N}\ge(\log N)^{6}$ we can define on a space
$(\Omega_{3},\mathcal{A}_{3},Q_{3})$ a process $X_{\cdot}$ with
law $P_{q_{z}}$ and two independent Poisson point processes $\mu_{1},\mu_{2}$
on $\mathcal{T}_{\tilde{B}}$ with intensities $u(1-\delta)\nu$ and
$2\delta u\nu$ respectively such that\begin{equation}
\begin{array}{rr}
Q_{3}(\mathcal{I}(\mu_{1})\subset\cup_{i=1}^{\left[uK_{N}\right]}X(R_{i},D_{i})\subset\mathcal{I}(\mu_{1})\bigcup\mathcal{I}(\mu_{2}))\ge1-cuN^{-3d-1}.\end{array}\label{eq:CoupleExcWithPP}\end{equation}
\end{cor}
\begin{proof}
Let $c_{4}=2c_{3}$ so that we can apply \prettyref{pro:CoupleRWwithIIDExc}
with $\frac{\delta}{2}$ in place of $\delta$ to get a space $(\Omega_{2},\mathcal{A}_{2},Q_{2})$
with a process $X_{\cdot}$ with law $P_{q_{z}}$ and processes $(\hat{X}^{k})_{k\ge1},(\hat{X}^{'k})_{k\ge1}$,
as in \prettyref{eq:TheSequences} such that $Q_{2}(I^{'c})\le cuN^{-3d-1}$,
where $I'$ is the event in \prettyref{eq:IIDExcInclusionEvent} with
$\delta$ replaced by $\frac{\delta}{2}$. We define $(\Omega_{3},\mathcal{A}_{3},Q_{3})$
by extending $(\Omega_{2},\mathcal{A}_{2},Q_{2})$ with independent
Poisson random variables $J_{1}$ with parameter $K_{N}u(1-\delta)$,
$J_{2}$ with parameter $K_{N}u\frac{3\delta}{2}$ and $J_{3}$ with
parameter $K_{N}u\frac{\delta}{2}$, which are also independent from
$(\hat{X}^{k})_{k\ge1},(\hat{X}^{'k})_{k\ge1}$. We then define\begin{equation}
\mu_{1}=\sum_{1\le k\le J_{1}}\delta_{\hat{X}^{k}}\mbox{ and }\mu_{2}=\sum_{J_{1}+1\le k\le J_{1}+J_{2}}\delta_{\hat{X}^{k}}+\sum_{1\le k\le J_{3}}\delta_{\hat{X}^{'k}}.\label{eq:DefOfMus}\end{equation}
Then $\mu_{1}$ and $\mu_{2}$ are independent Poisson point processes
with intensities $u(1-\delta)\nu$ and $2\delta u\nu$. It thus only
remains to show \prettyref{eq:CoupleExcWithPP}. Note that the complement
of the event in the left-hand side of \prettyref{eq:CoupleExcWithPP}
is included in\begin{eqnarray*}
 & \{J_{1}>[u(1-\frac{\delta}{2})K_{N}]\}\cup\{J_{1}+J_{2}<[u(1+\frac{\delta}{4})K_{N}]\}\cup\{J_{3}<[u\frac{\delta}{4}K_{N}]\}\cup I'^{c}.\end{eqnarray*}
But using standard large deviation bounds we see that the probabilities
of the first three events in the union are bounded above by $\exp(-cuK_{N}\delta^{2})\overset{uK_{N}\delta^{2}\ge(\log N)^{2},N\ge c}{\le}cuN^{-3d-1}$
and thus \eqref{eq:CoupleExcWithPP} follows since we already know
$Q_{3}(I'^{c})\le cuN^{-3d-1}$.
\end{proof}
In finishing the proof of \prettyref{cor:CoupleExcwithPPP} we have
now proved the first of the two main ingredients that were used to
prove \prettyref{thm:CouplingManyBoxes}.

\section{\label{sec:CouplePPoEwithRI}Coupling the Poisson process of excursions
with random interlacements}

In this section we state and prove \prettyref{pro:PPoEandRIcoupling},
which couples the Poisson process of excursions with random interlacements
and whose application was an important part of the proof of \prettyref{thm:CouplingManyBoxes}.
It states that if we have a Poisson process of excursions $\mu$ (i.e.
Poisson process on $\mathcal{T}_{\tilde{B}}$ of intensity $u\nu$
for $u\ge0$, see \prettyref{eq:PPoEIntensity}) then for any $x\in\mathbb{T}_{N}\times[-\frac{N}{2},\frac{N}{2}]$
and $\varepsilon\in(0,1)$ we can, provided $N$ is large enough and
$u_{-}<u$ and $u_{+}>u$ are {}``sufficiently far'' from $u$,
construct independent random sets $\mathcal{I}_{1},\mathcal{I}_{2}\subset\mathsf{A}=B(0,N^{1-\varepsilon})$
such that \begin{equation}
\mathcal{I}_{1}\mbox{ has the law of }\mathcal{I}^{u_{-}}\cap\mathsf{A}\mbox{ under }Q_{0}\mbox{ and }\mathcal{I}_{2}\mbox{ has the law of }\mathcal{I}^{u_{+}-u_{-}}\cap\mathsf{A}\mbox{ under }Q_{0}\label{eq:LawofI1andI2}\end{equation}
and with high probability $\mathcal{I}_{1}\subset(\mathcal{I}(\mu)-x)\cap\mathsf{A}\subset\mathcal{I}_{1}\cup\mathcal{I}_{2}$
(recall that $\mathcal{I}^{u}$ under $Q_{0}$ is a random interlacement
and that $(\mathcal{I}_{1},\mathcal{I}_{1}\cup\mathcal{I}_{2})\overset{\mbox{law}}{=}(\mathcal{I}^{u_{-}}\cap\mathsf{A},\mathcal{I}^{u_{+}}\cap\mathsf{A})$
by \prettyref{eq:markovpropforI}). More precisely: 
\begin{prop}
\label{pro:PPoEandRIcoupling}Assume $\varepsilon\in(0,1),$ $N\ge c(\varepsilon)$,
$x\in\mathbb{T}_{N}\times[-\frac{N}{2},\frac{N}{2}]$ and let $\mathsf{A}=B(0,N^{1-\varepsilon})$.
Suppose that we have a Poisson process $\mu$ on $\mathcal{T}_{\tilde{B}}$
with intensity $u\nu,u\ge0$ defined on a space $(\Omega,\mathcal{A},Q)$
. Then if $0\le u_{-}<u<u_{+},\frac{u_{+}}{u},\frac{u}{u_{-}}\ge N^{-c_{5}}$,
where $c_{5}=c_{5}(\varepsilon)>0$, we can define a space $(\Omega^{'},\mathcal{A}^{'},Q^{'})$
and random sets $\mathcal{I}_{1},\mathcal{I}_{2}\subset\mathsf{A}$
on the product space $(\Omega\times\Omega^{'},\mathcal{A}\otimes\mathcal{A}^{'},Q\otimes Q^{'})$
such that \prettyref{eq:LawofI1andI2} holds, \begin{equation}
\mathcal{I}_{1},\mathcal{I}_{2},\mbox{ are independent, }\sigma(\mu)\otimes\mathcal{A}^{'}-\mbox{measurable and}\label{eq:CorrectMeasurability}\end{equation}
\begin{eqnarray}
Q\otimes Q^{'}(\mathcal{I}_{1}\subset(\mathcal{I}(\mu)-x)\cap\mathsf{A}\subset\mathcal{I}_{1}\cup\mathcal{I}_{2}) & \ge1-cu_{+}N^{-10(d+1)}.\label{eq:InclusionBound}\end{eqnarray}

\end{prop}
Before starting the proof of \prettyref{pro:PPoEandRIcoupling} we
make some definitions and state \prettyref{pro:XiBound}, all of which
we will need in the proof. We define the box\begin{equation}
\mathsf{A}'=B(x,N^{1-\varepsilon}).\label{eq:DefIfAPrime}\end{equation}
The first step in the proof of \prettyref{pro:PPoEandRIcoupling}
will be to extract from $\mu$ a Poisson process $\mu'$, by keeping
only trajectories in $\mu$ that hit $\mathsf{A}'$ (the others are
irrelevant for the coupling). We will see that what is left, i.e.
$\mu'$, is a Poisson process on $\mathcal{T}_{\tilde{B}}$ of intensity
$u\kappa_{e_{\mathsf{A}',\tilde{B}}}$. We define the boxes\begin{equation}
\mathsf{B}=B(0,N^{1-\varepsilon/2}),\mathsf{C}=B(0,\frac{N}{4}),\mathsf{B}^{'}=B(x,N^{1-\varepsilon/2}),\mathsf{C}^{'}=B(x,\frac{N}{4}).\label{eq:DefOfBCBPrimeCPrime}\end{equation}
Note that for $N\ge c(\varepsilon)$\begin{equation}
\mathsf{A}\subset\mathsf{B}\subset\mathsf{C}\mbox{ and }\mathsf{A}'\subset\mathsf{B}'\subset\mathsf{C}'.\label{eq:ABCInclusion}\end{equation}
(One should not confuse $\mathsf{B}$ with $B$ from \prettyref{eq:rnbnbbtilde}.)
We further define for $k\ge1$ the successive returns $\tilde{R}_{k}$
to $\mathsf{A}$ and departures $\tilde{D}_{k}$ from $\mathsf{B}$,
and returns $\tilde{R}_{k}^{'}$ to $\mathsf{A}^{'}$ and departures
$\tilde{D}_{k}^{'}$ from $\mathsf{B}'$, analogously to \prettyref{eq:ExcursionDef}.
For $1\le l<\infty$ we then introduce the maps $\phi_{l},\phi_{l}^{'}$
from $\{\tilde{D}_{k}<\infty=\tilde{R}_{k+1}\}\subset W$ and $\{\tilde{D}_{k}^{'}<T_{\tilde{B}}<\tilde{R}_{k+1}^{'}\}\subset\mathcal{T}_{\tilde{B}}$
respectively into $W^{\times l}$:\begin{equation}
\phi_{l}(w)\overset{\mbox{def}}{=}(w((\tilde{R}_{i}+\cdot)\wedge\tilde{D}_{i}))_{1\le i\le l},\phi_{l}^{'}(w)\overset{\mbox{def}}{=}(w((\tilde{R}_{i}^{'}+\cdot)\wedge\tilde{D}_{i}^{'})-x)_{1\le i\le l}.\label{eq:DefOfPhiPhiPrime}\end{equation}
For $1\le l<\infty$ we will then consider $\mu_{l}$, the image of
$\mu^{'}$ under $\phi_{l}^{'}$, that is from each trajectory we
will only keep its excursions between $\mathsf{A}'$ and $\partial_{e}\mathsf{B}'$.
Essentially speaking we will then be left with Poisson point processes
$\mu_{l},l\ge1$, on the spaces $W^{\times l},l\ge1$, of intensities
$u\xi_{E}^{l},l\ge1$, where for $w_{1},...,w_{l}\in W$\begin{eqnarray}
\xi_{E}^{l}(w_{1},...,w_{l}) & \overset{\mbox{def}}{=} & \left(\phi_{l}^{'}\circ(1_{\{\tilde{D}_{l}^{'}<T_{\tilde{B}}<\tilde{R}_{l+1}^{'}\}}\kappa_{e_{\mathsf{A}',\tilde{B}}})\right)(w_{1}+x,...,w_{l}+x).\label{eq:PsiEDef}\end{eqnarray}
Recall from \prettyref{eq:lawofIuCAPK} that $\mathcal{I}^{s}\cap\mathsf{A}$
has the law of $\mathcal{I}(\mu_{\mathsf{A},s})\cap\mathsf{A}$ for
any $s\ge0$. It will turn out that if we consider the image of $1_{\{\tilde{D}_{l}<\infty=\tilde{R}_{l+1}\}}\mu_{\mathsf{A},s}$
under $\phi_{l}$, thereby only keeping the excursions between $\mathsf{A}$
and $\partial_{e}\mathsf{B}$, we get Poisson processes \emph{on the
same spaces} $W^{\times l},l\ge1$, but with intensities $s\xi_{\mathbb{Z}^{d+1}}^{l}$,
where for $w_{1},...,w_{l}\in W$ \begin{eqnarray}
\xi_{\mathbb{Z}^{d+1}}^{l}(w_{1},...,w_{l}) & \overset{\mbox{def}}{=} & \left(\phi_{l}\circ(1_{\{\tilde{D}_{l}<\infty=\tilde{R}_{l+1}\}}P_{e_{\mathsf{A}}}^{\mathbb{Z}^{d+1}})\right)(w_{1},...,w_{l}).\label{eq:PsiZDef}\end{eqnarray}
The following comparison of $\xi_{\mathbb{Z}^{d+1}}^{l}$ and $\xi_{E}^{l}$
is crucial in the proof of \prettyref{pro:PPoEandRIcoupling}:
\begin{prop}
\label{pro:XiBound}($N\ge c(\varepsilon))$ For all $l\ge1$\begin{equation}
(1-c(l)N^{-c(\varepsilon)})\xi_{\mathbb{Z}^{d+1}}^{l}\le\xi_{E}^{l}\le(1+c(l)N^{-c(\varepsilon)})\xi_{\mathbb{Z}^{d+1}}^{l}.\label{eq:XiBound}\end{equation}

\end{prop}
We postpone the proof of \prettyref{pro:XiBound} until after the
proof of \prettyref{pro:PPoEandRIcoupling}. In \prettyref{pro:PPoEandRIcoupling}
we will use \prettyref{pro:XiBound} to {}``thin'' $\mu_{l},l=1,...,r$,
where $r$ is a constant, to get Poisson processes $\mu_{l}^{-},l=1,...,r$
with intensity $u_{-}\mathcal{\xi}_{\mathbb{Z}^{d+1}}^{l}$ and {}``thicken''
$\mu_{l}-\mu_{l}^{-}$ to get Poisson processes $\mu_{l}^{+}=1,...,r$
with intensity $u_{+}\mathcal{\xi}_{\mathbb{Z}^{d+1}}^{l}$, such
for each $l=1,...,r$ we have $\mu_{l}^{-}\le\mu_{l}\le\mu_{l}^{-}+\mu_{l}^{+}$.
We will then essentially speaking define the set $\mathcal{I}_{1}$
in terms of the traces of the $\mu_{l}^{-}$ and define $\mathcal{I}_{2}$
in terms of the traces of the $\mu_{l}^{+}$. We will pick $r$ (see
\prettyref{eq:DefOfR}) such that with high probability $\sum_{l>r}\mu_{l}=0$.
This together with the relation $\mu_{l}^{-}\le\mu_{l}\le\mu_{l}^{-}+\mu_{l}^{+}$
will allow us to prove that the inclusion in \prettyref{eq:InclusionBound}
holds with high probability. Also the relation between $\xi_{\mathbb{Z}^{d+1}}^{l}$,
$\mu_{\mathsf{A},s}$ and $\mathcal{I}^{s}\cap\mathsf{A}$ described
above will allow us to show \prettyref{eq:LawofI1andI2}. We now start
the proof.
\begin{proof}[Proof of \prettyref{pro:PPoEandRIcoupling}.]
We start with by constructing the processes $\mu_{l}$ in the following
lemma:
\begin{lem}
\label{lem:ConstructMuL}We can define on $(\Omega,\mathcal{A},Q)$
processes $\mu_{l},1\le l<\infty$, such that\begin{eqnarray}
 & \mu_{l}\mbox{ are independent }\sigma(\mu)-\mbox{measurable Poisson point processes,}\label{eq:MuLIndependence}\\
 & \mu_{l}\mbox{ has statespace }W^{\times l}\mbox{ and intensity }u\xi_{E}^{l},\label{eq:MuLIntensity}\\
 & (\mathcal{I}(\mu)-x)\cap\mathsf{A}=\cup_{l\ge1}\mathcal{I}(\mu_{l})\mbox{ almost surely.}\label{eq:MuMuLTrace}\end{eqnarray}
\end{lem}
\begin{proof}
Define on $(\Omega,\mathcal{A},Q)$ the processes $\mu^{'}=\sum_{n\ge0}1_{\{w_{n}\mbox{ hits }\mathsf{A}'\}}\delta_{w_{n}(H_{\mathsf{A}'}+\cdot)}$
when $\mu=\sum_{n\ge0}\delta_{w_{n}}$. Then $\mu^{'}$ is a Poisson
process on $\mathcal{T}_{\tilde{B}}$ of intensity \[
uK_{N}P_{q}(X_{(H_{\mathsf{A}'}+\cdot)\wedge T_{\tilde{B}}}\in dw)\overset{\prettyref{eq:GeometricLemmaStrongMarkov}}{=}u\kappa_{e_{\mathsf{A}',\tilde{B}}}(dw)\]
 and\begin{equation}
(\mathcal{I}(\mu)-x)\cap\mathsf{A}=(\mathcal{I}(\mu^{'})-x)\cap\mathsf{A}\mbox{ almost surely}.\label{eq:MuPrimeMuEquality}\end{equation}
Furthermore define for $1\le l<\infty$ the process $\mu_{l}$ as
the image of $1_{\{\tilde{D}_{l}^{'}<T_{\tilde{B}}<\tilde{R}_{l+1}^{'}\}}\mu^{'}$
under $\phi_{l}$. Then since $\{\tilde{D}_{l}^{'}<T_{\tilde{B}}<\tilde{R}_{l+1}^{'}\}$
are disjoint, and $\mu'$ only depends on $\mu$ we get \prettyref{eq:MuLIndependence}.
By \prettyref{eq:DefOfPhiPhiPrime} and \prettyref{eq:PsiEDef} we
get \prettyref{eq:MuLIntensity}. Finally \prettyref{eq:MuMuLTrace}
follows by \prettyref{eq:MuPrimeMuEquality} and \prettyref{eq:DefOfPhiPhiPrime}
and since $\cup_{l\ge1}\{\tilde{D}_{l}^{'}<T_{\tilde{B}}<\tilde{R}_{l+1}^{'}\}$
equals the support of $\mu'$. 
\end{proof}
The next step in the proof of \prettyref{pro:PPoEandRIcoupling} is
to construct the processes $\mu_{l}^{-},\mu_{l}^{+}$. First let us
define

\begin{equation}
r=\left[\frac{10(d+1)}{c_{6}(\varepsilon)}\right]+1,\label{eq:DefOfR}\end{equation}
where $c_{6}=c_{6}(\varepsilon)$ is the constant from \prettyref{lem:EscapeIsLikely}.
We have the following lemma:
\begin{lem}
$(N\ge c(\varepsilon)$,$0\le u_{-}<u<u_{+},\frac{u_{+}}{u},\frac{u_{-}}{u}\ge N^{-c_{5}}$)
We can construct a space $(\Omega',\mathcal{A}',Q')$ and processes
$\mu_{l}^{-},\mu_{l}^{+},1\le l\le r$ such that\begin{eqnarray}
 & \mu_{l}^{-},\mu_{l}^{+},1\le l\le r,\mbox{ are independent }\sigma(\mu)\times\mathcal{A}'-\mbox{measurable},\label{eq:MuLMinusMeasurability}\\
 & \mu_{l}^{-}\mbox{ is a Poisson point process on }W^{\times l}\mbox{ of intensity }u_{-}\mathcal{\xi}_{\mathbb{Z}^{d+1}}^{l},\label{eq:IntOfMuLMinus-1}\\
 & \mu_{l}^{+}\mbox{ is a Poisson point process on }W^{\times l}\mbox{ of intensity }(u_{+}-u_{-})\xi_{\mathbb{Z}^{d+1}}^{l},\label{eq:IntOfMuP}\\
 & \mu_{l}^{-}\le\mu_{l}\le\mu_{l}^{-}+\mu_{l}^{+}\mbox{ almost surely}.\label{eq:MuLinclusion-1}\end{eqnarray}
\end{lem}
\begin{proof}
By \prettyref{eq:XiBound} it follows (if we define $c_{5}=c_{5}(\varepsilon)$
appropriately) that for $N\ge c(\varepsilon)$\begin{equation}
u_{-}\xi_{\mathbb{Z}^{d+1}}^{l}\le u\xi_{E}^{l}\le u_{+}\xi_{\mathbb{Z}^{d+1}}^{l},\mbox{ for }1\le l\le r.\label{eq:Einequalitywithuplusminus}\end{equation}
Since $\mu_{l}$ has intensity $u\xi_{E}^{l}$ and $u_{-}\xi_{\mathbb{Z}^{d+1}}^{l}\le u\xi_{E}^{l}$
we can thin $\mu_{l}$ (by defining the appropriate random variables
on $(\Omega',\mathcal{A}',Q')$) to get $\mu_{l}^{-}$ such that $\mu_{l}^{-},\mu_{l}-\mu_{l}^{-}$
are independent, $\sigma(\mu)\times\mathcal{A}'$-measurable, $\mu_{l}^{-}\le\mu_{l}$
and \prettyref{eq:IntOfMuP} holds. The $\mu_{l}-\mu_{l}^{-}$ then
have intensity $u\xi_{E}^{l}-u_{-}\xi_{\mathbb{Z}^{d+1}}^{l}\overset{\prettyref{eq:Einequalitywithuplusminus}}{\le}(u_{+}-u_{-})\xi_{\mathbb{Z}^{d+1}}^{l}$
so we can (by extending the space $(\Omega',\mathcal{A}',Q')$) thicken
them to get $\mu_{l}^{+}$ such that \prettyref{eq:MuLMinusMeasurability},
\prettyref{eq:IntOfMuP} and \prettyref{eq:MuLinclusion-1} hold.
\end{proof}
We now continue the proof of \prettyref{pro:PPoEandRIcoupling} by
further extending $(\Omega^{'},\mathcal{A}^{'},Q^{'})$ with two independent
Poisson point process $\bar{\mu}^{+}$ and $\bar{\mu}^{-}$ on $W$
of respective intensities $u_{-}1_{\{\tilde{D}_{r+1}<\infty\}}P_{e_{\mathsf{A}}}^{\mathbb{Z}^{d+1}}$
and $(u_{+}-u_{-})1_{\{\tilde{D}_{r+1}<\infty\}}P_{e_{\mathsf{A}}}^{\mathbb{Z}^{d+1}}$
respectively and define on $(\Omega\times\Omega^{'},\mathcal{A}\otimes\mathcal{A}^{'},Q\otimes Q^{'})$\begin{equation}
\mathcal{I}_{1}=(\cup_{l=1}^{r}\mathcal{I}(\mu_{l}^{+})\cup\mathcal{I}(\bar{\mu}^{+}))\cap\mathsf{A},\mathcal{I}_{2}=(\cup_{l=1}^{r}\mathcal{I}(\mu_{l}^{-})\cup\mathcal{I}(\bar{\mu}^{-}))\cap\mathsf{A}.\label{eq:DefofI1I2}\end{equation}
Then \prettyref{eq:CorrectMeasurability} holds by \prettyref{eq:MuLMinusMeasurability}
and the definitions of $\bar{\mu}^{+}$ and $\bar{\mu}^{-}$. We check
\prettyref{eq:LawofI1andI2} in the following lemma:
\begin{lem}
\prettyref{eq:LawofI1andI2} holds for $\mathcal{I}_{1},\mathcal{I}_{2}$
as in \prettyref{eq:DefofI1I2}. \end{lem}
\begin{proof}
Recall the definition of $\mu_{\mathsf{A},u_{-}}$ from \prettyref{eq:lawofIuCAPK}.
Similarly to \prettyref{eq:MuPrimeMuEquality} we have: \begin{equation}
\mathcal{I}(\mu_{\mathsf{A},u_{-}})\cap\mathsf{A}=\left(\cup_{l=1}^{r}\mathcal{I}(\phi_{l}(1_{\{\tilde{D}_{l}<\infty=\tilde{R}_{l+1}\}}\mu_{\mathsf{A},u_{-}}))\cup\mathcal{I}(1_{\{\tilde{D}_{r+1}<\infty\}}\mu_{\mathsf{A},u_{-}})\right)\cap\mathsf{A}.\label{eq:MuAuDecomposition}\end{equation}
Also similarly to \prettyref{eq:MuLIndependence} and \prettyref{eq:MuLIntensity}
the $\phi_{l}(1_{\{\tilde{D}_{l}<\infty=\tilde{R}_{l+1}\}}\mu_{\mathsf{A},u_{-}}),1\le l\le r$,$1_{\{\tilde{D}_{r+1}<\infty\}}\mu_{\mathsf{A},u_{-}}$,
are independent Poisson point processes and have intensities $u_{-}\xi_{\mathbb{Z}^{d+1}}^{l},1\le l\le r$,
and $u_{-}1_{\{\tilde{D}_{r+1}<\infty\}}P_{e_{\mathsf{A}}}^{\mathbb{Z}^{d+1}}$
respectively. These coincide with the intensities of $\mu_{l}^{-},1\le l\le r,\bar{\mu}_{l}^{-}$,
so from \prettyref{eq:DefofI1I2} and \prettyref{eq:MuAuDecomposition}
we see that $\mathcal{I}_{1}\overset{\mbox{law}}{=}\mathcal{I}(\mu_{\mathsf{A},u_{-}})\cap\mathsf{A}\overset{\mbox{law},\prettyref{eq:lawofIuCAPK}}{=}\mathcal{I}^{u_{-}}\cap\mathsf{A}$.
Similarly $\mathcal{I}_{2}\overset{\mbox{law}}{=}\mathcal{I}^{u_{+}-u_{-}}\cap\mathsf{A}$
so the proof of the lemma is complete.
\end{proof}
Continuing with the proof of \prettyref{pro:PPoEandRIcoupling} we
see that we are done once we have shown \prettyref{eq:InclusionBound}.
We have (noting that by \prettyref{eq:MuLIntensity} the process $\sum_{l>r}\mu_{l}$
has intensity $u1_{\{\tilde{D}_{r+1}^{'}<\infty\}}\kappa_{e_{\mathsf{A}',\tilde{B}}}$)
\begin{eqnarray}
Q\otimes Q^{'}(\mathcal{I}_{1}\subset(\mathcal{I}(\mu)-x)\cap\mathsf{A}\subset\mathcal{I}_{1}\cup\mathcal{I}_{2}) & \overset{\prettyref{eq:MuMuLTrace},\prettyref{eq:MuLinclusion-1},\prettyref{eq:DefofI1I2},}{\ge}\nonumber \\
Q^{'}(\bar{\mu}^{+}=0\mbox{ and }\bar{\mu}^{-}=0)Q(\sum_{l>r}\mu_{l}=0) & =\nonumber \\
\exp(-\left(u_{+}P_{e_{\mathsf{A}}}^{\mathbb{Z}^{d+1}}(\tilde{D}_{r+1}<\infty)+uP_{e_{\mathsf{A}',\tilde{B}}}(\tilde{D}_{r+1}^{'}<\infty)\right)) & \ge\nonumber \\
1-cu_{+}\left\{ P_{e_{\mathsf{A}}}^{\mathbb{Z}^{d+1}}(\tilde{D}_{r+1}<\infty)+P_{e_{\mathsf{A}',\tilde{B}}}(\tilde{D}_{r+1}^{'}<T_{\tilde{B}})\right\}  & \ge\nonumber \\
1-cu_{+}\left\{ \left(\sup_{z\in\partial_{e}\mathsf{B}}P_{z}^{\mathbb{Z}^{d+1}}(H_{\mathsf{A}}<\infty)\right)^{r}+\left(\sup_{z\in\partial_{e}\mathsf{B}'}P_{z}(H_{\mathsf{A}'}<T_{\tilde{B}})\right)^{r}\right\} \label{eq:Apa}\end{eqnarray}
To bound the last line of the above formula we will need:
\begin{lem}
\label{lem:EscapeIsLikely}($N\ge c(\varepsilon),x\in\mathbb{T}_{N}\times[-\frac{N}{2},\frac{N}{2}]$)\begin{eqnarray}
\sup_{z\in\partial_{e}\mathsf{B}}P_{z}^{\mathbb{Z}^{d+1}}(H_{\mathsf{A}}<\infty) & \le & N^{-c_{6}(\varepsilon)},\label{eq:EscapeIsLikelyZ}\\
\sup_{z\in\partial_{e}\mathsf{B}'}P_{z}(H_{\mathsf{A}'}<T_{\tilde{B}}) & \le & N^{-c_{6}(\varepsilon)}.\label{eq:EscapeIsLikelyE}\end{eqnarray}
\end{lem}
\begin{proof}
To prove \prettyref{eq:EscapeIsLikelyE} note that by the strong Markov
property $\sup_{z\in\partial_{e}\mathsf{B}'}P_{z}(H_{\mathsf{A}'}<T_{\tilde{B}})\le\sup_{z\in\partial_{e}\mathsf{B}'}P_{z}(H_{\mathsf{A}'}<T_{\mathsf{C}'})+\sup_{z\in\partial_{e}\mathsf{C}'}P_{z}(H_{\partial_{e}\mathsf{B}'}<T_{\tilde{B}})\times\sup_{z\in\partial_{e}\mathsf{B}'}P_{z}(H_{\mathsf{A}'}<T_{\tilde{B}})$
which implies\begin{eqnarray}
\sup_{z\in\partial_{e}\mathsf{B}'}P_{z}(H_{\mathsf{A}'}<T_{\tilde{B}})\le\frac{\sup_{z\in\partial_{e}\mathsf{B}'}P_{z}(H_{\mathsf{A}'}<T_{\mathsf{C}'})}{\inf_{z\in\partial_{e}\mathsf{C}'}P_{z}(T_{\tilde{B}}<H_{\partial_{e}\mathsf{B}'})}.\label{eq:MCbound}\end{eqnarray}
By the invariance principle $\inf_{z\in\partial_{e}\mathsf{C}'}P_{z}(T_{\mathbb{T}_{N}\times(-N,N)}<H_{\mathsf{B}'})\ge c$
for $N\ge c$ and by a one dimensional random walk estimate we see
$\inf_{\mathbb{T}_{N}\times\{-N,N\}}P_{z}(T_{\tilde{B}}<H_{\mathsf{B}'})\ge c\frac{1}{(\log N)^{2}}$,
so $\inf_{z\in\partial_{e}\mathsf{C}'}P_{z}(T_{\tilde{B}}<H_{\partial_{e}\mathsf{B}'})\ge c\frac{1}{(\log N)^{2}}$.
We also have \[
\sup_{z\in\partial_{e}\mathsf{B}'}P_{z}(H_{\mathsf{A}'}<T_{\mathsf{C}'})\le\sup_{z\in\partial_{e}\mathsf{B}}P_{z}^{\mathbb{Z}^{d+1}}(H_{\mathsf{A}}<\infty).\]
But Proposition 1.5.10, p. 36 of \cite{LawlersLillaGrona} implies
that $\sup_{z\in\partial_{e}\mathsf{B}}P_{z}^{\mathbb{Z}^{d+1}}(H_{\mathsf{A}}<\infty)\le N^{-2c_{6}(\varepsilon)}$,
thus proving \prettyref{eq:EscapeIsLikelyZ} and also, via \prettyref{eq:MCbound},
completing the proof of \prettyref{eq:EscapeIsLikelyE}.\end{proof}
\begin{rem}
Let us record for later that a similar argument (introducing a set
$\mathsf{C}''=B(x,\frac{N}{3})\supset\mathsf{C}'$ which plays the
role of $\mathsf{C}'$) also proves\begin{equation}
\sup_{z\in\partial_{e}\mathsf{C}}P_{z}^{\mathbb{Z}^{d+1}}(H_{\partial_{e}\mathsf{B}}<\infty)\le N^{-c(\varepsilon)}\mbox{ and }\sup_{z\in\partial_{e}\mathsf{C}'}P_{z}(H_{\partial_{e}\mathsf{B}'}<T_{\tilde{B}})\le N^{-c(\varepsilon)}.\label{eq:EscapeIsLikelyCB}\end{equation}
\qed
\end{rem}
We now continue with the proof of \prettyref{pro:PPoEandRIcoupling}.
By \prettyref{eq:EscapeIsLikelyZ} and \prettyref{eq:EscapeIsLikelyE}
we see that\begin{eqnarray*}
\left(\sup_{z\in\partial_{e}\mathsf{B}}P_{z}^{\mathbb{Z}^{d+1}}(H_{\mathsf{A}}<\infty)\right)^{r}+\left(\sup_{z\in\partial_{e}\mathsf{B}'}P_{z}(H_{\mathsf{A}'}<\infty)\right)^{r} & \le & c(cN^{-c_{6}})^{r}\\
 & \overset{\prettyref{eq:DefOfR}}{\le} & cN^{-10(d+1)}.\end{eqnarray*}
 Thus \prettyref{eq:InclusionBound} follows from \prettyref{eq:Apa}.
This completes the proof of \prettyref{pro:PPoEandRIcoupling}.
\end{proof}
It remains to prove \prettyref{pro:XiBound}. 
\begin{proof}[Proof of \prettyref{pro:XiBound}]
For $w\in\mathcal{T}_{\mathsf{B}}$ (see under \prettyref{eq:TraceDefinition}
for the notation) let $w^{s}$ denote the vertex at which $w$ starts
and let $w^{e}$ denote the vertex at which it ends (i.e. stays constant).
Note that for all $\mathbf{w}=(w_{1},...,w_{l})\in(\mathcal{T}_{\mathsf{B}})^{\times l}$
\begin{eqnarray}
\xi_{\mathbb{Z}^{d+1}}^{l}(\mathbf{w}) & \overset{\prettyref{eq:PsiZDef},\prettyref{eq:DefOfPhiPhiPrime}}{=} & P_{e_{\mathsf{A}}}^{\mathbb{Z}^{d+1}}(\tilde{D}_{l}<\infty=\tilde{R}_{l+1},X_{(\tilde{R}_{k}+\cdot)\wedge\tilde{D}_{k}}=w_{k},1\le k\le l\}\nonumber \\
 & = & e_{\mathsf{A}}(w_{1}^{s})\left(\prod_{i=1}^{l}r(w_{l})\right)\left(\prod_{i=1}^{l-1}s_{\mathbb{Z}^{d+1}}(w_{e}^{i},w_{s}^{i+1})\right)t_{\mathbb{Z}^{d+1}}(w_{l}^{e}),\label{eq:EZ}\end{eqnarray}
where the last equality follows by several applications of the strong
Markov property and where we define \begin{eqnarray}
r(w) & = & P_{w^{s}}^{\mathbb{Z}^{d+1}}(X_{\cdot\wedge T_{\mathsf{B}}}=w)=P_{w^{s}+x}(X_{\cdot\wedge T_{\mathsf{B}'}}=w+x)\mbox{ for }w\in\mathcal{T}_{\mathsf{B}},\nonumber \\
s_{\mathbb{Z}^{d+1}}(z,y) & = & P_{z}^{\mathbb{Z}^{d+1}}(H_{\mathsf{A}}<\infty,X_{H_{\mathsf{A}}}=y)\mbox{ for }z\in\partial_{e}\mathsf{B},y\in\partial_{i}\mathsf{A}\mbox{ and }\label{eq:DefOfSZ}\\
t_{\mathbb{Z}^{d+1}}(z) & = & P_{z}^{\mathbb{Z}^{d+1}}(H_{\mathsf{A}}=\infty)\mbox{ for }z\in\partial_{e}\mathsf{B}.\label{eq:DefOfTZ}\end{eqnarray}
Similarly for all $\mathbf{w}=(w_{1},...,w_{l})\in(\mathcal{T}_{\mathsf{B}})^{\times l}$
\begin{eqnarray}
\xi_{E}^{l}(\mathbf{w}) & \overset{\prettyref{eq:PsiEDef},\prettyref{eq:DefOfPhiPhiPrime}}{=} & P_{e_{\mathsf{A}',\tilde{B}}}(\tilde{D}_{l}^{'}<T_{\tilde{B}}<\tilde{R}_{l+1}^{'},X_{(\tilde{R}_{k}^{'}+\cdot)\wedge\tilde{D}_{k}^{'}}-x=w_{k},1\le k\le l\}\nonumber \\
 & = & e_{\mathsf{A}',\tilde{B}}(w_{1}^{s}+x)\left(\prod_{i=1}^{l}r(w_{l})\right)\left(\prod_{i=1}^{l}s_{E}(w_{e}^{i},w_{s}^{i+1})\right)t_{E}(w_{l}^{e}),\label{eq:EE}\end{eqnarray}
where\begin{eqnarray}
s_{E}(z,y) & = & P_{z+x}(H_{\mathsf{A}'}<T_{\tilde{B}},X_{H_{A'}}=y+x)\mbox{ for }z\in\partial_{e}\mathsf{B},y\in\partial_{i}\mathsf{A},\label{eq:DefOfSE}\\
t_{E}(z) & = & P_{z+x}(H_{\mathsf{A}'}>T_{\tilde{B}})\mbox{ for }z\in\partial_{e}\mathsf{B}.\label{eq:DefOfTE}\end{eqnarray}
We will make a factor by factor comparison of the right-hand sides
of \prettyref{eq:EZ} and \prettyref{eq:EE} to obtain \prettyref{eq:XiBound}.
For this we will need the following lemmas:
\begin{lem}
\label{lem:EquilBound}($N\ge c(\varepsilon),x\in\mathbb{T}_{N}\times[-\frac{N}{2},\frac{N}{2}]$)
For all $z\in\partial_{i}\mathsf{A}$\begin{equation}
e_{\mathsf{A}}(z)(1-cN^{-c(\varepsilon)})\le e_{\mathsf{A}^{'},\tilde{B}}(z+x)\le e_{\mathsf{A}}(z)(1+cN^{-c(\varepsilon)}).\label{eq:EquilBound}\end{equation}

\end{lem}

\begin{lem}
\label{lem:tinequality}($N\ge c(\varepsilon),x\in\mathbb{T}_{N}\times[-\frac{N}{2},\frac{N}{2}]$)
For all $z\in\partial_{e}\mathsf{B}$\textup{\begin{equation}
(1-cN^{-c(\varepsilon)})t_{\mathbb{Z}^{d+1}}(z)\le t_{E}(z)\le(1+cN^{-c(\varepsilon)})t_{\mathbb{Z}^{d+1}}(z).\label{eq:tinequality}\end{equation}
}
\end{lem}

\begin{lem}
\label{lem:sinequality}($N\ge c(\varepsilon),x\in\mathbb{T}_{N}\times[-\frac{N}{2},\frac{N}{2}]$)For
all $z\in\partial_{e}\mathsf{B}$ and $y\in\partial_{i}\mathsf{A}$\begin{equation}
(1-cN^{-c(\varepsilon)})s_{\mathbb{Z}^{d+1}}(z,y)\le s_{E}(z,y)\le(1+cN^{-c(\varepsilon)})s_{\mathbb{Z}^{d+1}}(z,y).\label{eq:sinequality}\end{equation}

\end{lem}
Before proving these lemmas we note that by comparing \prettyref{eq:EZ}
and \prettyref{eq:EE} and applying \prettyref{eq:EquilBound}, \prettyref{eq:tinequality}
and \prettyref{eq:sinequality} we get \prettyref{eq:XiBound}. The
proof of \prettyref{pro:XiBound} is thus done once we have proved
Lemmas \ref{lem:EquilBound}, \ref{lem:tinequality} and \ref{lem:sinequality}.
We start with \prettyref{lem:EquilBound}:
\begin{proof}[Proof of \prettyref{lem:EquilBound}]
The upper bound follows by the argument in the proof of Lemma 4.4
of \cite{Sznitman2009-UBonDTofDCandRI} (that lemma proves the upper
bound with $B(0,2[\frac{N^{1-\varepsilon}}{8}])$ in the place of
$\mathsf{A}'$, but the special form of the radius and that the centre
is at $0$ plays essentially no role in the argument). The lower bound
follows by the argument leading up to (6.4) of \cite{Sznitman2009-OnDOMofRWonDCbyRI}
(that formula is the upper bound in the case $x=0$ , but similarly
the fact that $x=0$ plays no essential role in the argument). 
\end{proof}
We now continue with the proof of \prettyref{pro:XiBound} by proving
\prettyref{lem:tinequality}.
\begin{proof}[Proof of \prettyref{lem:tinequality}]
We will compare $t_{E}(z)$ and $t_{\mathbb{Z}^{d+1}}(z)$ with\begin{equation}
t_{\mathsf{C}}(z)=P_{z}^{\mathbb{Z}^{d+1}}(H_{\mathsf{A}}>T_{\mathsf{C}})\overset{\prettyref{eq:DefIfAPrime},\prettyref{eq:DefOfBCBPrimeCPrime}}{=}P_{z+x}(H_{\mathsf{A}'}>T_{\mathsf{C}'})\mbox{ for }z\in\partial_{e}\mathsf{B}.\label{eq:DefOfTC}\end{equation}
It is obvious from \prettyref{eq:DefOfTZ} that $t_{\mathbb{Z}^{d+1}}(z)\le t_{\mathsf{C}}(z)$,
so to show the first inequality of \prettyref{eq:tinequality} it
suffices to show $(1-cN^{-c(\varepsilon)})t_{\mathsf{C}}(z)\le t_{E}(z)$.
But this follows by the following upper bound on $t_{\mathsf{C}}(z)$:
\begin{eqnarray*}
t_{\mathsf{C}}(z) & \overset{\prettyref{eq:DefOfTC}}{=} & P_{z+x}(H_{\mathsf{A}'}>T_{\tilde{B}})+P_{z}(T_{\tilde{B}}>H_{\mathsf{A}'}>T_{\mathsf{C}'})\\
 & \overset{\prettyref{eq:DefOfTE},\prettyref{eq:DefOfTC}}{\le} & t_{E}(a)+t_{\mathsf{C}}(z)\sup_{z'\in\partial_{e}\mathsf{C}'}P_{z'}(H_{\mathsf{A}'}<T_{\tilde{B}})\\
 & \overset{\prettyref{eq:EscapeIsLikelyE},\mathsf{B}'\subset\mathsf{C}'}{\le} & t_{E}(z)+t_{\mathsf{C}}(z)N^{-c(\varepsilon)}.\end{eqnarray*}
To show the second inequality of \prettyref{eq:tinequality} note
that from \prettyref{eq:DefOfTE}, \prettyref{eq:DefOfTC} and $\mathsf{C}\subset\tilde{B}$
it is obvious that $t_{E}(z)\le t_{\mathsf{C}}(z)$, so it suffices
to show $t_{E}(z)\le(1+cN^{-c(\varepsilon)})t_{\mathbb{Z}^{d+1}}(z)$.
But this follows from by the following upper bound on $t_{\mathsf{C}}(z)$:\begin{eqnarray*}
t_{\mathsf{C}}(z) & \overset{\prettyref{eq:DefOfTC}}{=} & P_{z}^{\mathbb{Z}^{d+1}}(H_{\mathsf{A}}=\infty)+P_{z}^{\mathbb{Z}^{d+1}}(\infty>H_{\mathsf{A}}>T_{\mathsf{C}})\\
 & \overset{\prettyref{eq:DefOfTZ},\prettyref{eq:DefOfTC}}{\le} & t_{\mathbb{Z}^{d+1}}(z)+t_{\mathsf{C}}(z)\sup_{z'\in\partial_{e}\mathsf{C}}P_{z'}^{\mathbb{Z}^{d+1}}(H_{\mathsf{A}}<\infty)\\
 & \overset{\prettyref{eq:EscapeIsLikelyZ},\mathsf{B}\subset\mathsf{C}}{\le} & t_{\mathbb{Z}^{d+1}}(z)+t_{\mathsf{C}}(z)cN^{-c(\varepsilon)}.\end{eqnarray*}
This completes the proof of \prettyref{lem:tinequality}.
\end{proof}
Finally we prove \prettyref{lem:sinequality}; the argument is close
to that of Lemma 5.3 in \cite{Sznitman2009-OnDOMofRWonDCbyRI} and
Lemma 3.2 of \cite{Sznitman2009-UBonDTofDCandRI}.
\begin{proof}[Proof of \prettyref{lem:sinequality}]
For $z\in\partial_{e}\mathsf{B}$, $y\in\partial_{i}\mathsf{A}$
we will compare $s_{\mathbb{Z}^{d+1}}(z,y)$ and $s_{E}(z,y)$ with\begin{equation}
s_{\mathsf{C}}(z,y)=P_{z}^{\mathbb{Z}^{d+1}}(H_{\mathsf{A}}<T_{\mathsf{C}},X_{H_{\mathsf{A}}}=y)\overset{\prettyref{eq:DefIfAPrime},\prettyref{eq:DefOfBCBPrimeCPrime}}{=}P_{z+x}(H_{\mathsf{A}'}<T_{\mathsf{C}'},X_{H_{\mathsf{A}'}}=y+x).\label{eq:DefOfSC}\end{equation}
Recalling \prettyref{eq:DefOfSZ} and using the decomposition $W=\{T_{\mathsf{C}}<H_{\mathsf{A}}\}\cup\{H_{\mathsf{A}}<T_{\mathsf{C}}\}$,
and similarly recalling \prettyref{eq:DefOfSE} and using the decomposition
$\mathcal{T}_{\tilde{B}}=\{T_{\mathsf{C}'}<H_{\mathsf{A}'}\}\cup\{H_{\mathsf{A}'}<T_{\mathsf{C}'}\}$,
we see that\begin{equation}
\begin{array}{ccrcl}
s_{\mathsf{C}}(z,y) & \le & s_{\mathbb{Z}^{d+1}}(z,y) & = & s_{\mathsf{C}}(z,y)+P_{z}^{\mathbb{Z}^{d+1}}(T_{\mathsf{C}}<H_{\mathsf{A}}<\infty,X_{H_{\mathsf{A}}}=y)\mbox{ and }\\
s_{\mathsf{C}}(z,y) & \le & s_{E}(z,y) & = & s_{\mathsf{C}}(z,y)+P_{z+x}(T_{\mathsf{C}'}<H_{\mathsf{A}'}<T_{\tilde{B}},X_{H_{\mathsf{A}'}}=y).\end{array}\label{eq:SComparison}\end{equation}
To prove \prettyref{eq:sinequality} it suffices to show\begin{eqnarray}
P_{z}^{\mathbb{Z}^{d+1}}(T_{\mathsf{C}}<H_{\mathsf{A}}<\infty,X_{H_{\mathsf{A}}}=y) & \le & cN^{-c(\varepsilon)}s_{\mathbb{Z}^{d+1}}(z,y)\mbox{ and }\label{eq:SuffToShowSZ}\\
P_{z+x}(T_{\mathsf{C}'}<H_{\mathsf{A}'}<T_{\tilde{B}},X_{H_{\mathsf{A}'}}=y) & \le & cN^{-c(\varepsilon)}s_{E}(z,y),\label{eq:SuffToShowSE}\end{eqnarray}
since then $s_{\mathbb{Z}^{d+1}}(z,y)(1-cN^{-c(\varepsilon)})\le s_{E}(z,y)$
by using the upper right-hand side of \prettyref{eq:SComparison},
\prettyref{eq:SuffToShowSZ} and then the lower left-hand side of
\prettyref{eq:SComparison}, and similarly $s_{E}(z,y)(1-cN^{-c(\varepsilon)})\le s_{\mathbb{Z}^{d+1}}(z,y)$.

We start with \prettyref{eq:SuffToShowSZ}. We have \begin{eqnarray}
\begin{array}{rl}
\sup_{z\in\partial_{e}\mathsf{B}} & P_{z}^{\mathbb{Z}^{d+1}}(T_{\mathsf{C}}<H_{\mathsf{A}}<\infty,X_{H_{\mathsf{A}}}=y)\\
\le & \sup_{z'\in\partial_{e}\mathsf{C}}P_{z'}^{\mathbb{Z}^{d+1}}(H_{\partial_{e}\mathsf{B}}<\infty)\sup_{z''\in\partial_{e}\mathsf{B}}P_{z''}^{\mathbb{Z}^{d+1}}(H_{\mathsf{A}}<\infty,X_{H_{\mathsf{A}}}=y)\\
\overset{\prettyref{eq:EscapeIsLikelyCB},\prettyref{eq:DefOfSZ}}{\le} & cN^{-c(\varepsilon)}\sup_{z''\in\partial_{e}\mathsf{B}}s_{\mathbb{Z}^{d+1}}(z'',y).\end{array}\label{eq:asdasd}\end{eqnarray}
Note that the map $z\rightarrow P_{z}^{\mathbb{Z}^{d+1}}(H_{\mathsf{A}}<\infty,X_{H_{\mathsf{A}}}=y)$
is positive harmonic on $\mathbb{Z}^{d+1}\backslash\mathsf{A}$ so
that by Harnack's inequality (Theorem 1.7.2 p. 42 of \cite{LawlersLillaGrona})
and a standard covering argument we get $\sup_{z''\in\partial_{e}\mathsf{B}}s_{\mathbb{Z}^{d+1}}(z'',y)\le c\inf_{z''\in\partial_{e}\mathsf{B}}s_{\mathbb{Z}^{d+1}}(z'',y)$.
Combining this with \prettyref{eq:asdasd} we get \prettyref{eq:SuffToShowSZ}.

It remains to show \prettyref{eq:SuffToShowSE}. Similarly to \prettyref{eq:asdasd}
we have: \begin{equation}
\sup_{z\in\partial_{e}\mathsf{B}}P_{z+x}(T_{\mathsf{C}'}<H_{\mathsf{A}'}<T_{\tilde{B}},X_{H_{\mathsf{A}'}}=y)\overset{\prettyref{eq:EscapeIsLikelyCB},\prettyref{eq:DefOfSE}}{\le}cN^{-c(\varepsilon)}\sup_{z\in\partial_{e}\mathsf{B}}s_{E}(z,y).\label{eq:asdasd2}\end{equation}
Now the map $z\rightarrow P_{z}(H_{\mathsf{A}'}<T_{\tilde{B}},X_{H_{\mathsf{A}'}}=y)$
is positive harmonic on $\tilde{B}\backslash\mathsf{A}'\supset\mathsf{C}'\backslash\mathsf{A}'$.
Since $\mathsf{C}'\backslash\mathsf{A}'$ can be identified as a subset
of $\mathbb{Z}^{d+1}$ we have similarly to above by Harnack's inequality
that $\sup_{z\in\partial_{e}\mathsf{B}}s_{E}(z,y)\le c\inf_{z\in\partial_{e}\mathsf{B}}s_{E}(z,y)$.
Combining this with \prettyref{eq:asdasd2} we get \prettyref{eq:SuffToShowSE}.
This completes the proof of \prettyref{lem:sinequality}.
\end{proof}
This also completes the proof of \prettyref{pro:XiBound}. 
\end{proof}
All the steps in the proof of \prettyref{thm:GumbelForLocTime} and
its corollaries have now been completed. We conclude with an open
question and a comment on the use of \prettyref{thm:CouplingManyBoxes}
as a {}``transfer mechanism''.
\begin{rem}
\label{rem:EndRemark}(1) The Gumbel distribution has been proven
to arise as a distributional limit for rescaled cover times of certain
finite graphs (see \cite{MatthewsCoveringProbsForMCs,DevroyeSbihiRWonHighlySymmGraphs}).
One important graph in the study of cover times for which a Gumbel
distributional limit has been conjectured (see Chapter 7, Section
2.2, p. 23 of \cite{aldous-fill:book}), but not proved, is the discrete
torus $\mathbb{T}_{N}=(\mathbb{Z}/N\mathbb{Z})^{d},d\ge3$. It is
an open question whether the methods of the proof of \prettyref{thm:GumbelForLocTime}
could be used to prove that conjecture. A strategy could be to reduce
it to the statement \prettyref{eq:QuantCovLevResult-1} with the help
of a coupling with random interlacements. For bounded $u$ and fixed
$\delta$ a coupling of random interlacements and the trace of random
walk in the torus (in \emph{one} local box) has been produced in \cite{TeixeiraWindischOnTheFrag}.

(2) A coupling of random walk with random interlacements can be used
as a {}``transfer mechanism'' to reduce the proofs of properties
of random walk in the cylinder to proofs of properties purely in term
of random interlacements (as we reduced \prettyref{thm:GumbelForLocTime}
to \prettyref{eq:QuantCovLevResult-1}). Sometimes such transfers
require the use of both inclusions (cf. \prettyref{eq:CouplingManyBoxesEvent})
simultaneously and therefore need a coupling of random walk with \emph{joint}
random interlacements, such as \prettyref{thm:CouplingManyBoxes}.
An example arises when using the random interlacement concept of strong
supercriticality of levels $u>0$ (see Definition 2.4 of \cite{TeixeiraWindischOnTheFrag})
to {}``patch up'' components of the vacant set $(X(0,n))^{c},n\ge1$,
in various local boxes where the walk is coupled with random interlacements
(as was done in the case of the torus in Proposition 2.7, see also
Lemma 2.6, of \cite{TeixeiraWindischOnTheFrag}). For instance if
one could prove that all $u<u^{\star}$ are strongly supercritical
(where $u^{\star}$ is the critical parameter of interlacement percolation,
see (0.13) of \cite{Sznitman2007} and Remark 2.5 (2) of \cite{TeixeiraWindischOnTheFrag})
then \prettyref{thm:CouplingManyBoxes} would be the kind of coupling
that could be used to derive from this, using the aforementioned {}``patching'',
the {}``correct'' lower bound on the disconnection time $T_{N}$
of the cylinder, (and thus improve on Theorem 7.3 of \cite{Sznitman2009-OnDOMofRWonDCbyRI},
see also Remark 7.5 (2) of \cite{Sznitman2009-OnDOMofRWonDCbyRI}).

\qed
\end{rem}


\begin{thebibliography}{10}

\bibitem{Aldous-OnTimeTaken...} David~J. Aldous. \newblock {On the time taken by random walks on finite groups to visit every   state.} \newblock {\em Z. Wahrscheinlichkeitstheor. Verw. Geb.}, 62:361--374, 1983.

\bibitem{Aldous-ThresholdLimitsforCT} David~J. Aldous. \newblock Threshold limits for cover times. \newblock {\em J. Theoret. Probab.}, 4(1):197--211, 1991.

\bibitem{aldous-fill:book} David~J. Aldous and James~A. Fill. \newblock {\em {Reversible {M}arkov Chains and Random Walks on Graphs}}. \newblock Book in preparation, available at \url{http://www.stat.berkeley.edu/~aldous/RWG/book.html}.

\bibitem{Belius2010} David Belius. \newblock{Cover levels and random interlacements.} \newblock{Accepted for publication in Ann. Appl. Probab., also available at \url{http://www.math.ethz.ch/~dbelius/artiklar/clri.pdf}}.

\bibitem{Brummelhuis1991} M.~J. A.~M. Brummelhuis and H.~J. Hilhorst. \newblock Covering of a finite lattice by a random walk. \newblock {\em Physica A: Statistical and Theoretical Physics},   176(3):387--408, 1991.

\bibitem{StrongInvForLocTime} E.~Cs{\'a}ki and P.~R{\'e}v{\'e}sz. \newblock Strong invariance for local times. \newblock {\em Z. Wahrsch. Verw. Gebiete}, 62(2):263--278, 1983.

\bibitem{DaleyVereJones_IntroToTheTheoryOfPointProcesses} D.~J. Daley and D.~Vere-Jones. \newblock {\em An introduction to the theory of point processes. {V}ol. {II}}. \newblock Probability and its Applications (New York). Springer, New York,   second edition, 2008. \newblock General theory and structure.

\bibitem{DemboPeresEtAl-CoverTimesforBMandRWin2D} Amir Dembo, Yuval Peres, Jay Rosen, and Ofer Zeitouni. \newblock Cover times for {B}rownian motion and random walks in two dimensions. \newblock {\em Ann. of Math. (2)}, 160(2):433--464, 2004.

\bibitem{Sznitman2006} Amir Dembo and Alain-Sol Sznitman. \newblock On the disconnection of a discrete cylinder by a random walk. \newblock {\em Probab. Theory Related Fields}, 136(2):321--340, 2006.

\bibitem{DemboZeitouniLargeDeviations} Amir Dembo and Ofer Zeitouni. \newblock {\em Large deviations techniques and applications}, volume~38 of {\em   Applications of Mathematics (New York)}. \newblock Springer-Verlag, New York, second edition, 1998.

\bibitem{DevroyeSbihiRWonHighlySymmGraphs} Luc Devroye and Amine Sbihi. \newblock Random walks on highly symmetric graphs. \newblock {\em J. Theoret. Probab.}, 3(4):497--514, 1990.

\bibitem{LawlersLillaGrona} Gregory~F. Lawler. \newblock {\em Intersections of random walks}. \newblock Probability and its Applications. Birkh{\"a}user Boston Inc., Boston,   MA, 1991.

\bibitem{MatthewsCoveringProbsForMCs} Peter Matthews. \newblock Covering problems for {M}arkov chains. \newblock {\em Ann. Probab.}, 16(3):1215--1228, 1988.

\bibitem{Resnick_ExtremeValRegVarAndPP} Sidney~I. Resnick. \newblock {\em Extreme values, regular variation, and point processes},   volume~4 of {\em Applied Probability. A Series of the Applied Probability   Trust}. \newblock Springer-Verlag, New York, 1987.

\bibitem{RevuzYor-ContMartAndBM} Daniel Revuz and Marc Yor. \newblock {\em Continuous martingales and {B}rownian motion}, volume 293. \newblock Springer-Verlag, Berlin, third edition, 1999.

\bibitem{SidoraviciusSznitman2009} Vladas Sidoravicius and Alain-Sol Sznitman. \newblock {Percolation for the vacant set of random interlacements.} \newblock {\em Commun. Pure Appl. Math.}, 62(6):831--858, 2009.

\bibitem{SznitmanNewExofRWRE} Alain-Sol Sznitman. \newblock On new examples of ballistic random walks in random environment. \newblock {\em Ann. Probab.}, 31(1):285--322, 2003.

\bibitem{Sznitman-HowUniveral...} Alain-Sol Sznitman. \newblock How universal are asymptotics of disconnection times in discrete   cylinders? \newblock {\em Ann. Probab.}, 36(1):1--53, 2008.

\bibitem{Sznitman2009-RWonDTandRI} Alain-Sol Sznitman. \newblock {Random walks on discrete cylinders and random interlacements.} \newblock {\em Probab. Theory Relat. Fields}, 145(1-2):143--174, 2009.

\bibitem{Sznitman2009-UBonDTofDCandRI} Alain-Sol Sznitman. \newblock {Upper bound on the disconnection time of discrete cylinders and   random interlacements.} \newblock {\em Ann. Probab.}, 37(5):1715--1746, 2009.

\bibitem{Sznitman2009-OnDOMofRWonDCbyRI} Alain-Sol Sznitman. \newblock On the domination of random walk on a discrete cylinder by random   interlacements. \newblock {\em Electron. J. Probab.}, 14:no. 56, 1670--1704, 2009.

\bibitem{Sznitman2007} Alain-Sol Sznitman. \newblock Vacant set of random interlacements and percolation. \newblock {\em Ann. of Math. (2)}, 171(3):2039--2087, 2010.

\bibitem{TeixeiraWindischOnTheFrag} Augusto Teixeira and David Windisch. \newblock On the fragmentation of a torus by random walk. \newblock Available at arXiv:1007.0902.

\bibitem{Windisch2008} David Windisch. \newblock Random walk on a discrete torus and random interlacements. \newblock {\em Electron. Commun. Probab.}, 13:140--150, 2008.

\end{thebibliography}
\end{document}